\documentclass[final]{siamltex}
\usepackage{amsmath,amsfonts,amssymb}

\newtheorem{assumption}{Assumption}[section]

\newtheorem{aassumption}{}

\newtheorem{bproof}{}

\newtheorem{remark}{}

\begin{document}

\title{Convergence Rate of Stochastic Gradient Search
in the Case of Multiple and Non-Isolated Minima} 

\author{Vladislav B. Tadi\'{c}
\thanks{Department of Mathematics, University of Bristol,
University Walk, Bristol BS8 1TW, United Kingdom. 
(v.b.tadic@bristol.ac.uk). }}

\maketitle

\begin{abstract} The convergence rate of  
stochastic gradient search is analyzed in this paper. 
Using arguments based on differential geometry and Lojasiewicz inequalities, 
tight bounds on the convergence rate of 
general stochastic gradient algorithms 
are derived. 
As opposed to the existing results, 
the results presented in this paper 
allow the objective function to have multiple, non-isolated minima, 
impose no restriction on the values of the Hessian (of the objective function)
and do not require the algorithm estimates to have a single limit point. 
Applying these new results, 
the convergence rate of 
recursive prediction error identification algorithms  
is studied. 
The convergence rate of supervised and temporal-difference learning algorithms 
is also analyzed 
using the results derived in the paper. 
\end{abstract}



\section{Introduction} 

Stochastic gradient algorithms are a recursive optimization 
method of 
the stochastic approximation type. 
This method is commonly used to compute minima (or maxima) of a function 
whose values are available only through noise-corrupted observations. 
It has found a wide range of applications 
in the areas such as automatic control, system identification, 
signal processing, machine learning, operations research, statistical inference, 
economics and management (to name a few). 
For further details, see \cite{caines}, \cite{ljung2}, \cite{ljung3}, \cite{pflug}, \cite{polyak}, \cite{powell}, \cite{spall} 
and the references cited therein. 

Due to their practical importance, 
the asymptotic behavior of stochastic gradient algorithms has been 
thoroughly studied in a large number of papers and books. 
A significant attention has been given to the rate of convergence, 
as this property directly characterizes the efficiency
and  
enables a construction of reliable stopping rules 
(see \cite{benveniste}, \cite{kushner&yin}, \cite{ljung2}, 
\cite{polyak}, \cite{spall} 
and the references given therein). 
Although the existing results on the convergence rate  
provide 
a good insight into the efficiency and asymptotic behavior   
of stochastic gradient algorithms, 
they hold under very restrictive conditions. 
More specifically, 
the existing results 
require the algorithm estimates to converge 
to an isolated minimum of the objective function
at which the Hessian 
(of the objective function)
is strictly positive definite. 
Unfortunately, such conditions are practically impossible 
to verify for complex, high-dimensional and high-nonlinear 
stochastic gradient algorithms. 

In this paper, 
the rate of convergence of stochastic gradient algorithms 
is analyzed for the case when the objective function has multiple, 
non-isolated minima
(note that the Hessian can be only semi-definite at a non-isolated minimum)
and when the algorithm estimates do not necessarily converge to a single limit 
point. 
Using arguments based on differential geometry and Lojasiewicz inequalities, 
relatively tight upper bounds on the convergence rate are derived. 
The obtained results cover a broad class of complex stochastic 
gradient algorithms. 
We show how they can be used to evaluate the convergence rate of 
recursive prediction error algorithms for 
identification of 
linear stochastic dynamical systems. 
We also show how the convergence rate of supervised and temporal-difference 
learning algorithms 
can be assessed using the 
results derived in the paper. 

The paper is organized as follows. 
The main results are presented in 
Section \ref{section1}, where 
stochastic gradient algorithms with additive noise 
are considered.  
In Section \ref{section2}, the convergence rare of 
stochastic gradient algorithms with Markovian dynamics is analyzed. 
Sections \ref{section4} and \ref{section6} are devoted to   
examples of the results presented in Sections \ref{section1} and \ref{section2}. 
In Section \ref{section4}, supervised learning algorithms for 
feedforward neural networks and their convergence rate are studied, 
while 
the rate of convergence of temporal-difference learning algorithms 
is considered in Section \ref{section5}. 
The convergence rate of 
recursive prediction error algorithms for the identification of 
linear stochastic systems is analyzed in 
Section \ref{section6}. 
Sections \ref{section1*} -- \ref{section6*} contain
the proofs of the results presented in Sections \ref{section1} -- \ref{section6}. 

\section{Main Results} \label{section1}

In this section, the rate of convergence of the following algorithm is analyzed: 
\begin{align} \label{1.1}
	\theta_{n+1}
	=
	\theta_{n}
	-
	\alpha_{n} (\nabla f(\theta_{n} ) + w_{n} ), 
	\;\;\; 
	n\geq 0. 
\end{align}
In this recursion, $f:\mathbb{R}^{d_{\theta} } \rightarrow \mathbb{R}$ is a differentiable function, 
while $\{\alpha_{n} \}_{n\geq 0}$ is a sequence of positive real numbers.  
while 
$\theta_{0}$ is an $\mathbb{R}^{d_{\theta } }$-valued random variable 
defined on 
a probability space 
$(\Omega, {\cal F}, P )$, 
while $\{w_{n} \}_{n\geq 0}$ is an $\mathbb{R}^{d_{\theta} }$-valued stochastic process
defined on the same probability space. 
To allow more generality, we assume that for each $n \geq 0$, 
$w_{n}$ is a random function of $\theta_{0},\dots,\theta_{n}$. 
In the area of stochastic optimization, 
recursion (\ref{1.1}) is known as a stochastic gradient algorithm
(or stochastic gradient search),
while function $f(\cdot )$ is referred to as an objective function. 
For further details see \cite{pflug}, \cite{spall} and 
references given therein. 

Throughout the paper, unless otherwise stated, 
the following notation is used. 
The Euclidean norm is denoted by $\|\cdot \|$, 
while $d(\cdot, \cdot )$ stands for the distance induced by the Euclidean norm. 
$S$ and $C$ are the sets of stationary and critical points of $f(\cdot )$, i.e., 
\begin{align*}
	S=\{\theta \in \mathbb{R}^{d_{\theta } }:\nabla f(\theta ) = 0 \}, 
	\;\;\;\;\;  
	C=\{f(\theta ): \theta \in S \}. 
\end{align*}	
Sequence $\{\gamma_{n} \}_{n\geq 0}$
is defined by 
$\gamma_{0}=0$
and 
\begin{align*}
	\gamma_{n} = \sum_{i=0}^{n-1} \alpha_{i}
\end{align*}
for $n\geq 1$. 	
For $t \in (0, \infty )$ and $n\geq 0$, 
$a(n,t)$ is an integer defined as 
\begin{align*}
	a(n,t)=\max\left\{k\geq n: \gamma_{k} - \gamma_{n} \leq t \right\}.  
\end{align*} 

Algorithm (\ref{1.1}) is analyzed under the following assumptions:  

\begin{assumption} \label{a1.1}
$\lim_{n\rightarrow \infty } \alpha_{n} = 0$ 
and 
$\sum_{n=0}^{\infty } \alpha_{n} = \infty$. 
\end{assumption}

\begin{assumption} \label{a1.2}
There exists a real number 
$r \in (0, \infty )$ such that 
\begin{align*}
	w
	=
	\limsup_{n\rightarrow \infty } 
	\max_{n\leq k < a(n,1) }
	\left\|
	\sum_{i=n}^{k} \alpha_{i} \gamma_{i}^{r} w_{i} 
	\right\|
	< 
	\infty 
\end{align*}
w.p.1 on 
$\{\sup_{n\geq 0} \|\theta_{n} \| < \infty \}$. 
\end{assumption} 

\begin{assumption} \label{a1.3}
For any compact set $Q \subset \mathbb{R}^{d_{\theta } }$ and any 
$a\in f(Q)$, there exist real numbers 
$\delta_{Q,a} \in (0,1)$, 
$\mu_{Q,a} \in (1,2]$, 
$M_{Q,a} \in[1,\infty )$
such that 
\begin{align} \label{a1.3.1}
	|f(\theta ) - a |
	\leq 
	M_{Q,a} \|\nabla f(\theta ) \|^{\mu_{Q,a} } 
\end{align}
for all $\theta \in Q$ satisfying 
$|f(\theta ) - a | \leq \delta_{Q,a}$. 
\end{assumption}

\begin{assumption} \label{a1.4}
For any compact set $Q \subset \mathbb{R}^{d_{\theta } }$, 
there exist real numbers 
$\nu_{Q} \in (0,1]$, $N_{Q} \in [1,\infty )$ such that 
\begin{align} \label{a1.4.1}
	d(\theta, S )
	\leq 
	N_{Q} \|\nabla f(\theta ) \|^{\nu_{Q} }
\end{align} 
for all $\theta \in Q$. 
\end{assumption}

\begin{remark}
In order to show that Assumption \ref{a1.3} holds,
it is sufficient to demonstrate its `local version,'
i.e., that there exists  
an open vicinity $U$ of $S$
with the following property: 
For any compact set $Q \subset U$ and any $a\in f(Q)$, 
there exit real numbers 
$\delta_{Q,a} \in (0,1]$, 
$\mu_{Q,a} \in (1,2]$, 
$M_{Q,a} \in [1,\infty )$ 
such that (\ref{a1.3.1}) holds 
for all $\theta \in Q$ satisfying 
$|f(\theta ) - a | \leq \delta_{Q,a}$
(for details see the appendix at the end of the paper).
Similar conclusions apply to Assumption \ref{a1.4}. 
\end{remark}

Assumption \ref{a1.1} correspond to the sequence 
$\{\alpha_{n} \}_{n\geq 0}$ and is widely used in the asymptotic 
analysis of stochastic gradient and stochastic approximation 
algorithms. 
Assumption \ref{a1.2} is a noise condition. 
In this or a similar form, it is involved in most of 
the results on the convergence rate of stochastic gradient search
and stochastic approximation. 
It holds for algorithms with Markovian dynamics 
(see the next section). 
It is also satisfied when  
when  
$\{w_{n} \}_{n\geq 0}$ is a 
a martingale-difference sequence. 
Assumptions \ref{a1.3} and \ref{a1.4} are related to the stability 
of the gradient flow 
$d\theta/dt = - \nabla f(\theta )$, or more specifically, to the
geometry of the set of stationary points $S$. 
In the area of differential geometry, 
relations (\ref{a1.3.1}) and (\ref{a1.4.1}) 
are known as the Lojasiewicz inequalities
(see \cite{lojasiewicz1} and \cite{lojasiewicz2} for details). 
They hold if $f(\cdot )$ is analytic or subanalytic in 
an open vicinity of $S$
(see \cite{bierstone&milman}, \cite{lojasiewicz2} for the proof; 
for the form of Lojasiewicz inequality appeared in 
Assumption \ref{a1.3} see \cite[Theorem \L I, p. 775]{kurdyka}; 
for the definition and properties of analytic and 
subanalytic functions, consult 
\cite{bierstone&milman}, \cite{krantz&parks}). 
Although analyticity and subanalyticity are fairly strong conditions, 
they hold for the objective functions of many stochastic gradient 
algorithms commonly used in the areas of 
system identification, signal processing, machine learning, 
operations research and statistical inference. 
E.g., in this paper, we show that     
the objective functions associated with 
supervised and temporal-difference learning  
are analytical
(Sections \ref{section4} and \ref{section5}).  
We also demonstrate the same property for 
recursive prediction error identification (Section \ref{section6}). 
Furthermore, in \cite{tadic5}, we
show analyticity for the objective functions 
associated with   
recursive identification methods for  
hidden Markov models. 
It is also worth mentioning that the objective functions 
associated with recursive algorithms 
for principal and independent component analysis
(as well as with many other adaptive signal processing algorithms)
are usually polynomial or rational, and hence, analytic, too
(see e.g., \cite{cichocki&amari} and references cited therein). 

In order to state the main results of this section, we need further notation. 
For a compact set $Q \subset \mathbb{R}^{d_{\theta } }$, 
$C_{Q} \in [1,\infty )$ 
stands for an upper bound of 
$\|\nabla f(\cdot )\|$ on $Q$ 
and for a Lipschitz constant of $\nabla f(\cdot )$ on the same set. 
$\hat{A}$ denotes the set of accumulation points of 
$\{\theta_{n} \}_{n\geq 0}$
(notice that $\hat{A}$ is a random set), 
while 
\begin{align*}
	\hat{f}
	=
	\liminf_{n\rightarrow \infty } f(\theta_{n} ). 
\end{align*}
$\hat{Q}$ is a random set defined as
\begin{align*}
	\hat{Q}
	=
	\begin{cases}
	\{\theta: d(\theta, \hat{A} ) \leq \rho \}, 
	& \text{ if } 
	\sup_{n\geq 0} \|\theta_{n} \| < \infty
	\\
	\hat{A}, 
	&
	\text{ otherwise }
	\end{cases}
\end{align*}
where $\rho$ is an arbitrary positive 
(deterministic or random) quantity. 
$\hat{\delta}$, $\hat{\mu}$, $\hat{\nu}$, 
$\hat{C}$, $\hat{M}$ and $\hat{N}$
are random quantities defined by 
\begin{align}\label{1.21}
	\hat{\delta}
	=
	\delta_{\hat{Q}, \hat{f} }, 
	\;\;\;
	\hat{\mu}
	=
	\mu_{\hat{Q}, \hat{f} }, 
	\;\;\; 
	\hat{\nu}
	=
	\mu_{\hat{Q}, \hat{f} } \: \nu_{\hat{Q} } \: / 2, 
	\;\;\; 
	\hat{C}
	=
	C_{\hat{Q} }, 
	\;\;\; 
	\hat{M}
	=
	M_{\hat{Q}, \hat{f} }, 
	\;\;\; 
	\hat{N}
	=
	N_{\hat{Q} }
\end{align}
when 
$\sup_{n\geq 0} \|\theta_{n} \| < \infty$
and by 
\begin{align}\label{1.23}
	\hat{\delta}
	=
	1, 
	\;\;\;
	\hat{\mu}
	=
	2, 
	\;\;\; 
	\hat{\nu}
	=
	1, 
	\;\;\; 
	\hat{C}
	=
	1, 
	\;\;\; 
	\hat{M}
	=
	1, 
	\;\;\; 
	\hat{N}
	=
	1
\end{align}
otherwise  
(symbol $\:\hat{}$ is used to emphasize the dependence on $\hat{f}$ 
and $\hat{Q}$).  
Moreover, 
let 
\begin{align}\label{1.25}
	\hat{r} 
	= 
	\begin{cases} 
	1/(2 - \hat{\mu} ), 
	&\text{if } \hat{\mu} < 2
	\\
	\infty, 
	&\text{if } \hat{\mu} = 2
	\end{cases}, 
	\;\;\;\;\; 
	\hat{p} = \hat{\mu} \min\{r,\hat{r} \}, 
	\;\;\;\;\; 
	\hat{q} = \hat{\nu} \min\{r,\hat{r} \}.  
\end{align}
Furthermore, let 
\begin{align*}
	\phi(w)
	=
	\begin{cases}
	w, 
	&\text{if } r < \hat{r} \\
	1+w, 
	&\text{if } r = \hat{r} \\
	1, 
	&\text{if } r> \hat{r}  
	\end{cases}
\end{align*}

\begin{remark}
Since  
$\hat{f} \in f(\hat{Q} )$
when $\sup_{n\geq 0} \|\theta_{n} \| < \infty$, 
it is obvious that 
random quantities 
$\hat{\delta}$, $\hat{\mu}$, $\hat{\nu}$, 
$\hat{p}$, $\hat{q}$, $\hat{r}$, $\hat{C}$, $\hat{M}$, $\hat{N}$
are well-defined. 
Moreover, it is easy to conclude that 
inequalities 
$0 < \hat{\delta} \leq 1$, 
$1 < \hat{\mu} \leq 2$, 
$\hat{p} > \min\{1,r\}$, 
$\hat{q} > 1$, 
$\hat{r} > 1$, 
$1\leq \hat{C}, \hat{M}, \hat{N} < \infty$
hold everywhere
(i.e., on entire $\Omega$). 
It can also be demonstrated that 
(Lojasiewicz coefficients) 
$\delta_{Q,a}$, $\mu_{Q,a}$, $\nu_{Q}$, $M_{Q,a}$, $N_{Q}$
have `measurable versions'
such that 
$\hat{\delta}$, $\hat{\mu}$, $\hat{\nu}$, 
$\hat{p}$, $\hat{q}$, $\hat{r}$, $\hat{M}$, $\hat{N}$
are random variables in probability space
$(\Omega, {\cal F}, P )$
(i.e., measurable with respect to ${\cal F}$; 
details are provided in the appendix at the end of the paper).
Furthermore, 
as a consequence of Assumption \ref{a1.3}, we have
\begin{align}\label{1.1'}
	|f(\theta ) - \hat{f} |
	\leq 
	\hat{M} \|\nabla f(\theta ) \|^{\hat{\mu} }
\end{align}
on $\{\sup_{n\geq 0} \|\theta_{n} \| < \infty \}$
for all $\theta \in \hat{Q}$ 
satisfying $|f(\theta ) - \hat{f} | \leq \hat{\delta}$. 
\end{remark}

Our main results on the convergence and convergence rate of the recursion (\ref{1.1}) 
are contained in the next two theorems.  

\begin{theorem} \label{theorem1.1}
Let Assumptions \ref{a1.1} -- \ref{a1.3} hold. 
Then, $\lim_{n\rightarrow \infty } \nabla f(\theta_{n} ) = 0$
and 
$\lim_{n\rightarrow \infty } f(\theta_{n} ) = \hat{f}$
w.p.1 
on $\{\sup_{n\geq 0} \|\theta_{n} \| < \infty \}$. 
\end{theorem}

\begin{theorem} \label{theorem1.2}
Let Assumptions \ref{a1.1} -- \ref{a1.3} hold. 
Then, 
there exists a random quantity 
$\hat{K}$ (which is a deterministic function of $\hat{C}, \hat{M}$)
such that $1\leq \hat{K} < \infty$ everywhere and 
such that 
\begin{align}
	& \label{t1.1.1*}
	\limsup_{n\rightarrow \infty } 
	\gamma_{n}^{\hat{p} }
	\|\nabla f(\theta_{n} ) \|^{2} 
	\leq 
	\hat{K}
	\big(
	\phi(w)
	\big)^{\hat{\mu} }, 
	\\
	& \label{t1.1.3*}
	\limsup_{n\rightarrow \infty } 
	\gamma_{n}^{\hat{p} }
	|f(\theta_{n} ) - \hat{f} |
	\leq 
	\hat{K} 
	\big(
	\phi(w)
	\big)^{\hat{\mu} }
\end{align}
w.p.1 on 
$\{\sup_{n\geq 0} \|\theta_{n} \| < \infty \}$.
If additionally, Assumption \ref{a1.4} is satisfied, 
then, there exists another random quantity 
$\hat{L}$
(which is a deterministic function of $\hat{C}, \hat{M}, \hat{N}$)
such that $1\leq \hat{L} < \infty$ everywhere and 
such that 
\begin{align}
	& \label{t1.1.5*}
	\limsup_{n\rightarrow \infty } 
	\gamma_{n}^{\hat{q} }
	d(\theta_{n}, S ) 
	\leq 
	\hat{L} 
	\big(
	\phi(w)
	\big)^{\hat{\nu} }
\end{align}
w.p.1 on 
$\{\sup_{n\geq 0} \|\theta_{n} \| < \infty \}$.
\end{theorem}

The proofs are provided in Section \ref{section1*}. 
As an immediate consequence of the previous theorems, we get the following corollaries: 

\begin{corollary} \label{corollary1.1}
Let Assumptions \ref{a1.1} -- \ref{a1.4} hold. 
Then, the following is true: 
\begin{align*}
	\|\nabla f(\theta_{n} ) \|^{2}
	=
	o\big(\gamma_{n}^{-\hat{p} } \big),
	\;\;\; 
	d(f(\theta_{n} ), C ) 
	=
	o\big(\gamma_{n}^{-\hat{p} } \big), 
	\;\;\; 
	d(\theta_{n}, S )
	=
	o\big(\gamma_{n}^{-\hat{q} } \big)
\end{align*} 
w.p.1 on 
$\{\sup_{n\geq 0} \|\theta_{n} \| < \infty \} \cap
\{w = 0, \hat{r} > r \}$,  
and 
\begin{align*}
	\|\nabla f(\theta_{n} ) \|^{2}
	=
	O\big(\gamma_{n}^{-\hat{p} } \big),
	\;\;\; 
	d(f(\theta_{n} ), C ) 
	=
	O\big(\gamma_{n}^{-\hat{p} } \big), 
	\;\;\; 
	d(\theta_{n}, S )
	=
	O\big(\gamma_{n}^{-\hat{q} } \big)
\end{align*}
w.p.1 on 
$\{\sup_{n\geq 0} \|\theta_{n} \| < \infty \} \cap 
\{w = 0, \hat{r} > r \}^{c}$. 
\end{corollary}

\begin{corollary} \label{corollary1.2} 
Let Assumptions \ref{a1.1} -- \ref{a1.3} hold. 
Then, 
\begin{align*}
	\|\nabla f(\theta_{n} ) \|^{2} = o(\gamma_{n}^{-p} ), 
	\;\;\;\;\; 
	d(f(\theta_{n} ), C ) = o(\gamma_{n}^{-p} )
\end{align*}	
w.p.1 on 
$\{\sup_{n\geq 0} \|\theta_{n} \| < \infty \}$, 
where $p=\min\{1,r\}$. 
\end{corollary} 

In the literature on stochastic and deterministic optimization, 
the asymptotic behavior of gradient search is usually 
characterized by
the gradient, objective and estimate convergence, 
i.e., 
by the convergence of sequences 
$\{\nabla f(\theta_{n} ) \}_{n\geq 0}$,  
$\{f(\theta_{n} ) \}_{n\geq 0}$ and   
$\{\theta_{n} \}_{n\geq 0}$
(see e.g., \cite{bertsekas}, \cite{bertsekas&tsitsiklis2}, 
\cite{polyak&tsypkin}, \cite{polyak}  
are references quoted therein). 
Similarly, the convergence rate can be described 
by the rates at which 
$\{\nabla f(\theta_{n} ) \}_{n\geq 0}$,  
$\{f(\theta_{n} ) \}_{n\geq 0}$  
and 
$\{\theta_{n} \}_{n\geq 0}$
tend to the sets of their limit points. 
Theorem \ref{theorem1.2} and Corollary \ref{corollary1.1}
provide relatively tight upper bounds on these rates in the terms 
of the asymptotic properties of 
noise $\{w_{n} \}_{n\geq 0}$
and 
the gradient flow $d\theta/dt = - \nabla f(\theta )$. 
Basically, the theorem and its corollary claim that 
the convergence rate of 
$\{\|\nabla f(\theta_{n} ) \|^{2} \}_{n\geq 0}$ 
and 
$\{f(\theta_{n} ) \}_{n\geq 0}$ is the slower of the rates 
$O(\gamma_{n}^{- \hat{r} \hat{\mu} } )$
(the rate of the gradient flow 
$d\theta/dt = - \nabla f(\theta )$ sampled at instants 
$\{\gamma_{n} \}_{n\geq 0}$) 
and 
$O(\gamma_{n}^{- r \hat{\mu} } )$
(the rate of the noise averages 
$\max_{k\geq n} \|\sum_{i=n}^{k} \alpha_{i} w_{i} \|^{\hat{\mu} }$). 
Apparently, the rates provided in 
Theorem \ref{theorem1.1} and Corollary \ref{corollary1.1} 
are of a local nature: 
They hold only on the event where 
algorithm (\ref{1.1}) is stable
(i.e., where sequence 
$\{\theta_{n} \}_{n\geq 0}$ is bounded). 
Stating results on the convergence rate 
in such a local form is quite reasonable due to the following 
reasons. 
The stability of stochastic gradient search is based on 
well-understood arguments which are rather different from 
the arguments used in the analysis of the convergence rate. 
Moreover and more importantly, 
it is straightforward to get a global version of the rates 
provided in Theorem \ref{theorem1.1} and Corollary \ref{corollary1.1} 
by combining the theorem with 
the methods used to verify or ensure the stability
(e.g., with the results of 
\cite{borkar&meyn} and \cite{chen}). 

Due to its practical and theoretical importance, 
the rate of convergence of stochastic gradient search 
(and stochastic approximation) 
has been the subject of a large number of papers and books
(see see \cite{benveniste}, \cite{kushner&yin}, \cite{ljung2}, 
\cite{polyak}, \cite{spall} 
and references cited therein). 
Although the existing results 
provide a good insight into 
the asymptotic behavior and efficiency 
of stochastic gradient algorithms, 
they are based on fairly restrictive assumptions: 
Literally, they all require the objective function 
$f(\cdot )$ to have an isolated minimum $\hat{\theta}$
(sometimes even to be strongly unimodal) 
such that 
Hessian $\nabla^{2} f(\hat{\theta } )$ is strictly positive definite
and $\lim_{n\rightarrow \infty } \theta_{n} = \hat{\theta}$ w.p.1. 
Unfortunately, 
in the case of high-dimensional and high-nonlinear stochastic gradient algorithms
(such as online machine learning and recursive identification), 
it is hard (if not impossible at all) to 
show even the existence of an isolated minimum, 
let alone the definiteness of $\nabla^{2} f(\cdot )$
and the point-convergence of $\{\theta_{n} \}_{n\geq 0}$.  
Relying on the Lojasiewicz inequalities, 
Theorem \ref{theorem1.1} and Corollary \ref{corollary1.1} 
overcome these difficulties: 
Both the theorem and its corollary allow the objective function 
$f(\cdot )$
to have multiple, non-isolated minima, 
impose no restriction on the values of 
$\nabla^{2} f(\cdot )$ 
(notice that $\nabla^{2} f(\cdot )$ cannot be strictly definite 
at a non-isolated minimum or maximum)
and permit $\{\theta_{n} \}_{n\geq 0}$ to have 
multiple limit points. 
Moreover, they 
cover a broad class of complex stochastic gradient algorithms
(see Sections \ref{section4} and \ref{section6}; see also \cite{tadic5}). 
To the best or our knowledge, 
these are the only results on the convergence rate with such features. 

Regarding the results of Theorem \ref{theorem1.1} and 
Corollary \ref{corollary1.1}, 
it is worth mentioning that they are not just a combination 
of the Lojasiewicz inequalities and the existing techniques for 
the asymptotic analysis of stochastic gradient search and 
stochastic approximation. 
On the contrary, the existing techniques seem to be inapplicable to 
the case of multiple non-isolated minima. 
The reason comes out of the fact that 
these techniques 
crucially rely on 
the Lyapunov function 
$u(\theta ) = 
(\theta - \hat{\theta } )^{T} \nabla^{2} f(\hat{\theta } ) (\theta - \hat{\theta } )$, 
where 
$\hat{\theta }$ is an isolated minimum such that 
$\lim_{n\rightarrow \infty } \theta_{n} = \hat{\theta}$ w.p.1 
and 
$\nabla^{2} f(\cdot )$ is strictly positive definite. 
Unfortunately, in the case of multiple, non-isolated minima, 
neither does 
$\{\theta_{n} \}_{n\geq 0}$ necessarily have a single limit point 
(limit cycles can occur), 
nor $\nabla^{2} f(\cdot )$ can be a strictly positive definite matrix. 
In order to overcome this problem, 
we use a `singular' Lyapunov function 
$v(\theta ) = 1/(f(\theta ) - \hat{f} )^{1/p}$,  
where $p \in (0, \hat{\mu}/(2-\hat{\mu} )]$  
and $\theta\in \{\vartheta \in \mathbb{R}^{d_{\theta } }: f(\vartheta ) > \hat{f} \}$. 
Although subtle techniques are needed to 
handle such a Lyapunov function
(see Section \ref{section1*}), 
$v(\cdot )$ provides intuitively clear explanation of the results of 
Theorem \ref{theorem1.2} and Corollary \ref{corollary1.1}. 
The explanation is based on the heuristic analysis of the following two 
cases. 

{\em Case 1:
$\sup_{n\geq 0} \|\theta_{n} \| < \infty$ and 
$\liminf_{n\rightarrow \infty } 
\gamma_{n}^{r\hat{\mu} } (f(\theta_{n} ) - \hat{f} ) = -\infty$.}
\newline
In this case, 
there exists an increasing integer sequence
$\{n_{k} \}_{k\geq 0}$
such that 
$f(\theta_{n_{k} } ) < \hat{f}$ for each $k\geq 0$
and 
$\lim_{n\rightarrow \infty } 
\gamma_{n_{k} }^{r\hat{\mu} } (f(\theta_{n_{k} } ) - \hat{f} ) = -\infty$.
Therefore, Assumption \ref{a1.3} implies 
$\lim_{n\rightarrow \infty } 
\gamma_{n_{k} }^{r} \|\nabla f(\theta_{n_{k} } ) \| = \infty$. 
Since 
$\max_{k\geq n} \left\|\sum_{i=n}^{k} \alpha_{i} w_{i} \right\|
= O(\gamma_{n}^{-r} )$
(see Lemma \ref{lemma1.1}), 
there exists a large integer $m\gg 1$ such that 
$f(\theta_{m} ) < \hat{f}$
and 
$\max_{n\geq m} \left\|\sum_{i=m}^{n} \alpha_{i} w_{i} \right\|
\leq \|\nabla f(\theta_{m} ) \|/2$.
Then, for $n\geq a(m,1)$, Taylor formula yields 
\begin{align*}
	f(\theta_{n} ) 
	\approx &
	f(\theta_{m} ) 
	-
	(\nabla f(\theta_{m} ) )^{T}
	\sum_{i=m}^{n-1} 
	\alpha_{i} (\nabla f(\theta_{i} ) + w_{i} )
	\\
	\approx &
	f(\theta_{m} )
	-
	\|\nabla f(\theta_{m} ) \|^{2} 
	(\gamma_{n} - \gamma_{m} )
	-
	(\nabla f(\theta_{m} ) )^{T}
	\sum_{i=m}^{n-1} 
	\alpha_{i} w_{i} 
	\\
	\leq &
	f(\theta_{m} )
	-
	\frac{\|\nabla f(\theta_{m} ) \|^{2} }{2}
	-
	\|\nabla f(\theta_{m} ) \|
	\left(
	\frac{\|\nabla f(\theta_{m} ) \| }{2}
	-
	\left\|
	\sum_{i=m}^{n-1} \alpha_{i} w_{i} 
	\right\|
	\right)
	\\
	\leq &
	f(\theta_{m} )
\end{align*}
(notice that $\gamma_{n} - \gamma_{m} \geq 1$). 
Hence, $f(\theta_{n} ) \leq f(\theta_{m} ) < \hat{f}$ 
for $n\geq a(m,1)$, 
which is impossible as 
$\lim_{n\rightarrow \infty } f(\theta_{n} ) = \hat{f}$. 

{\em Case 2:
$\sup_{n\geq 0} \|\theta_{n} \| < \infty$ and 
$\limsup_{n\rightarrow \infty } 
\gamma_{n}^{r\hat{\mu} } (f(\theta_{n} ) - \hat{f} ) = \infty$.}
\newline
Similarly as in the previous case, 
there exists an increasing integer sequence 
$\{n_{k} \}_{k\geq 0}$ such that 
$f(\theta_{n_{k} } ) > \hat{f}$ for each $k\geq 0$
and 
$\lim_{n\rightarrow \infty } 
\gamma_{n_{k} }^{r\hat{\mu} } (f(\theta_{n_{k} } ) - \hat{f} ) = \infty$.
Consequently, Assumption \ref{a1.3} yields 
$\lim_{k\rightarrow \infty } \gamma_{n_{k} }^{r} \|\nabla f(\theta_{n_{k} } ) \| 
= \infty$ and 
\begin{align*}
	\frac{\|\nabla f(\theta_{n_{k} } ) \|^{2} }
	{(f(\theta_{n_{k} } ) - \hat{f} )^{1+1/p} }
	\geq 
	\frac{1}
	{\hat{M}^{2/\hat{\mu} } 
	(f(\theta_{n_{k} } ) - \hat{f} )^{1 + 1/p - 2/\hat{\mu} } }
\end{align*}
for $k\geq 0$. 
Since 
$1 + 1/p \geq 2/\hat{\mu}$, 
$\lim_{n\rightarrow \infty } f(\theta_{n} ) = \hat{f}$
and 
$\max_{k\geq n} \left\|\sum_{i=n}^{k} \alpha_{i} w_{i} \right\|
= O(\gamma_{n}^{-r} )$, 
there exists a large integer 
$m\gg 1$
such that 
$\max_{n\geq m} \left\|\sum_{i=m}^{n} \alpha_{i} w_{i} \right\|
\leq \|\nabla f(\theta_{m} ) \|/2$, 
$f(\theta_{m} ) \geq \hat{f}$
and 
\begin{align*}
	\frac{\|\nabla f(\theta_{m} ) \|^{2} }
	{(f(\theta_{m} ) - \hat{f} )^{1 + 2/p} }
	\geq 
	\frac{1}{\hat{M}^{2/\hat{\mu} } }. 
\end{align*}
Then, for any $n\geq a(m,1)$ satisfying $f(\theta_{n} ) > \hat{f}$, 
Taylor formula implies 
\begin{align*}
	v(\theta_{n} )
	\approx &
	v(\theta_{m} )
	-
	(\nabla v(\theta_{m} ) )^{T} 
	\sum_{i=m}^{n-1} 
	\alpha_{i} (\nabla f(\theta_{i} ) + w_{i} )
	\\
	\approx &
	v(\theta_{m} )
	+
	\frac{\|\nabla f(\theta_{m} ) \|^{2} }
	{p (f(\theta_{m} ) - \hat{f} )^{1+1/p} } 	
	(\gamma_{n} - \gamma_{m} ) 
	+
	\frac{(\nabla f(\theta_{m} ) )^{T} }
	{p (f(\theta_{m} ) - \hat{f} )^{1+1/p} }
	\sum_{i=m}^{n-1} \alpha_{i} w_{i}
	\\
	\geq &
	v(\theta_{m} )
	+
	\frac{1}{2 p \hat{M}^{2/\hat{\mu} } } 
	(\gamma_{n} - \gamma_{m} ) 
	+
	\frac{\|\nabla f(\theta_{m} ) \| }
	{p (f(\theta_{m} ) - \hat{f} )^{1+1/p} }
	\left(
	\frac{1}{2}
	-
	\left\|
	\sum_{i=m}^{n-1} \alpha_{i} w_{i} 
	\right\|	
	\right)
	\\
	\geq &
	\frac{1}{2 p \hat{M}^{2/\hat{\mu} } } 
	(\gamma_{n} - \gamma_{m} ). 
\end{align*}
Thus, 
$f(\theta_{n} ) - \hat{f} 
\leq (2p\hat{M} )^{2p} (\gamma_{n} - \gamma_{m} )^{-p}$
for $n\geq a(m,1)$
(notice that $\hat{\mu} > 1$). 
\vspace{0.3em}

Following the reasoning outlined in the above cases, 
it can easily be concluded that 
the slower of 
$O(\gamma_{n}^{-p})$ and $O(\gamma_{n}^{-r\hat{\mu} })$
is the rate at which $f(\theta_{n} )$ tends to $\hat{f}$. 
Since $p$ can be any number from 
$(0, \hat{r} \hat{\mu} ]$
(in the proof of Theorem \ref{theorem1.1}, Section \ref{section1*}, 
value 
$p = \hat{p} = \hat{\mu} \min\{r,\hat{r} \}$ is used), 
it is also straightforward to deduce that 
$O(\gamma_{n}^{-\hat{p} } )$
is the convergence rate of 
$\{f(\theta_{n} ) \}_{n\geq 0}$. 
In addition to this, the previously described heuristics 
indicate that in the terms of $r$ and $\hat{\mu}$, 
$O(\gamma_{n}^{-\hat{p} } )$ is probably the tightest estimate 
of the convergence rate of $\{f(\theta_{n} ) \}_{n\geq 0}$.  
The same conclusion is suggested by the following two special cases: 

{\em Case (a): $w_{n} = 0$ for each $n\geq 0$.}  
\newline
Due to Assumption \ref{a1.3}, we have 
\begin{align*}
	\frac{d(f(\theta(t) ) - \hat{f} )}{dt}
	=
	-\|\nabla f(\theta(t) ) \|^{2} 
	\leq 
	- 
	\left(1/\hat{M} \right)^{2/\hat{\mu} }
	(f(\theta(t)) - \hat{f} )^{2/\hat{\mu} }
\end{align*}
for a solution $\theta(\cdot )$ of 
$d\theta/dt = - \nabla f(\theta )$
satisfying  
$\theta(t) \in \hat{Q}$ for all $t\in [0,\infty )$ and
$\lim_{t\rightarrow \infty } f(\theta(t) ) = \hat{f}$. 
Consequently, 
$f(\theta(t) ) - \hat{f} = 
O(t^{-\hat{\mu}/(2 - \hat{\mu} ) } ) =
O(t^{-\hat{r} \hat{\mu} })$. 
As $\{\theta_{n} \}_{n\geq 0}$ is asymptotically equivalent to 
$\theta(\cdot )$ sampled at time instances  
$\{\gamma_{n} \}_{n\geq 0}$, 
we get 
$f(\theta_{n} ) - \hat{f} = O(\gamma_{n}^{-\hat{r} \hat{\mu} } )$. 
The same result is implied by Theorem \ref{theorem1.1} 
and Corollary \ref{corollary1.1}. 

{\em Case (b): $f(\theta ) = \theta^{T} A \theta$ and $A$ is a strictly positive definite
matrix.}  
\newline
Recursion (\ref{1.1}) reduces to a linear stochastic approximation 
algorithm in this case. 
For such an algorithm, it is known that the tightest estimate of 
the convergence rate is 
$f(\theta_{n} ) = O(\gamma_{n}^{-2r} )$ if $w>0$, 
and 
$f(\theta_{n} ) = o(\gamma_{n}^{-2r} )$ for $w=0$
(see \cite{tadic1}). 
The same rate is provided by Theorem \ref{theorem1.2} and 
Corollary \ref{corollary1.1}. 

\section{Stochastic Gradient Algorithms with Markovian Dynamics} \label{section2}

In order to illustrate the results of Section \ref{section1} and 
to set up a framework for the analysis carried out in Sections 
\ref{section4} and \ref{section6}, 
we apply Theorems \ref{theorem1.1}, \ref{theorem1.2} and 
Corollaries \ref{corollary1.1}, \ref{corollary1.2} 
to stochastic gradient algorithms with Markovian dynamics. 
These algorithms are defined by the following difference equation: 
\begin{align} \label{2.1}
	\theta_{n+1}
	=
	\theta_{n} 
	-
	\alpha_{n} F(\theta_{n}, \xi_{n+1} ), 
	\;\;\; 
	n\geq 0. 
\end{align}
In this recursion, 
$F: \mathbb{R}^{d_{\theta} } \times \mathbb{R}^{d_{\xi} } \rightarrow \mathbb{R}^{d_{\theta } }$ is 
a Borel-measurable function, 
while $\{\alpha_{n} \}_{n\geq 0}$ is a sequence of positive real numbers. 
$\theta_{0}$ is an $\mathbb{R}^{d_{\theta } }$-valued random variable defined on 
a probability space $(\Omega, {\cal F}, P)$, 
while 
$\{\xi_{n} \}_{n\geq 0}$ is an $\mathbb{R}^{d_{\xi } }$-valued stochastic process 
defined on the same probability space. 
$\{\xi_{n} \}_{n\geq 0}$ is a Markov process controlled by 
$\{\theta_{n} \}_{n\geq 0}$, i.e., 
there exists a family of transition probability kernels 
$\{\Pi_{\theta }(\cdot,\cdot ) \}_{\theta \in \mathbb{R}^{d_{\theta } } }$
defined on $\mathbb{R}^{d_{\xi } }$
such that 
\begin{align*}
	P(\xi_{n+1} \in B|\theta_{0},\xi_{0},\dots,\theta_{n},\xi_{n} ) 
	=
	\Pi_{\theta_{n} }(\xi_{n}, B )
\end{align*}
w.p.1 for any Borel-measurable set $B \subseteq \mathbb{R}^{d_{\xi } }$ and $n\geq 0$. 
In the context of stochastic gradient search, 
$F(\theta_{n},\xi_{n+1} )$ is regarded to as an estimator of 
$\nabla f(\theta_{n} )$. 

The algorithm (\ref{2.1}) is analyzed under the following assumptions. 

\begin{assumption} \label{a2.1}
$\lim_{n\rightarrow \infty } \alpha_{n} = 0$, 
$\limsup_{n\rightarrow \infty } |\alpha_{n+1}^{-1} - \alpha_{n}^{-1} | < \infty$ 
and 
$\sum_{n=0}^{\infty } \alpha_{n} = \infty$. 
There exists a real number $r\in (0,\infty )$
such that 
$\sum_{n=0}^{\infty } \alpha_{n}^{2} \gamma_{n}^{2r} < \infty$. 
\end{assumption}

\begin{assumption} \label{a2.2}
There exist a differentiable function 
$f: \mathbb{R}^{d_{\theta } } \rightarrow \mathbb{R}$ 
and 
a Borel-measurable function 
$\tilde{F}: \mathbb{R}^{d_{\theta} } \times \mathbb{R}^{d_{\xi} } \rightarrow \mathbb{R}^{d_{\theta } }$
such that 
$\nabla f(\cdot )$ is locally Lipschitz continuous and 
such that 
\begin{align*}
	F(\theta, \xi ) 
	-
	\nabla f(\theta )
	=
	\tilde{F}(\theta, \xi ) 
	-
	(\Pi\tilde{F} )(\theta, \xi )
\end{align*}
for each $\theta\in \mathbb{R}^{d_{\theta } }$, 
$\xi \in \mathbb{R}^{d_{\xi } }$, 
where 
$(\Pi\tilde{F} )(\theta, \xi ) 
= \int \tilde{F}(\theta, \xi' ) \Pi_{\theta }(\xi, d\xi' )$. 
\end{assumption}

\begin{assumption} \label{a2.3}
For any compact set $Q \subset \mathbb{R}^{d_{\theta } }$ and $s \in (0,1)$, 
there exists a Borel-measurable function 
$\varphi_{Q, s }: \mathbb{R}^{d_{\xi } } \rightarrow [1,\infty )$ such that 
\begin{align*}
	&
	\max\{
	\|F(\theta,\xi ) \|, \|\tilde{F}(\theta,\xi ) \|, \|(\Pi \tilde{F} )(\theta,\xi ) \|
	\}
	\leq 
	\varphi_{Q, s }(\xi ), 
	\\
	&
	\|(\Pi \tilde{F} )(\theta',\xi ) - (\Pi \tilde{F} )(\theta'',\xi ) \|
	\leq 
	\varphi_{Q, s }(\xi ) \|\theta' - \theta'' \|^{s} 
\end{align*}
for all 
$\theta, \theta', \theta'' \in Q$, $\xi \in \mathbb{R}^{d_{\xi } }$.  
\end{assumption}

\begin{assumption} \label{a2.4}
Given a compact set $Q \subset \mathbb{R}^{d_{\theta } }$ and $s \in (0,1)$, 
\begin{align*}
	\sup_{n\geq 0}
	E\left(
	\varphi_{Q, s }^{2}(\xi_{n} ) 
	I_{\{\tau_{Q} \geq n \} }
	|\theta_{0}=\theta, \xi_{0}=\xi 
	\right)
	< 
	\infty
\end{align*}
for all $\theta \in \mathbb{R}^{d_{\theta } }$, $\xi \in \mathbb{R}^{d_{\xi } }$, 
where 
$\tau_{Q} = 
\inf\{n\geq 0: \theta_{n} \not\in Q \}$. 
\end{assumption}

The main results on the convergence rate of recursion 
(\ref{2.1}) are in the next theorem. 

\begin{theorem} \label{theorem2.1}
Let Assumptions \ref{a2.1} -- \ref{a2.4} hold, 
and suppose that $f(\cdot )$ (introduced in Assumption \ref{a2.2}) 
satisfies Assumptions \ref{a1.3} and \ref{a1.4}. 
Then, 
\begin{align*}
	\|\nabla f(\theta_{n} ) \|^{2} = o(\gamma_{n}^{-p} ), 
	\;\;\;\;\; 
	d(f(\theta_{n} ), C ) = o(\gamma_{n}^{-p} )
\end{align*} 
w.p.1 on 
$\{\sup_{n\geq 0} \|\theta_{n} \| < \infty \}$. 
Moreover, 
the following is true: 
\begin{align*}
	\|\nabla f(\theta_{n} ) \|^{2}
	=
	o\big(\gamma_{n}^{-\hat{p} } \big),
	\;\;\; 
	d(f(\theta_{n} ), C )
	=
	o\big(\gamma_{n}^{-\hat{p} } \big),
	\;\;\; 
	d(\theta_{n}, S )
	=
	o\big(\gamma_{n}^{-\hat{q} } \big)
\end{align*} 
w.p.1 on 
$\{\sup_{n\geq 0} \|\theta_{n} \| < \infty \} 
\cap \{\hat{r} > r\}$,  
and 
\begin{align*}
	\|\nabla f(\theta_{n} ) \|^{2}
	=
	O\big(\gamma_{n}^{-\hat{p} } \big),
	\;\;\; 
	d(f(\theta_{n} ), C )
	=
	O\big(\gamma_{n}^{-\hat{p} } \big), 
	\;\;\; 
	d(\theta_{n}, S )
	=
	O\big(\gamma_{n}^{-\hat{q} } \big)
\end{align*} 
w.p.1 on 
$\{\sup_{n\geq 0} \|\theta_{n} \| < \infty \} 
\cap \{\hat{r} \leq r\}$. 
\end{theorem}

The proof is provided in Section \ref{section2*}. 
$C, S, p, \hat{p}, \hat{q}$ and $\hat{r}$
are defined in Section \ref{section1}. 

Assumption \ref{a2.1} is related to the sequence  
$\{\alpha_{n} \}_{n\geq 0}$. 
It holds if 
$\alpha_{n} = 1/n^{a}$ for $n\geq 1$, 
where $a \in (1/2,1]$ is a constant. 
On the other side,  
Assumptions \ref{a2.2} -- \ref{a2.4} 
correspond to   
the stochastic process 
$\{\xi_{n} \}_{n\geq 0}$
and are quite standard for the asymptotic analysis of 
stochastic approximation algorithms with Markovian dynamics. 
Assumptions \ref{a2.2} -- \ref{a2.4} have been introduced by 
Metivier and Priouret in \cite{metivier&priouret1} 
(see also \cite[Part II]{benveniste}), 
and later generalized by Kushner and his co-workers 
(see \cite{kushner&yin} and references cited therein). 
However, 
neither the results of Metivier and Priouret, 
nor the results of Kushner and his co-workers 
provide any information on the convergence rate 
of stochastic gradient search in 
the case of multiple, non-isolated 
minima.  

Regarding Theorem \ref{theorem2.1}, 
the following note is also in order.  
As already mentioned in the beginning of the section, 
the purpose of the theorem 
is illustrating the results of Theorem \ref{theorem1.1} 
and providing a framework for studying the examples 
presented in the next sections. 
Since these examples perfectly fit into the framework 
developed by Metivier and Priouret, 
more general assumptions and settings of 
\cite{kushner&yin} are not considered here 
in order just to keep the exposition as concise as possible. 

\section{Example 1: Supervised Learning} \label{section4} 

In this section, 
online algorithms for supervised learning in  
feedforward neural networks are analyzed using the 
results of Theorems \ref{theorem1.2} and \ref{theorem2.1}. 

To state the problem of supervised learning  
and to define the corresponding algorithms, we need the following notation. 
$N_{1}$ and $N_{2}$ are positive integers, 
while $d_{\theta } = N_{1} (N_{2} + 1 )$.  
$\phi_{1}, \phi_{2}: \mathbb{R} \rightarrow \mathbb{R}$ are differentiable functions, 
while 
$\psi_{1},\dots,\psi_{N_{2} }: \mathbb{R}^{d_{x} } \rightarrow \mathbb{R}$ are 
Borel-measurable functions. 
For $a'_{1},\dots,a'_{N_{1} } \in \mathbb{R}$, 
$a''_{1,1},\dots,a''_{N_{1},N_{2} } \in \mathbb{R}$, 
$x \in \mathbb{R}^{d_{x} }$, let  
\begin{align*}
	G_{\theta }(x) 
	=
	\phi_{1}\left(
	\sum_{i_{1}=1}^{N_{1} } a'_{i_{1} } 
	\phi_{2}\left(
	\sum_{i_{2}=1}^{N_{2}}
	a''_{i_{1},i_{2} } \psi_{i_{2}}(x) 
	\right)
	\right), 
\end{align*}
where $\theta = [a'_{1} \cdots a'_{N_{1}} \; a''_{1,1} \cdots a''_{N_{1},N_{2} }]^{T}$. 
Moreover, $\pi(\cdot,\cdot )$ denotes a probability measure 
on $\mathbb{R}^{d_{x} } \times \mathbb{R}$, 
while 
\begin{align*}
	f(\theta ) 
	= 
	\frac{1}{2}
	\int (y-G_{\theta }(x) )^{2} \pi(dx,dy)
\end{align*}
for $\theta \in \mathbb{R}^{d_{\theta } }$. 
Then, the mean-square error based supervised learning
in feedforward neural networks can be described as the minimization 
of $f(\cdot )$
in a situation when only samples from 
$\pi(\cdot,\cdot )$ are available. 
In this context, $G_{\theta }(\cdot )$ represents 
the input-output function 
(i.e., the architecture) 
of the feedforward neural network to be trained. 
$\phi_{1}(\cdot )$ and $\phi_{2}(\cdot )$ are the network activation functions, 
while 
$\theta$ is the vector of the network parameters to be tuned through the process 
of supervised learning. 
For more details on neural networks and supervised learning, 
see e.g., \cite{hastie&tibshirani&friedman}, 
\cite{haykin} and references cited therein. 

Function $f(\cdot )$ is usually minimized by 
the following stochastic gradient algorithm: 
\begin{align} \label{4.1} 
	\theta_{n+1} 
	=
	\theta_{n} 
	+
	\alpha_{n} 
	(y_{n} - G_{\theta_{n} }(x_{n} ) ) H_{\theta_{n} }(x_{n} ), 
	\;\;\; 
	n\geq 0. 
\end{align}
In this recursion, 
$\{\alpha_{n} \}_{n\geq 0}$
is a sequence of positive real numbers, 
while 
$H_{\theta }(\cdot ) = \nabla_{\theta } G_{\theta }(\cdot )$. 
$\theta_{0}$
is an $\mathbb{R}^{d_{\theta } }$-valued random variable
defined on a probability space $(\Omega, {\cal F}, P)$,  
while 
$\{(x_{n}, y_{n} ) \}_{n\geq 0}$
is an $\mathbb{R}^{d_{\theta } } \times \mathbb{R}$-valued 
stochastic process defined on the same probability space. 
In the context of supervised learning, 
$\{x_{n}, y_{n} \}_{n\geq 0}$
is regarded to as a training sequence. 

The asymptotic behavior of algorithm (\ref{4.1}) is analyzed under the 
following assumptions: 

\begin{assumption} \label{a4.1}
$\phi_{1}(\cdot )$ and $\phi_{2}(\cdot )$ are real-analytic. 
Moreover, 
$\phi_{1}(\cdot )$ and $\phi_{2}(\cdot )$ have 
(complex-valued) continuations 
$\hat{\phi}_{1}(\cdot )$ and $\hat{\phi}_{2}(\cdot )$ 
(respectively) with the following properties: 
\begin{romannum}
\item
$\phi_{1}(z)$ and $\phi_{2}(z)$ map $z\in \mathbb{C}$ into $\mathbb{C}$
($\mathbb{C}$ denotes the set of complex numbers). 
\item
$\hat{\phi}_{1}(x) = \phi_{1}(x)$
and 
$\hat{\phi}_{2}(x) = \phi_{2}(x)$
for all $x\in \mathbb{R}$. 
\item
There exist real numbers
$\varepsilon \in (0,1)$, $K \in [1,\infty )$
such that 
$\hat{\phi}_{1}(\cdot )$ and $\hat{\phi}_{2}(\cdot )$
are analytic
on $\hat{V}_{\varepsilon } = 
\{z\in \mathbb{C}: d(z, \mathbb{R} ) \leq \varepsilon \}$, and such that 
\begin{align*}
	&
	|\hat{\phi}_{1}(z) |
	\leq 
	K (1 + |z| ), 
	\\
	&
	\max\{|\hat{\phi}'_{1}(z) |, |\hat{\phi}_{2}(z) |, |\hat{\phi}'_{2}(z) | \}
	\leq 
	K
\end{align*} 
for all $z \in \hat{V}_{\varepsilon }$
($\hat{\phi}_{1}(\cdot )$, $\hat{\phi}_{2}(\cdot )$
are the derivatives of $\hat{\phi}_{1}(\cdot )$, $\hat{\phi}_{2}(\cdot )$). 
\end{romannum}
\end{assumption} 

\begin{assumption} \label{a4.2}
$\{(x_{n}, y_{n} ) \}_{n\geq 0}$ are i.i.d. random variables 
distributed according 
the probability measure $\pi(\cdot, \cdot )$. 
There exists a real number $L \in [1,\infty )$ such that 
$\max_{1\leq k \leq N_{2} } |\psi_{k}(x_{0})| \leq L$ and $|y_{0} | \leq L$ w.p.1. 
\end{assumption}

Our main results on the properties of 
objective function $f(\cdot )$
and algorithm (\ref{4.1})
are contained in the next two theorems. 

\begin{theorem} \label{theorem4.1}
Let Assumptions \ref{a4.1} and \ref{a4.2} hold.
Then, $f(\cdot )$ is analytic on entire $\mathbb{R}^{d_{\theta } }$, 
i.e., 
it satisfies Assumptions \ref{a1.3} and \ref{a1.4}. 
\end{theorem}

\begin{theorem} \label{theorem4.2}
Let Assumptions \ref{a2.1}, \ref{a4.1} and \ref{a4.2} hold. 
Then, 
\begin{align*}
	\|\nabla f(\theta_{n} ) \|^{2} = o(\gamma_{n}^{-p} ), 
	\;\;\;\;\; 
	d(f(\theta_{n} ), C ) = o(\gamma_{n}^{-p} )
\end{align*}
w.p.1 on 
$\{\sup_{n\geq 0} \|\theta_{n} \| < \infty \}$. 
Moreover, 
the following is true: 
\begin{align*}
	\|\nabla f(\theta_{n} ) \|^{2}
	=
	o\big(\gamma_{n}^{-\hat{p} } \big),
	\;\;\; 
	d(f(\theta_{n} ), C )
	=
	o\big(\gamma_{n}^{-\hat{p} } \big),
	\;\;\; 
	d(\theta_{n}, S )
	=
	o\big(\gamma_{n}^{-\hat{q} } \big)
\end{align*} 
w.p.1 on 
$\{\sup_{n\geq 0} \|\theta_{n} \| < \infty \} \cap \{\hat{r} > r\}$,  
and 
\begin{align*}
	\|\nabla f(\theta_{n} ) \|^{2}
	=
	O\big(\gamma_{n}^{-\hat{p} } \big),
	\;\;\; 
	d(f(\theta_{n} ), C )
	=
	O\big(\gamma_{n}^{-\hat{p} } \big), 
	\;\;\; 
	d(\theta_{n}, S )
	=
	O\big(\gamma_{n}^{-\hat{q} } \big)
\end{align*} 
w.p.1 on 
$\{\sup_{n\geq 0} \|\theta_{n} \| < \infty \} \cap \{\hat{r} \leq r\}$. 
\end{theorem}

The proofs are provided in Section \ref{section4*}. 
$C, S, p, \hat{p}, \hat{q}$ and $\hat{r}$
are defined in Section \ref{section1}. 

Assumption \ref{a4.1} is related to the neural network being trained. 
It covers some of the most popular feedforward architectures 
such as backpropagation networks with logistic activations\footnote
{Since 
\begin{align*}
	|1 + \exp(-z) |^{2}
	=
	1 
	+ 
	\exp(-2\text{Re}(z) )
	+
	2 \exp(-\text{Re}(z) ) \cos(\text{Im}(z) )
	\geq 
	1 + \exp(-2\text{Re}(z) )
\end{align*}
when $|\text{Im}(z) | \leq \pi/2$, 
complex-valued logistic function 
$h(z) = (1 + \exp(-z) )^{-1}$
is analytical on 
$\{z\in \mathbb{C}: d(z,\mathbb{R}) \leq \pi/2 \}$. 
Due to the same reason, 
$\max\{|h(z)|, |h'(z) | \} \leq 1$
on $\{z\in \mathbb{C}: d(z,\mathbb{R}) \leq \pi/2 \}$. 
} 
and 
radial basis function networks with Gaussian activations\footnote
{Complex-valued Gaussian activation 
$h(z) = (2\pi )^{-1/2} \exp(-z^{2}/2 )$
is analytical on entire $\mathbb{C}$. 
As
\begin{align*}
	(1 + |z| ) \exp(-z^{2}/2 )
	\leq 
	(1 + |\text{Re}(z) | + |\text{Im}(z) | )
	\exp(- \text{Re}^{2}(z)/2 + \text{Im}^{2}(z)/2 )
	\leq 
	3e
\end{align*}
when $|\text{Im}(z) | \leq 1$, 
we have 
$\max\{|h(z)|,|h'(z)| \} \leq 3e$
on $\{z\in \mathbb{C}: d(z,\mathbb{R}) \leq 1 \}$. 
}. 
On the other side, Assumption \ref{a4.2} corresponds to the 
training sequence 
$\{x_{n}, y_{n} \}_{n\geq 0}$, 
and is quite common for the analysis of supervised learning. 

The asymptotic properties of supervised learning algorithms have been 
studied in a large number of papers 
(see \cite{hastie&tibshirani&friedman}, \cite{haykin} 
and references cited therein). 
Unfortunately, the available literature does not provide 
any information on the rate of convergence which can be 
verified for the feedforward networks with nonlinear 
activation functions. 
The main difficulty comes out of the fact that 
the existing results on the convergence rate of stochastic 
gradient search require the objective function to have an isolated 
minimum at which the Hessian is strictly positive definite. 
Since the objective function is highly nonlinear in the case of 
supervised learning algorithms, 
it is hard (if not impossible) to show even the existence 
of isolated minima, let alone the definiteness of the Hessian. 
As opposed to the existing results, 
Theorem \ref{theorem4.2} does not invoke any of these requirements 
and covers some of the most widely used feedforward neural networks. 

\section{Example 2: Temporal Difference Learning} \label{section5} 

In this section, the results of Theorems \ref{theorem1.2} and 
\ref{theorem2.1} are illustrated by 
applying them to the analysis of temporal-difference learning 
algorithms. 

In order to explain temporal-difference learning and 
to define the corresponding algorithms, 
we use the following notation. 
$N > 1$ is an integer, 
while 
${\cal X} = \{1,\dots,N \}$. 
$\{x_{n} \}_{n\geq 0}$
is an ${\cal X}$-valued Markov chain defined on a 
probability space 
$(\Omega, {\cal F}, P )$, while 
$\{ c(i) \}_{i\in {\cal X} }$
are real numbers.  
$\beta \in (0,1)$ is a constant, while
\begin{align*}
	g(i)
	=
	E\left(\left. 
	\sum_{n=0}^{\infty } \beta^{n} c(x_{n} ) 
	\right| x_{0}= i
	\right)
\end{align*}
for $i\in {\cal X}$. 
For each $i\in {\cal X}$, 
$G_{\theta }(i)$ is a real-valued differentiable function of 
$\theta \in \mathbb{R}^{d_{\theta } }$, 
while 
\begin{align*}
	f(\theta )
	=
	\frac{1}{2}
	\lim_{n\rightarrow \infty } 
	E(g(x_{n} ) - G_{\theta }(x_{n} ) )^{2}
\end{align*}
for $\theta\in \mathbb{R}^{d_{\theta } }$. 
With this notation, the problem of temporal-difference learning can be 
posed as the minimization of 
$f(\cdot )$ in a situation when only a realization of 
$\{x_{n} \}_{n\geq 0}$ is available. 
In this context, 
$c(i)$ is considered as a cost of visiting state
$i$, 
while 
$g(i)$ is regarded to as a total discounted cost incurred by 
$\{x_{n} \}_{n\geq 0}$ 
when $\{x_{n} \}_{n\geq 0}$ starts from state $i$. 
$G_{\theta }(\cdot )$
is a parameterized approximation of $g(\cdot )$, 
while $\theta$ is the parameter to be tuned 
through the process of temporal-difference learning. 
For more details on temporal-difference learning, 
see e.g., 
\cite{bertsekas&tsitsiklis1}, \cite{powell}, \cite{sutton&barto} 
and references cited therein. 

Function $f(\cdot )$ can be minimized by the following 
algorithm: 
\begin{align}
	& \label{5.1}
	\theta_{n+1} 
	=
	\theta_{n} 
	+
	\alpha_{n} 
	(c(x_{n} ) + \beta G_{\theta_{n} }(x_{n+1} ) - G_{\theta_{n} }(x_{n} ) )
	y_{n}, 
	\\
	& \label{5.3}
	y_{n+1}
	=
	\beta y_{n} 
	+
	H_{\theta_{n} }(x_{n+1} ), 
	\;\;\; 
	n\geq 0. 
\end{align}
In this recursion, 
$\{\alpha_{n} \}_{n\geq 0}$ is a sequence of positive reals, 
while 
$H_{\theta }(\cdot ) = \nabla_{\theta } G_{\theta }(\cdot )$. 
$\theta_{0}$ is an $\mathbb{R}^{d_{\theta } }$-valued random variable, 
which is defined on probability space 
$(\Omega, {\cal F}, P )$ and independent of 
$\{x_{n} \}_{n\geq 0}$. 
In the literature on reinforcement learning, 
recursion (\ref{5.1}), (\ref{5.3}) is known as 
$TD(1)$ temporal-difference learning algorithm
with a nonlinear function approximation, 
while $G_{\theta }(\cdot )$ is referred to as 
a function approximation, or just as an `approximator.'

We analyze algorithm 
(\ref{5.1}), (\ref{5.3}) under the following assumptions: 

\begin{assumption} \label{a5.1}
$\{x_{n} \}_{n\geq 0}$ is geometrically ergodic. 
\end{assumption} 

\begin{assumption} \label{a5.2} 
For each $i$, $G_{\theta }(i)$ is analytic in $\theta$
on entire $\mathbb{R}^{d_{\theta } }$. 
\end{assumption}

Our main results on the properties of 
$f(\cdot )$
and asymptotic behavior of the algorithm 
(\ref{5.1}), (\ref{5.3}) are presented in the next two theorems. 

\begin{theorem} \label{theorem5.1}
Let Assumptions \ref{a5.1} and \ref{a5.2} hold.
Then, $f(\cdot )$ is analytic on entire $\mathbb{R}^{d_{\theta } }$, 
i.e., 
it satisfies Assumptions \ref{a1.3} and \ref{a1.4}. 
\end{theorem}

\begin{theorem} \label{theorem5.2}
Let Assumptions \ref{a2.1}, \ref{a5.1} and \ref{a5.2} hold. 
Then, 
\begin{align*}
	\|\nabla f(\theta_{n} ) \|^{2} = o(\gamma_{n}^{-p} ), 
	\;\;\; 
	d(f(\theta_{n} ), C ) = o(\gamma_{n}^{-p} )
\end{align*}	
w.p.1 on 
$\{\sup_{n\geq 0} \|\theta_{n} \| < \infty \}$. 
Moreover, 
the following is true: 
\begin{align*}
	\|\nabla f(\theta_{n} ) \|^{2}
	=
	o\big(\gamma_{n}^{-\hat{p} } \big),
	\;\;\; 
	d(f(\theta_{n} ), C )
	=
	o\big(\gamma_{n}^{-\hat{p} } \big),
	\;\;\; 
	d(\theta_{n}, S )
	=
	o\big(\gamma_{n}^{-\hat{q} } \big)
\end{align*} 
w.p.1 on 
$\{\sup_{n\geq 0} \|\theta_{n} \| < \infty \} \cap \{\hat{r} > r\}$,  
and 
\begin{align*}
	\|\nabla f(\theta_{n} ) \|^{2}
	=
	O\big(\gamma_{n}^{-\hat{p} } \big),
	\;\;\; 
	d(f(\theta_{n} ), C )
	=
	O\big(\gamma_{n}^{-\hat{p} } \big), 
	\;\;\; 
	d(\theta_{n}, S )
	=
	O\big(\gamma_{n}^{-\hat{q} } \big)
\end{align*} 
w.p.1 on 
$\{\sup_{n\geq 0} \|\theta_{n} \| < \infty \} \cap \{\hat{r} \leq r\}$. 
\end{theorem}

The proofs are provided in Section \ref{section5*}. 
$C, S, p, \hat{p}, \hat{q}$ and $\hat{r}$
are defined in Section \ref{section1}. 

Assumption \ref{a5.1} corresponds to the stability of Markov chain 
$\{x_{n} \}_{n\geq 0}$. 
In this or similar form, it is involved in any result on the 
asymptotic behavior of temporal-difference learning. 
On the other side, 
Assumption \ref{a5.2} is related to the properties of  
$G_{\theta }(\cdot )$. 
It covers some of the most popular function approximations 
used in the area of reinforcement learning 
(e.g., polynomial approximations and 
feedforward neural networks with analytic activation functions; 
for details see \cite{bertsekas&tsitsiklis1}, \cite{powell}, \cite{sutton&barto}). 

Asymptotic properties of temporal-difference learning 
have been the subject of a number of papers
(see \cite{bertsekas&tsitsiklis1}, \cite{powell} and references cited therein). 
However, the available literature on reinforcement learning 
does not offer any information 
on the rate of convergence of the algorithm 
(\ref{5.1}), (\ref{5.3}) in the case when 
$G_{\theta }(\cdot )$ is nonlinear in $\theta$. 
Similarly as in the case of supervised learning, 
the main difficulty is caused by 
the fact that the existing results on the convergence rate of 
stochastic gradient search require 
$f(\cdot )$ to have an isolated minimum at which 
$\nabla^{2} f(\cdot )$ is strictly positive definite. 
Unless $G_{\theta }(\cdot )$ is linear in $\theta$, 
$f(\cdot )$ is so complex that these requirements are practically impossible 
to show. 
On the other side, Theorem \ref{theorem5.2} does not impose any restriction on 
the topological properties of the minima of $f(\cdot )$, or on the values of 
$\nabla^{2} f(\cdot )$. 
Moreover, it can be applied to many temporal-difference learning algorithms 
met in practice. 

Regarding the results of this section, the following note is also in order. 
Using the arguments Theorems \ref{theorem4.1} and \ref{theorem5.2}    
are based on, 
it is possible (at the cost of increasing significantly the amount of 
technical details) 
to generalize Theorems \ref{theorem5.1} and \ref{theorem5.2}
to the case when $\{x_{n} \}_{n\geq 0}$ is a continuous state Markov chain, 
as well as 
to actor-critic learning algorithms proposed in \cite{konda&tsitsiklis}. 

\section{Example 2: Identification of Linear Stochastic Dynamical Systems} \label{section6}

In this section, the general results presented in Sections \ref{section1} and \ref{section2} 
are applied to the asymptotic analysis of recursive prediction error 
algorithms for identification of linear stochastic dynamical systems. 
To avoid unnecessary technical details and complicated notation, 
only the identification of one dimensional ARMA models is considered here. 
However, it is straightforward to generalize the obtained results to 
any linear stochastic dynamical system. 

In order to state the problem of recursive prediction error identification in ARMA models, 
we use the following notation. 
$M$ and $N$ are positive integers, 
while $d_{\theta } = M+N$. 
For
$a_{1}, \dots, a_{M} \in \mathbb{R}$ and $b_{1}, \dots, b_{N} \in \mathbb{R}$, 
let
\begin{align*}
	A_{\theta }(z) 
	=
	1 
	-
	\sum_{k=1}^{M} a_{k} z^{-k}, 
	\;\;\;\;\; 
	B_{\theta }(z) 
	=
	1 
	+
	\sum_{k=1}^{N} b_{k} z^{-k},  
\end{align*}
where  
$\theta = [a_{1} \cdots a_{M} \; b_{1} \cdots b_{N} ]^{T}$
and $z\in \mathbb{C}$ ($\mathbb{C}$ denotes the set of complex numbers). 
Moreover, let 
\begin{align*}
	\Theta
	=
	\{\theta \in \mathbb{R}^{d_{\theta} }: B_{\theta}(z) = 0 \Rightarrow |z| > 1 \}.  
\end{align*}
On the other side, 
$\{y_{n} \}_{n\geq 0}$ is a real-valued signal generated by the actual system
(i.e., by the system being identified). 
For $\theta \in \Theta$, 
$\{y_{n}^{\theta} \}_{n\geq 0}$ is the output of the ARMA model 
\begin{align} \label{6.1'}
	A_{\theta }(q) y_{n}^{\theta } 
	=
	B_{\theta }(q) e_{n}, 
	\;\;\; n \geq 0, 
\end{align}
where 
$\{e_{n} \}_{\geq 0}$ is a real-valued white noise and 
$q^{-1}$ is the backward time-shift operator.  
$\{\varepsilon_{n}^{\theta } \}_{n\geq 0}$ is the process 
generated by the recursion 
\begin{align} \label{6.1''}
	B_{\theta }(q) \varepsilon_{n}^{\theta } 
	=
	A_{\theta }(q) y_{n}, 
	\;\;\; n\geq 0, 
\end{align}
while $\hat{y}_{n}^{\theta } = y_{n} - \varepsilon_{n}^{\theta }$
and 
\begin{align*}
	f(\theta )
	=
	\frac{1}{2} 
	\lim_{n\rightarrow \infty } 
	E\left((\varepsilon_{n}^{\theta } )^{2} \right). 
\end{align*}
Then, 
$\hat{y}_{n}^{\theta }$ is 
a mean-square optimal estimate  
of $y_{n}$ given 
$y_{0},\dots,y_{n-1}$
(which the model (\ref{6.1'}) can provide; 
see e.g., \cite{ljung2}, \cite{ljung3}). 
Consequently, $\varepsilon_{n}^{\theta }$ can be interpreted 
as the estimation error. 

The parametric identification in ARMA models can be defined as the following 
estimation problem: 
Given a realization of $\{y_{n} \}_{n\geq 0}$, 
estimate the values of $\theta$ for which the model (\ref{6.1'})
provides the best approximation to the signal 
$\{y_{n} \}_{n\geq 0}$. 
If the identification is based on the prediction error principle, 
the estimation problem reduces to the minimization of 
$f(\cdot )$ over $\Theta$. 
As the asymptotic value of the second moment of $\varepsilon_{n}^{\theta }$
is rarely available analytically, 
$f(\cdot )$ is minimized by a stochastic gradient (or stochastic Newton) 
algorithm. 
Such an algorithm is defined by the following difference equations: 
\begin{align}
	& \label{6.1}
	\phi_{n}
	=
	[y_{n} \cdots y_{n-M+1 } \; \varepsilon_{n} \cdots \varepsilon_{n-N+1 } ]^{T}, 
	\\
	& \label{6.3}
	\varepsilon_{n+1} 
	=
	y_{n+1} - \phi_{n}^{T} \theta_{n}, 
	\\
	& \label{6.5}
	\psi_{n+1} 
	=
	\phi_{n}
	-
	[\psi_{n} \cdots \psi_{n-N+1 } ]^{T} 
	A_{0} \theta_{n}, 
	\\
	& \label{6.7}
	\theta_{n+1} 
	=
	\theta_{n} 
	+
	\alpha_{n} \psi_{n+1} \varepsilon_{n+1}, 
	\;\;\;\;\; n\geq 0.  
\end{align}
In this recursion, 
$\{\alpha_{n} \}_{n\geq 0}$ denotes a sequence of positive reals, 
while $A_{0}$ is a composite matrix 
defined as $A_{0} = [0_{N \times M } \; I_{N \times N } ]$. 
$\{y_{n} \}_{n\geq -M }$ is a real-valued stochastic process 
defined on a probability space
$(\Omega, {\cal F}, P)$, while 
$\theta_{0} \in \Theta$, 
$\varepsilon_{0},\dots,\varepsilon_{1-N} \in \mathbb{R}$ 
and 
$\psi_{0},\dots,\psi_{1-N} \in \mathbb{R}^{d_{\theta } }$ 
are random variables defined on the same probability space. 
$\theta_{0}, \varepsilon_{0},\dots,\varepsilon_{1-N}, 
\psi_{0},\dots,\psi_{1-N} \in \mathbb{R}^{d_{\theta } }$ 
represent the initial conditions of the algorithm
(\ref{6.1}) -- (\ref{6.7}). 

In the literature on system identification, 
recursion (\ref{6.1}) -- (\ref{6.7}) is known as 
the recursive prediction error algorithm for ARMA models
(for more details 
\cite{ljung2}, \cite{ljung3} 
and references cited therein). 
It usually involves a projection (or truncation) device which ensures
that estimates 
$\{\theta_{n} \}_{n\geq 0}$ remain in $\Theta$. 
However, in order to avoid unnecessary technical details and to 
keep the exposition as concise as possible, 
this aspect of algorithm (\ref{6.1}) -- (\ref{6.7}) is not discussed here. 
Instead, similarly as in \cite{ljung1} -- \cite{ljung3}, 
we state our asymptotic results 
(Theorem \ref{theorem6.2})
in a local form. 

Algorithm (\ref{6.1}) -- (\ref{6.7}) is analyzed under the following 
assumptions: 

\begin{assumption} \label{a6.1} 
There exist a positive integer $L$, 
a matrix $A \in \mathbb{R}^{L \times L}$, 
a vector $b\in \mathbb{R}^{L}$
and $\mathbb{R}^{L}$-valued stochastic processes 
$\{x_{n} \}_{n> -M}$, $\{w_{n} \}_{n> -M}$
(defined on $(\Omega, {\cal F}, P )$) 
such that 
the following holds: 
\begin{romannum}
\item
$x_{n+1} = A x_{n} + w_{n}$
and 
$y_{n} = b^{T} x_{n}$ 
for $n> -M$. 
\item
The eigenvalues of $A$ lie in 
$\{z\in \mathbb{C}: |z| < 1 \}$. 
\item
$\{w_{n} \}_{n\geq -M}$ are i.i.d. and 
independent of 
$\theta_{0}$, $x_{1-M}$, 
$\varepsilon_{0}, \dots, \varepsilon_{1-N}$, 
$\psi_{0}, \dots, \psi_{1-N}$. 
\item
$E\|w_{0} \|^{4} < \infty$. 
\end{romannum}
\end{assumption}

\begin{assumption} \label{a6.2}
For any compact set $Q \subset \Theta$, 
\begin{align} \label{6.9}
	\sup_{n\geq 0} 
	E\left(
	(\varepsilon_{n}^{4} + \|\psi_{n} \|^{4} )
	I_{ \{\tau_{Q} \geq n \} }
	\right)
	< \infty,  
\end{align}
where $\tau_{Q} = \inf\{n\geq 0: \theta_{n} \notin Q \}$. 
\end{assumption}

Our main result on the analyticity of $f(\cdot )$
is contained in the next theorem. 

\begin{theorem} \label{theorem6.1}
Suppose that $\{y_{n} \}_{n\geq 0}$
is a weakly stationary process such that 
\begin{align*}
	\sum_{n=0}^{\infty } 
	|{\rm Cov}(y_{0}, y_{n} ) |
	< \infty. 
\end{align*}
Then, $f(\cdot )$ is analytic on entire $\Theta$, 
i.e., 
the following is true: 
For any compact set $Q \subset \Theta$ and 
any $a\in f(Q)$, 
there exist real numbers 
$\delta_{Q,a}$,   
$\mu_{Q,a} \in (1,2]$, $\nu_{Q} \in (0,1]$, $M_{Q,a} \in [1,\infty )$, $N_{Q}$
such that (\ref{a1.4.1}) holds for all $\theta \in Q$
and such that (\ref{a1.3.1}) is satisfied for 
each $\theta \in Q$
fulfilling $|f(\theta) - a | \leq \delta_{Q,a}$.  
\end{theorem}

In order to state our main result of the convergence rate of 
algorithm (\ref{6.1}) -- (\ref{6.7}), 
we use the following notation. 
$\Lambda$ is the event defined by
\begin{align*}
	\Lambda 
	= 
	\left\{
	\sup_{n\geq 0} \|\theta_{n} \| < \infty, 
	\inf_{n\geq 0} d(\theta_{n}, \partial \Theta ) > 0
	\right\}. 
\end{align*}
$\hat{A}$ is the set of accumulation points of 
$\{\theta_{n} \}_{n\geq 0}$, 
while 
\begin{align*}
	\hat{\rho}
	=
	2^{-1} d(\hat{A},\partial \Theta ) \: I_{\Lambda }, 
	\;\;\;\;\; 
	\hat{f}
	=
	\liminf_{n\rightarrow \infty } f(\theta_{n} ).  
\end{align*}
$\hat{Q}$ is the random set defined as
\begin{align*}
	\hat{Q}
	=
	\begin{cases}
	\left\{
	\theta\in \mathbb{R}^{d_{\theta } }: d(\theta, \hat{A} ) \hat{\rho} 
	\right\}, 
	&\text{ on } \Lambda
	\\
	\hat{A}, 
	&\text{ otherwise }
	\end{cases}. 
\end{align*}
$\hat{\delta}$, $\hat{\mu}$, $\hat{\nu}$ are random quantities 
defined by (\ref{1.21}) on $\Lambda$
and by (\ref{1.23}) otherwise. 
Random quantities $\hat{p}$, $\hat{q}$, $\hat{r}$
are defined by (\ref{1.25}). 
With this notation, 
our main result on the convergence rate of 
algorithm (\ref{6.1}) -- (\ref{6.7}) reads as follows. 

\begin{theorem} \label{theorem6.2}
Let Assumptions \ref{a2.1}, \ref{a6.1} and \ref{a6.2} hold. 
Then, 
\begin{align*}
	\|\nabla f(\theta_{n} ) \|^{2} = o(\gamma_{n}^{-p} ), 
	\;\;\;\;\;
	d(f(\theta_{n} ), C ) = o(\gamma_{n}^{-p} )
\end{align*}	
w.p.1 on 
$\Lambda$. 
Moreover, 
the following is true: 
\begin{align*}
	\|\nabla f(\theta_{n} ) \|^{2}
	=
	o\big(\gamma_{n}^{-\hat{p} } \big),
	\;\;\; 
	d(f(\theta_{n} ), C )
	=
	o\big(\gamma_{n}^{-\hat{p} } \big),
	\;\;\; 
	d(\theta_{n}, S )
	=
	o\big(\gamma_{n}^{-\hat{q} } \big)
\end{align*}
w.p.1 
on 
$\Lambda 
\cap \{\hat{r} > r \}$, and  
\begin{align*}
	\|\nabla f(\theta_{n} ) \|^{2}
	=
	O\big(\gamma_{n}^{-\hat{p} } \big),
	\;\;\; 
	d(f(\theta_{n} ), C )
	=
	O\big(\gamma_{n}^{-\hat{p} } \big), 
	\;\;\; 
	d(\theta_{n}, S )
	=
	O\big(\gamma_{n}^{-\hat{q} } \big)
\end{align*}
w.p.1 on 
$\Lambda 
\cap \{\hat{r} \leq r \}$. 
\end{theorem}

The proofs are provided in Section \ref{section6*}. 
$C$ and $S$
are defined in Section \ref{section1}. 

Assumption \ref{a6.1} corresponds to the signal 
$\{y_{n} \}_{n\geq 0}$. 
It is quite common for the asymptotic analysis of 
recursive identification algorithm (see e.g., \cite[Part I]{benveniste}) 
and cover all stable linear Markov models. 
Assumption \ref{a6.2} is related to the stability of 
subrecursion (\ref{6.1}) -- (\ref{6.5}) 
and its output 
$\{\varepsilon_{n} \}_{\geq 0}$, $\{\psi_{n} \}_{n\geq 0}$. 
In this or a similar form, 
Assumption \ref{a6.2} is involved in most of 
the asymptotic results on the recursive prediction error 
identification algorithms. 
E.g., 
\cite[Theorems 4.1 -- 4.3]{ljung2} 
(which are probably the most general and famous results of this kind)  
require sequence 
$\{(\varepsilon_{n}, \psi_{n} ) \}_{n\geq 0}$
to visit a fixed compact set infinitely often w.p.1 
on event $\Lambda$. 
When $\{y_{n} \}_{n\geq 0}$ is generated by a stable linear 
Markov system, such a requirement is practically equivalent to (\ref{6.9}). 

Various aspects of recursive prediction error identification in 
linear stochastic dynamical systems have been the subject 
of numerous papers and books
(see \cite{ljung2}, \cite{ljung3} 
and references cited therein). 
Despite providing a deep insight into the asymptotic behavior of 
recursive prediction error identification algorithms, 
the available results do not offer information about the convergence rate
which 
can be verified for models of a moderate or high order
(e.g., $M$ and $N$ are three or above). 
The main difficulty is the same as in the case of 
supervised learning. 
The existing results on convergence rate of stochastic gradient search 
require 
$f(\cdot )$ to have an isolated minimum
which is the limit of $\{\theta_{n} \}_{n\geq 0}$
and at which $\nabla^{2} f(\cdot )$ is strictly positive definite. 
Unfortunately, $f(\cdot )$ is so complex (even for relatively 
small $M$ and $N$) that these requirements are practically 
impossible to verify. 
Apparently, Theorem \ref{theorem6.2} relies on none of them. 

Regarding Theorems \ref{theorem6.1} and \ref{theorem6.2}, 
it should be mentioned that 
these results can be generalized in several ways. 
E.g., it is straightforward to extend them to 
practically any stable multiple-input, multiple-output 
linear system. 
Moreover, it is possible to show that the results 
also hold for signals $\{y_{n} \}_{n\geq 0}$
satisfying mixing conditions of the type 
\cite[Condition S1, p. 169]{ljung2}. 

\section{Proof of Theorems \ref{theorem1.1} and \ref{theorem1.2}} \label{section1*}

In this section, the following notation is used. 
Let $\Lambda$ be the event 
\begin{align*}
	\Lambda = \left\{\sup_{n\geq 0} \|\theta_{n} \| < \infty \right\}. 
\end{align*}	
For $\varepsilon \in (0, \infty )$, let 
\begin{align*}
	\phi_{\varepsilon }(w) = \phi(w) + \varepsilon. 
\end{align*}
For 
$\theta \in \mathbb{R}^{d_{\theta } }$, let 
\begin{align*}
	u(\theta ) 
	=
	f(\theta ) - \hat{f}, 
	\;\;\;\;\;  
	v(\theta ) 
	=
	\begin{cases} 
	(f(\theta ) - \hat{f} )^{-1/\hat{p} }, 
	&\text{ if } f(\theta ) > \hat{f}
	\\
	0, 
	&\text{otherwise} 
	\end{cases}  
\end{align*}
($\hat{p}$ is introduced in Section \ref{section1}). 
On the other side, 
for $0 \leq n <k$, let 
$u_{n,n}=0$, 
$v_{n,n}=v'_{n,n}=v''_{n,n}=0$ and 
\begin{align*}	
	&
	u_{n,k}
	=
	\sum_{i=n}^{k-1} \alpha_{i} w_{i}, 
	\\	
	&
	v'_{n,k}
	=
	- (\nabla f(\theta_{n} ) )^{T} 
	\sum_{i=1}^{k-1} \alpha_{i} (\nabla f(\theta_{i} ) - \nabla f(\theta_{n} ) ), 
	\\
	&
	v''_{n,k} 
	=
	\int_{0}^{1} 
	(\nabla f(\theta_{n} + s (\theta_{k} - \theta_{n} ) ) - \nabla f(\theta_{n} ) )^{T}
	(\theta_{k} - \theta_{n} ) ds, 
	\\
	&
	v_{n,k} = v'_{n,k} + v''_{n,k}.
\end{align*}
Then, it is straightforward to show 
\begin{align} \label{1.1*}
	f(\theta_{k} ) - f(\theta_{n} )
	= &
	-
	(\gamma_{k} - \gamma_{n} ) 
	\|\nabla f(\theta_{n} ) \|^{2} 
	-
	(\nabla f(\theta_{n} ) )^{T} u_{n,k}  	
	+
	v_{n,k}			
\end{align}
for $0 \leq n \leq k$. 

Regarding the notation, the following note is also in order: 
$\:\tilde{}\:$ symbol is used for locally defined quantities, i.e., 
for a quantity whose definition holds only in the proof where such a quantity 
appears. 

\begin{lemma} \label{lemma1.1}
Let Assumptions \ref{a1.1} and \ref{a1.2} hold. 
Then, there exists an event $N_{0} \in {\cal F}$
such that 
$P(N_{0} ) = 0$ and 
\begin{align} \label{l1.1.1*}
	\limsup_{n\rightarrow \infty } \gamma_{n}^{r} 
	\max_{n\leq k \leq a(n,1) }
	\|u_{n,k} \|
	\leq 
	w
	< 
	\infty 
\end{align}
on $\Lambda\setminus N_{0}$. 
\end{lemma}

\begin{proof}
It is straightforward to verify 
\begin{align*}
	u_{n,k}
	=
	\sum_{i=n}^{k-1} 
	(\gamma_{i}^{-r} - \gamma_{i+1}^{-r} ) 
	\left(
	\sum_{j=n}^{i} \alpha_{j} \gamma_{j}^{r} w_{j} 
	\right)
	+
	\gamma_{k}^{-r} \sum_{i=n}^{k-1} \alpha_{i} \gamma_{i}^{r} w_{i} 
\end{align*}
for $0\leq n < k$. 
Consequently, 
\begin{align*}
	\|u_{n,k} \|
	\leq &
	\left(
	\gamma_{k}^{-r} 
	+ 
	\sum_{i=n}^{k-1} (\gamma_{i}^{-r} - \gamma_{i+1}^{-r} )
	\right)
	\max_{n\leq j < k } 
	\left\|
	\sum_{i=n}^{j} \alpha_{i} \gamma_{i}^{r} w_{i} 
	\right\|
	=
	\gamma_{n}^{-r} 
	\max_{n\leq j < k } 
	\left\|
	\sum_{i=n}^{j} \alpha_{i} \gamma_{i}^{r} w_{i} 
	\right\|
\end{align*}
for $0 \leq n < k$. 
Thus, 
\begin{align*}
	\gamma_{n}^{r} 
	\|u_{n,k} \|
	\leq &
	\max_{n\leq j < a(n,1) } 
	\left\|
	\sum_{i=n}^{j} \alpha_{i} \gamma_{i}^{r} w_{i} 
	\right\|
\end{align*}
for $0 \leq n \leq k \leq a(n,1)$. 
Then, the lemma's assertion directly follows from Assumption \ref{a1.2}. 
\end{proof} 

\begin{lemma} \label{lemma1.2}
Suppose that Assumptions \ref{a1.1} -- \ref{a1.3} hold. 
Moreover, let $\varepsilon \in (0,\infty )$
be an arbitrary positive real number. 
Then, 
there exist random quantities  
$\hat{C}_{1}$, $\hat{t}$
(which are deterministic functions of 
$\hat{C}$; $\hat{C}$ is defined in 
Section \ref{section1}) 
and a non-negative integer-valued 
random variable 
$\sigma_{\varepsilon}$ such that 
$1\leq \hat{C} < \infty$, 
$0 < \hat{t} \leq 1$, 
$0\leq \sigma_{\varepsilon } < \infty$
everywhere and 
such that 
\begin{align}
	& \label{l1.2.1*} 	
	\begin{aligned}[b]
		\max_{n\leq k \leq a(n,\hat{t} ) }
		(f(\theta_{k} ) - f(\theta_{n} ) )
		\leq &	
		\gamma_{n}^{-\hat{p}/\hat{\mu} } 
		\|\nabla f(\theta_{n} ) \| \phi_{\varepsilon }(w) 	
		+ 
		\hat{C}_{1} 
		\gamma_{n}^{-2\hat{p}/\hat{\mu} } 
		(\phi_{\varepsilon }(w) )^{2}, 
	\end{aligned}
	\\
	& \label{l1.2.3*}
	\begin{aligned}[b]
		f(\theta_{a(n,\hat{t} ) } ) - f(\theta_{n} ) 
		\leq &
		-
		\hat{t} \|\nabla f(\theta_{n} ) \|^{2}/2
		+
		\gamma_{n}^{-\hat{p}/\hat{\mu} } 
		\|\nabla f(\theta_{n} ) \| \phi_{\varepsilon }(w) 
		+ 
		\hat{C}_{1} 
		\gamma_{n}^{-2\hat{p}/\hat{\mu} } 
		(\phi_{\varepsilon }(w) )^{2} 
	\end{aligned} 	
\end{align}
on $\Lambda\setminus N_{0}$ 
for  
$n>\sigma_{\varepsilon }$
($\hat{\mu}$ is introduced in Section \ref{section2}). 
\end{lemma}

\begin{proof}
Let 
$\hat{C}_{1} = 12 \hat{C}^{3} \exp(2\hat{C} )$, 
$\hat{t} = 1/(4 \hat{C}_{1} )$, 
while  
\begin{align*}
	&
	\tilde{\sigma}_{1}
	=
	\max\left(
	\{n\geq 0: \theta_{n} \not\in \hat{Q} \} 
	\cup \{0\}
	\right),
	\\
	&
	\tilde{\sigma}_{2}
	=
	\max\left(
	\left\{n\geq 0: \alpha_{n}>\hat{t}/3 \right\}
	\cup \{0\}
	\right), 
	\\
 	&
	\tilde{\sigma}_{3,\varepsilon}
	=
	\max\left(
	\left\{
	n\geq 0: 
	\max_{n\leq k \leq a(n,1) } \|u_{n,k} \| 
	> 
	\gamma_{n}^{-\hat{p}/\hat{\mu} }
	\phi_{\varepsilon}(w) 
	\right\}
	\cup 
	\{0\}
	\right)
\end{align*}
and 
$\sigma_{\varepsilon } = 
\max\{\tilde{\sigma}_{1}, \tilde{\sigma}_{2}, 
\tilde{\sigma}_{3,\varepsilon } \}
I_{\Lambda\setminus N_{0} }$.  
Then, it is obvious that 
$\sigma_{\varepsilon }$ is well-defined. 
On the other side, Lemma \ref{lemma1.1} yields
\begin{align*}
	\limsup_{n\rightarrow \infty }
	\gamma_{n}^{\hat{p}/\hat{\mu} }
	\max_{n\leq k \leq a(n,1) }
	\|u_{n,k} \|
	=
	\limsup_{n\rightarrow \infty }
	\gamma_{n}^{r}
	\max_{n\leq k \leq a(n,1) }
	\|u_{n,k} \|
	=
	w
	<
	\phi_{\varepsilon}(w)
\end{align*}
on $(\Lambda\setminus N_{0} ) \cap \{\hat{r}\geq r \}$
(notice that if $r\leq \hat{r}$, 
then $\hat{p}/\hat{\mu} = r$
and 
$\phi_{\varepsilon}(w) \geq w+\varepsilon > w$)
and
\begin{align*}
	\limsup_{n\rightarrow \infty }
	\gamma_{n}^{\hat{p}/\hat{\mu} }
	\max_{n\leq k \leq a(n,1) }
	\|u_{n,k} \|
	=
	\limsup_{n\rightarrow \infty }
	\gamma_{n}^{\hat{p}/\hat{\mu} - r}
	w
	=
	0
	<
	\phi_{\varepsilon}(w)
\end{align*}
on $(\Lambda\setminus N_{0} ) \cap \{\hat{r}< r \}$
(notice that if $r> \hat{r}$, 
then $\hat{p}/\hat{\mu} = \hat{r} < r$
and 
$\phi_{\varepsilon}(w) \geq \varepsilon > 0$). 
Thus, 
$0\leq \sigma_{\varepsilon} < \infty$
everywhere. 
Moreover, we have 
\begin{align} 
	& \label{l1.2.1}
	\max_{n\leq k \leq a(n,1) } 
	\|u_{n,k} \|
	\leq 
	\gamma_{n}^{-\hat{p}/\hat{\mu} } \phi_{\varepsilon }(w),  
	\\
	& \label{l1.2.3}
	\hat{t}
	\geq
	\gamma_{a(n,\hat{t} )} - \gamma_{n} 
	=
	\gamma_{a(n,\hat{t} )+1} - \gamma_{n} - \alpha_{a(n,\hat{t} )} 
	\geq 
	2\hat{t}/3
\end{align}
on $\Lambda \setminus N_{0}$ for $n>\sigma_{\varepsilon }$. 
On the other side, (\ref{l1.2.1}) yields 
\begin{align*}
	\|\nabla f(\theta_{k} ) \| 
	\leq &
	\|\nabla f(\theta_{n} ) \| 
	+
	\|\nabla f(\theta_{k} ) - \nabla f(\theta_{n} ) \|
	\\
	\leq &
	\|\nabla f(\theta_{n} ) \| 
	+
	\hat{C} \|\theta_{k} - \theta_{n} \|
	\\
	\leq &
	\|\nabla f(\theta_{n} ) \| 
	+
	\hat{C} 
	\sum_{i=n}^{k-1} \alpha_{i} \|\nabla f(\theta_{i} ) \| 
	+
	\hat{C} \|u_{n,k} \|
	\\
	\leq &
	\|\nabla f(\theta_{n} ) \| 
	+
	\hat{C} \gamma_{n}^{-\hat{p}/\hat{\mu} } \phi_{\varepsilon }(w) 
	+
	\hat{C} 
	\sum_{i=n}^{k-1} \alpha_{i} \|\nabla f(\theta_{i} ) \| 
\end{align*}
on $\Lambda$ for $\sigma_{\varepsilon } < n \leq k$.  
Then, Bellman-Gronwall inequality implies 
\begin{align*}
	\|\nabla f(\theta_{k} ) \|	
	\leq &
	\left(
	\|\nabla f(\theta_{n} ) \|
	+
	\hat{C} \gamma_{n}^{-\hat{p}/\hat{\mu} } \phi_{\varepsilon }(w) 
	\right)
	\exp\left(
	\hat{C} (\gamma_{a(n,1) } - \gamma_{n} )
	\right) 
	\\
	\leq & 	
	\hat{C} \exp(\hat{C} )
	\left(
	\|\nabla f(\theta_{n} ) \|
	+
	\gamma_{n}^{-\hat{p}/\hat{\mu} } \phi_{\varepsilon }(w) 
	\right)
\end{align*}
on $\Lambda\setminus N_{0}$ for $\sigma_{\varepsilon } < n \leq k \leq a(n,1)$
(notice that $\gamma_{a(n,1)} - \gamma_{n} \leq 1$). 
Consequently, (\ref{l1.2.1}) gives 
\begin{align*}
	\|\theta_{k} - \theta_{n} \|
	\leq &
	\sum_{i=n}^{k-1} 
	\alpha_{i} \|\nabla f(\theta_{i} ) \| 
	+
	\|u_{n,k} \|
	\\
	\leq &
	\hat{C} \exp(\hat{C} )
	\left(
	\|\nabla f(\theta_{n} ) \|
	+
	\gamma_{n}^{-\hat{p}/\hat{\mu} } \phi_{\varepsilon }(w) 
	\right)
	(\gamma_{k} - \gamma_{n} )
	+ 
	\gamma_{n}^{-\hat{p}/\hat{\mu} } \phi_{\varepsilon }(w) 
	\\
	\leq &
	2\hat{C} \exp(\hat{C} ) 
	\left(
	(\gamma_{k} - \gamma_{n} )	\|\nabla f(\theta_{n} ) \|
	+ 
	\gamma_{n}^{-\hat{p}/\hat{\mu} } \phi_{\varepsilon }(w) 
	\right)
\end{align*}
on $\Lambda\setminus N_{0}$ for $\sigma_{\varepsilon } < n \leq k \leq a(n,1)$.
Therefore, 
\begin{align*}	
	&
	\begin{aligned}[b]
		|v'_{n,k} |
		\leq &
		\hat{C} 
		\|\nabla f(\theta_{n} ) \|
		\sum_{i=n}^{k-1} \alpha_{i} \|\theta_{i} - \theta_{n} \|	
		\\
		\leq &
		2\hat{C}^{2} \exp(\hat{C} ) 
		\left(
		(\gamma_{k} - \gamma_{n} )^{2} \|\nabla f(\theta_{n} ) \|^{2}
		+ 
		\gamma_{n}^{-\hat{p}/\hat{\mu} } (\gamma_{k} - \gamma_{n} )
		\|\nabla f(\theta_{n} ) \| 	
		\phi_{\varepsilon }(w)
		\right)
		\\
		\leq &
		4\hat{C}^{2} \exp(\hat{C} ) 
		\left(
		(\gamma_{k} - \gamma_{n} )^{2} \|\nabla f(\theta_{n} ) \|^{2}
		+ 
		\gamma_{n}^{-2\hat{p}/\hat{\mu} } (\phi_{\varepsilon }(w) )^{2}
		\right),
	\end{aligned}
	\\	
	&
	\begin{aligned}[b]
		|v''_{n,k} |
		\leq &
		\hat{C} \|\theta_{k} - \theta_{n} \|^{2} 
		\\
		\leq &
		4\hat{C}^{3} \exp(2\hat{C} ) 
		\left(
		(\gamma_{k} - \gamma_{n} ) \|\nabla f(\theta_{n} ) \|
		+ 
		\gamma_{n}^{-\hat{p}/\hat{\mu} } \phi_{\varepsilon }(w)
		\right)^{2}
		\\
		\leq &
		8\hat{C}^{3} \exp(2\hat{C} ) 
		\left(
		(\gamma_{k} - \gamma_{n} )^{2}	\|\nabla f(\theta_{n} ) \|^{2}
		+ 
		\gamma_{n}^{-2\hat{p}/\hat{\mu} } (\phi_{\varepsilon }(w) )^{2}
		\right)
	\end{aligned}
\end{align*}
on $\Lambda\setminus N_{0}$ for $\sigma_{\varepsilon } < n \leq k \leq a(n,1)$. 
Thus, 
\begin{align} 
	&\label{l1.2.5}
	|v_{n,k} |
	\leq 
	\hat{C}_{1} 
	\left(
	(\gamma_{k} - \gamma_{n} )^{2}	\|\nabla f(\theta_{n} ) \|^{2}
	+ 
	\gamma_{n}^{-2\hat{p}/\hat{\mu} } (\phi_{\varepsilon }(w) )^{2} 
	\right)
\end{align}
on $\Lambda\setminus N_{0}$ for $\sigma_{\varepsilon } < n \leq k \leq a(n,1)$. 
Since 
\begin{align*}
	\hat{C}_{1} (\gamma_{k}-\gamma_{n} )	
	\leq 
	\hat{C}_{1} (\gamma_{a(n,\hat{t} ) }-\gamma_{n} )	
	\leq 
	\hat{C}_{1} \hat{t} \leq 1/4
\end{align*}
for $0\leq n \leq k \leq a(n,\hat{t} )$
(due to (\ref{l1.2.3})), 
(\ref{1.1*}), (\ref{l1.2.1}) and (\ref{l1.2.5}) yield 
\begin{align} \label{l1.2.7}
	f(\theta_{k} ) - f(\theta_{n} ) 
	\leq &
	-
	(\gamma_{k} - \gamma_{n} )
	\left(
	1
	-
	\hat{C}_{1} (\gamma_{k} - \gamma_{n} )
	\right)
	\|\nabla f(\theta_{n} ) \|^{2}
	\nonumber \\
	&
	+
	\gamma_{n}^{-\hat{p}/\hat{\mu} } 
	\|\nabla f(\theta_{n} ) \| \phi_{\varepsilon }(w) 		
	+ 
	\hat{C}_{1} 
	\gamma_{n}^{-2\hat{p}/\hat{\mu} } (\phi_{\varepsilon }(w) )^{2} 
	\nonumber \\
	\leq
	&
	-
	3(\gamma_{k} - \gamma_{n} )
	\|\nabla f(\theta_{n} ) \|^{2}/4 
	\nonumber \\
	&
	+
	\gamma_{n}^{-\hat{p}/\hat{\mu} } 
	\|\nabla f(\theta_{n} ) \| \phi_{\varepsilon }(w) 	
	+ 
	\hat{C}_{1} 
	\gamma_{n}^{-2\hat{p}/\hat{\mu} } (\phi_{\varepsilon }(w) )^{2} 
\end{align}
on $\Lambda\setminus N_{0}$ for 
$\sigma_{\varepsilon } < n \leq k \leq a(n,\hat{t} )$. 
As an immediate consequence of (\ref{l1.2.3}), (\ref{l1.2.7}), 
we get that 
(\ref{l1.2.1*}), (\ref{l1.2.3*})  
hold on $\Lambda\setminus N_{0}$ for $n > \sigma_{\varepsilon }$.  
\end{proof}

\begin{lemma} \label{lemma1.3}
Suppose that Assumptions \ref{a1.1} -- \ref{a1.3} hold. 
Then, 
$\lim_{n\rightarrow \infty } \nabla f(\theta_{n} ) = 0$ 
on $\Lambda \setminus N_{0}$. 
\end{lemma}

\begin{proof}
The lemma's assertion is proved by contradiction. 
We assume that 
$\limsup_{n\rightarrow \infty } \|\nabla f(\theta_{n} ) \| > 0$
for some sample 
$\omega \in \Lambda\setminus N_{0}$ 
(notice that all formulas which follow in the proof correspond to 
this $\omega$). 
Then, there exists $a\in (0,\infty )$ and an 
increasing sequence $\{l_{k} \}_{k\geq 0}$
such that 
$\liminf_{k\rightarrow \infty } \|\nabla f(\theta_{l_{k} } ) \| > a$. 
Since $\liminf_{k\rightarrow \infty } f(\theta_{a(l_{k},\hat{t} ) } ) \geq \hat{f}$, 
Lemma \ref{lemma1.2} (inequality (\ref{l1.2.3*})) gives 
\begin{align*}
	\hat{f} 
	-
	\liminf_{k\rightarrow \infty } f(\theta_{l_{k} } ) 
	\leq &
	\limsup_{k\rightarrow \infty } 
	(f(\theta_{a(l_{k},\hat{t} ) } ) - f(\theta_{l_{k} } ) )
	\nonumber\\
	\leq &
	-
	(\hat{t}/2 ) 
	\liminf_{k\rightarrow \infty } 
	\|\nabla f(\theta_{l_{k} } ) \|^{2}
	\nonumber\\
	\leq &
	- a^{2} \hat{t}/2.   
\end{align*}
Therefore, 
$\liminf_{k\rightarrow \infty } f(\theta_{l_{k} } ) \geq \hat{f} + a \hat{t}^{2}/2$. 
Consequently, 
there exist $b,c \in \mathbb{R}$ such that 
$\hat{f} < b < c < \hat{f} + a \hat{t}^{2}/2$, 
$b < \hat{f} + \hat{\delta}$
and $\limsup_{n\rightarrow \infty } f(\theta_{n} ) > c$. 
Thus, there exist sequences 
$\{m_{k} \}_{k\geq 0}$, $\{n_{k} \}_{k\geq 0}$
with the following properties: 
$m_{k} < n_{k} < m_{k+1}$, 
$f(\theta_{m_{k} } ) < b$, 
$f(\theta_{n_{k} } ) > c$ and 
\begin{align} \label{l1.3.5}
	\max_{m_{k} < n \leq n_{k} } f(\theta_{n} ) 
	\geq 
	b
\end{align}
for $k\geq 0$. 
Then, Lemma \ref{lemma1.2} (inequality (\ref{l1.2.1*})) implies 
\begin{align}
	&\label{l1.3.7}
	\limsup_{k\rightarrow \infty } 
	(f(\theta_{m_{k} + 1 } ) - f(\theta_{m_{k} } ) )
	\leq 
	0, 
	\\
	&\label{l1.3.9} 
	\limsup_{k\rightarrow \infty } 
	\max_{m_{k} \leq n \leq a(m_{k}, \hat{t} ) }
	(f(\theta_{n} ) - f(\theta_{m_{k} } ) )
	\leq
	0. 
\end{align}
Since 
\begin{align*}
	b
	>
	f(\theta_{m_{k} } ) 
	=
	f(\theta_{m_{k} + 1 } ) 
	-
	(f(\theta_{m_{k} + 1 } ) - f(\theta_{m_{k} } ) )
	\geq 
	b 
	-
	(f(\theta_{m_{k} + 1 } ) - f(\theta_{m_{k} } ) )
\end{align*}
for $k\geq 0$, 
(\ref{l1.3.7}) yields 
$\lim_{k\rightarrow \infty } f(\theta_{m_{k} } ) = b$. 
As
$f(\theta_{n_{k} } ) - f(\theta_{m_{k} } ) > c-b$
for $k\geq 0$, 
(\ref{l1.3.9}) implies 
$a(m_{k},\hat{t} ) < n_{k}$ for all, 
but infinitely many $k$
(otherwise, 
$\liminf_{k\rightarrow \infty } (f(\theta_{n_{k} } ) - f(\theta_{m_{k} } ) ) \leq 0$
would follow from (\ref{l1.3.9})). 
Consequently,
$\liminf_{k\rightarrow \infty } f(\theta_{a(m_{k},\hat{t} ) } ) ) \geq b$
(due to (\ref{l1.3.5})), 
while Lemma \ref{lemma1.2} 
(inequality (\ref{l1.2.3*})) gives 
\begin{align*}	
	0
	\leq
	\limsup_{k\rightarrow \infty } 
	f(\theta_{a(m_{k},\hat{t} )} ) 
	- 
	b 
	= &
	\limsup_{k\rightarrow \infty } 
	(f(\theta_{a(m_{k},\hat{t} )} ) - f(\theta_{m_{k} } ) )
	\\
	\leq &
	- 
	(\hat{t}/2)
	\liminf_{k\rightarrow \infty } 
	\|\nabla f(\theta_{m_{k} } ) \|^{2}.   
\end{align*}
Therefore, 
$\lim_{k\rightarrow \infty } \|\nabla f(\theta_{m_{k} } ) \| = 0$. 
Thus, 
there exists 
$k_{0} \geq 0$ such that 
$\theta_{m'_{k} } \in \hat{Q}$
and 
$f(\theta_{m_{k} } ) \geq (\hat{f} + b ) / 2$
for $k\geq k_{0}$
(notice that 
$\lim_{k\rightarrow \infty } f(\theta_{m_{k} } ) = b > (\hat{f} + b )/2$). 
Consequently, 
$\theta_{m_{k} } \in \hat{Q}$
and 
$0 < (b - \hat{f} )/2 \leq f(\theta_{m_{k} } ) - \hat{f} \leq \hat{\delta}$
for $k\geq k_{0}$
(notice that $f(\theta_{m_{k} } ) < b < \hat{f} + \hat{\delta}$
for $k\geq 0$). 
Then, owing to (\ref{1.1'}) (i.e., to Assumption \ref{a2.3}), 
we have 
\begin{align*} 
	0
	<
	(b - \hat{f} )/2
	\leq 
	f(\theta_{m_{k} } ) - \hat{f} 
	\leq 
	\hat{M} 
	\|\nabla f(\theta_{m_{k} } ) \|^{\hat{\mu} }
\end{align*}
for $k\geq k_{0}$. 
However, this directly contradicts the fact
$\lim_{k\rightarrow \infty } \|\nabla f(\theta_{m_{k} } ) \| = 0$. 
Hence, 
$\lim_{n\rightarrow \infty } \nabla f(\theta_{n} ) = 0$
on $\Lambda \setminus N_{0}$. 
\end{proof}

\begin{lemma} \label{lemma1.3'}
Suppose that Assumptions \ref{a1.1} -- \ref{a1.3} hold. 
Then, $\lim_{n\rightarrow \infty } f(\theta_{n} ) = \hat{f}$
on $\Lambda\setminus N_{0}$. 
\end{lemma}

\begin{proof}
We use contradiction to prove the lemma's assertion: 
Suppose that 
$\hat{f} < \limsup_{n\rightarrow \infty } f(\theta_{n} )$
for some sample $\omega \in \Lambda\setminus N_{0}$
(notice that all formulas which follow in the proof correspond to this $\omega$). 
Then, there exists 
$a\in \mathbb{R}$ such that 
$\hat{f} < a < \hat{f} + \hat{\delta}$
and 
$\limsup_{n\rightarrow \infty } f (\theta_{n} ) > a$. 
Thus, there exists an increasing sequence 
$\{n_{k} \}_{k\geq 0}$ such that 
$f(\theta_{n_{k} } ) < a$ and 
$f(\theta_{n_{k}+1} ) \geq a$ for $k\geq 0$. 
On the other side, 
Lemma \ref{lemma1.2} (inequality (\ref{l1.2.1*})) implies 
\begin{align}\label{l1.3'.1}
	\limsup_{k\rightarrow \infty } 
	(f(\theta_{n_{k}+1 } ) - f(\theta_{n_{k} } ) )
	\leq 
	0. 
\end{align}
Since
\begin{align*}
	a
	>
	f(\theta_{n_{k} } )
	=
	f(\theta_{n_{k}+1 } )
	-
	(f(\theta_{n_{k}+1 } ) - f(\theta_{n_{k} } ) )
	\geq 
	a 
	-
	(f(\theta_{n_{k}+1 } ) - f(\theta_{n_{k} } ) )
\end{align*}
for $k\geq 0$, 
(\ref{l1.3'.1}) yields 
$\lim_{k\rightarrow \infty } f(\theta_{n_{k} } ) = a$. 
Consequently, 
there exists 
$k_{0} \geq 0$
such that 
$\theta_{n_{k} } \in \hat{Q}$
and 
$f(\theta_{n_{k} } ) \geq (\hat{f} + a ) /2$
for $k\geq k_{0}$
(notice that 
$\lim_{k\rightarrow \infty } f(\theta_{n_{k} } ) = a > (\hat{f} + a )/2$). 
Thus, 
$\theta_{n_{k} } \in \hat{Q}$
and 
$0 < (a - \hat{f} )/ 2 \leq f(\theta_{n_{k} } ) - \hat{f} \leq \hat{\delta}$
for $k\geq k_{0}$
(notice that $f(\theta_{n_{k} } ) < a < \hat{f} + \hat{\delta}$ for $k\geq 0$). 
Then, due to (\ref{1.1'})
(i.e., to Assumption \ref{a1.3}), we have 
\begin{align*}
	0
	<
	(a - \hat{f} )/2
	\leq 
	f(\theta_{n_{k} } ) - \hat{f} 
	\leq 
	\hat{M} \|\nabla f(\theta_{n_{k} } ) \|^{\hat{\mu} }
\end{align*}
for $k\geq k_{0}$. 
However, this directly contradicts the fact
$\lim_{n\rightarrow \infty } \nabla f(\theta_{n} ) = 0$. 
Hence, $\lim_{n\rightarrow \infty } f(\theta_{n} ) = \hat{f}$
on $\Lambda \setminus N_{0}$.
\end{proof}

\begin{lemma} \label{lemma1.4}
Suppose that Assumptions \ref{a1.1} -- \ref{a1.3} hold. 
Moreover, let $\varepsilon\in (0,\infty )$ be an arbitrary positive real number. 
Then, there exist random quantities  
$\hat{C}_{2}$, $\hat{C}_{3}$ 
(which are deterministic functions of $r$, $\hat{C}$, $\hat{M}$)
and 
a non-negative integer-valued random variable 
$\tau_{\varepsilon}$ such that
$1\leq \hat{C}_{2}, \hat{C}_{3} < \infty$, 
$0\leq \tau_{\varepsilon } < \infty$
everywhere and such that 
the following is true: 
\begin{align}
	&\label{l1.4.1*}
	\left(
	u(\theta_{a(n,\hat{t} ) } ) 
	- 
	u(\theta_{n} ) 
	+
	\hat{t} \|\nabla f(\theta_{n} ) \|^{2}/4
	\right)
	I_{A_{n,\varepsilon } }
	\leq 
	0, 
	\\
	&\label{l1.4.3*}
	\left(
	u(\theta_{a(n,\hat{t} ) } ) 
	- 
	u(\theta_{n} ) 
	+
	(\hat{t}/\hat{C}_{3} ) \: u(\theta_{n} ) 
	\right)
	I_{B_{n,\varepsilon } }
	\leq 
	0, 
	\\
	&\label{l1.4.5*}
	\left(
	v(\theta_{a(n,\hat{t} ) } ) 
	- 
	v(\theta_{n} ) 
	- 
	(\hat{t}/\hat{C}_{3} )
	(\phi_{\varepsilon }(w) )^{-\hat{\mu}/\hat{p} }
	\right)
	I_{C_{n,\varepsilon } }
	\geq 
	0
\end{align}
on $\Lambda\setminus N_{0}$ for $n\geq \tau_{\varepsilon}$, 
where 
\begin{align*}
	&
	\begin{aligned}[b]
		A_{n,\varepsilon}
		= &
		\left\{
		\gamma_{n}^{\hat{p} } 
		|u(\theta_{n} ) |
		\geq 
		\hat{C}_{2} 
		(\phi_{\varepsilon }(w) )^{\hat{\mu} }
		\right\}
		\cup
		\Big\{
		\gamma_{n}^{\hat{p} } 
		\|\nabla f(\theta_{n} ) \|^{2}
		\geq 
		\hat{C}_{2} 
		(\phi_{\varepsilon }(w) )^{\hat{\mu} } 
		\Big\}, 
	\end{aligned}
	\\
	&
	B_{n,\varepsilon}
	= 
	\left\{
	\gamma_{n}^{\hat{p} } 
	u(\theta_{n} ) 
	\geq 
	\hat{C}_{2} 
	(\phi_{\varepsilon }(w) )^{\hat{\mu} }
	\right\}
	\cap
	\{\hat{\mu} = 2 \}, 
	\\
	&
	C_{n,\varepsilon }
	=
	\left\{
	\gamma_{n}^{\hat{p} } u(\theta_{n} ) 
	\geq 
	\hat{C}_{2} 
	(\phi_{\varepsilon }(w) )^{\hat{\mu} }
	\right\}
	\cap
	\left\{
	u(\theta_{a(n,\hat{t} ) } ) > 0
	\right\}
	\cap
	\left\{
	\hat{\mu} < 2
	\right\}. 
\end{align*} 
\end{lemma}

\begin{remark}
Inequalities (\ref{l1.4.1*}) -- (\ref{l1.4.5*}) can be represented in 
the following equivalent form: 
Relations 
\begin{align}
	&\label{l1.4.21*}
	\left(
	\gamma_{n}^{\hat{p} } 
	|u(\theta_{n} ) |
	\geq 
	\hat{C}_{2} 
	(\phi_{\varepsilon }(w) )^{\hat{\mu} }
	\;\vee\;
	\gamma_{n}^{\hat{p} } 
	\|\nabla f(\theta_{n} ) \|^{2}
	\geq 
	\hat{C}_{2} 
	(\phi_{\varepsilon }(w) )^{\hat{\mu} } 
	\right)
	\; \wedge \; 
	n > \tau_{\varepsilon} 
	\nonumber\\
	&
	\Longrightarrow
	u(\theta_{a(n,\hat{t} ) } ) 
	\leq
	u(\theta_{n} ) 
	-
	\hat{t} \|\nabla f(\theta_{n} ) \|^{2}/4, 
	\\
	&\label{l1.4.23*}
	\gamma_{n}^{\hat{p} } 
	u(\theta_{n} ) 
	\geq 
	\hat{C}_{2} 
	(\phi_{\varepsilon }(w) )^{\hat{\mu} }
	\;\wedge\;
	\hat{\mu}=2
	\; \wedge \; 
	n > \tau_{\varepsilon} 
	\nonumber\\
	&
	\Longrightarrow
	u(\theta_{a(n,\hat{t} ) } ) 
	\leq 
	\left(
	1
	-
	\hat{t} /\hat{C}_{3} 
	\right)
	u(\theta_{n} ), 
	\\
	&\label{l1.4.25*}
	\gamma_{n}^{\hat{p} } u(\theta_{n} ) 
	\geq 
	\hat{C}_{2} 
	(\phi_{\varepsilon }(w) )^{\hat{\mu} }
	\;\wedge\;
	u(\theta_{a(n,\hat{t} ) } ) > 0
	\;\wedge\;
	\hat{\mu} < 2
	\; \wedge \; 
	n > \tau_{\varepsilon} 
	\nonumber\\
	&
	\Longrightarrow
	v(\theta_{a(n,\hat{t} ) } ) 
	\geq
	v(\theta_{n} ) 
	+
	(\hat{t}/\hat{C}_{3} )
	(\phi_{\varepsilon }(w) )^{-\hat{\mu}/\hat{p} }
\end{align}
are true on $\Lambda\setminus N_{0}$. 
\end{remark}

\begin{proof}
Let 
$\tilde{C} = 8\hat{C}_{1}^{1/2}/\hat{t}$,   
$\hat{C}_{2} = \tilde{C}^{2} \hat{M}$
and 
$\hat{C}_{3} = 8 \hat{M}^{2} \max\{1,r\}$, 
while 
\begin{align}
	&
	\tilde{\tau}_{1}
	=
	\max\left(\left\{
	n\geq 0: \theta_{n} \not\in \hat{Q}
	\right\}
	\cup \{0\}
	\right),
	\nonumber \\
	&
	\tilde{\tau}_{2}
	=
	\max\left(
	\left\{
	n\geq 0: 
	|u(\theta_{n} ) |
	> 
	\hat{\delta} 
	\right\}
	\cup \{0\}
	\right), 
	\nonumber \\
	&\label{l1.4.31}
	\tilde{\tau}_{3,\varepsilon}
	=
	\max\left(
	\left\{
	n\geq 0: 
	\gamma_{n}^{-\hat{p}/2} 
	(\phi_{\varepsilon }(w) )^{\hat{\mu}/2 }
	< 
	\gamma_{n}^{-\hat{p}/\hat{\mu} } \phi_{\varepsilon }(w)
	\right\}
	\cup \{0\}
	\right)
\end{align}
and
$\tau_{\varepsilon} =
\max\{\sigma_{\varepsilon}, \tilde{\tau}_{1}, \tilde{\tau}_{2}, 
\tilde{\tau}_{3,\varepsilon} \} 
I_{\Lambda \setminus N_{0} }$. 
Obviously,  
$\tau_{\varepsilon }$ is well-defined. 
On the other side, 
Lemmas \ref{lemma1.3}, \ref{lemma1.4}
imply 
$0\leq \tau_{\varepsilon } < \infty$
everywhere
(in order to conclude that 
$\tilde{\tau}_{2}$ is finite, 
notice that 
$\lim_{n\rightarrow \infty } u(\theta_{n} ) = 0$
on 
$\Lambda \setminus N_{0}$; 
in order to deduce that 
$\tilde{\tau}_{3,\varepsilon}$ is finite,
notice that  
$\hat{p}/2 < \hat{p}/\hat{\mu}$
when $\hat{\mu}<2$, 
and that the left and right hand sides of the inequality in (\ref{l1.4.31})
are equal when $\hat{\mu}=2$).  
Moreover, we have
\begin{align}\label{l1.4.1'}
	\gamma_{n}^{-\hat{p}/2} 
	(\phi_{\varepsilon }(w) )^{\hat{\mu}/2 }
	\geq 
	\gamma_{n}^{-\hat{p}/\hat{\mu} } \phi_{\varepsilon }(w)
\end{align}
on $\Lambda \setminus N_{0}$ for 
$n>\tau_{\varepsilon}$. 
Since 
$\tau_{\varepsilon } \geq \sigma_{\varepsilon }$
on $\Lambda \setminus N_{0}$, 
Lemma \ref{lemma1.2} (inequality (\ref{l1.2.3*}))
yields
\begin{align}\label{l1.4.3'}
		u(\theta_{a(n,\hat{t} ) } ) - u(\theta_{n} ) 
		\leq &
		-
		\hat{t} \|\nabla f(\theta_{n} ) \|^{2}/2
		+
		\gamma_{n}^{-\hat{p}/\hat{\mu} } 
		\|\nabla f(\theta_{n} ) \| \phi_{\varepsilon }(w) 
		+ 
		\hat{C}_{1} 
		\gamma_{n}^{-2\hat{p}/\hat{\mu} } 
		(\phi_{\varepsilon }(w) )^{2} 
\end{align}
on $\Lambda\setminus N_{0}$ for 
$n>\tau_{\varepsilon}$. 
As $\theta_{n} \in \hat{Q}$
and $|u(\theta_{n} ) | \leq \hat{\delta}$
on $\Lambda\setminus N_{0}$ for 
$n>\tau_{\varepsilon }$, 
(\ref{1.1'})
(i.e., Assumption \ref{a1.3}) implies
\begin{align}\label{l1.4.5'}
	|u(\theta_{n} ) |
	\leq 
	\hat{M} \|\nabla f(\theta_{n} ) \|^{\hat{\mu} }
\end{align}
on $\Lambda\setminus N_{0}$ for 
$n>\tau_{\varepsilon }$.

Let $\omega$ be an arbitrary sample from 
$\Lambda\setminus N_{0}$
(notice that all formulas which follow in the proof correspond to this 
$\omega$). 
First, we show (\ref{l1.4.1*}). 
We proceed by contradiction: 
Suppose that 
(\ref{l1.4.1*}) is violated for some 
$n>\tau_{\varepsilon}$. 
Therefore, 
\begin{align} \label{l1.4.1} 
	u(\theta_{a(n,\hat{t} ) } ) - u(\theta_{n} )
	>
	-\hat{t} \|\nabla f(\theta_{n} ) \|^{2}/4 
\end{align}
and at least one of the following two inequalities is true: 
\begin{align}
		&\label{l1.4.3}
		|u(\theta_{n} ) |
		\geq 
		\hat{C}_{2} \hat{M}
		\gamma_{n}^{-\hat{p} } 
		(\phi_{\varepsilon }(w) )^{\hat{\mu} },
		\\
		&\label{l1.4.5}
		\|\nabla f(\theta_{n} ) \|^{2}
		\geq 
		\hat{C}_{2} 
		\gamma_{n}^{-\hat{p} } 
		(\phi_{\varepsilon }(w) )^{\hat{\mu} }.  
\end{align}
If (\ref{l1.4.3}) holds, 
then (\ref{l1.4.5'}) implies 
\begin{align*}
	\|\nabla f(\theta_{n} ) \|
	\geq 
	(|u(\theta_{n} ) |/\hat{M} )^{1/\hat{\mu} }
	\geq 
	(\hat{C}_{2}/\hat{M} )^{1/\hat{\mu} } \gamma_{n}^{-\hat{p}/\hat{\mu} } 
	\phi_{\varepsilon }(w)
	\geq 
	\tilde{C} \gamma_{n}^{-\hat{p}/\hat{\mu} } \phi_{\varepsilon }(w) 
\end{align*}
(notice that 
$(\hat{C}_{2}/\hat{M} )^{1/\hat{\mu} } \geq (\hat{C}_{2}/\hat{M} )^{1/2} = \tilde{C}$ 
owing to $\hat{\mu}\leq 2$). 
On the other side, if (\ref{l1.4.5}) is satisfied, 
then (\ref{l1.4.1'}) yields 
\begin{align*}
	\|\nabla f(\theta_{n} ) \|
	\geq 
	\hat{C}_{2}^{1/2} 
	\gamma_{n}^{-\hat{p}/2} 
	(\phi_{\varepsilon }(w) )^{\hat{\mu}/2 }
	\geq 
	\tilde{C} 
	\gamma_{n}^{-\hat{p}/\hat{\mu} } \phi_{\varepsilon }(w). 
\end{align*}
Thus, as a result of one of (\ref{l1.4.3}), (\ref{l1.4.5}), 
we get 
\begin{align*}
	\|\nabla f(\theta_{n} ) \|
	\geq 
	\tilde{C} 
	\gamma_{n}^{-\hat{p}/\hat{\mu} } \phi_{\varepsilon }(w). 
\end{align*}
Consequently, 
\begin{align*}
	&
	\hat{t} \|\nabla f(\theta_{n} ) \|^{2}/8
	\geq 
	(\tilde{C} \hat{t} /8 ) \gamma_{n}^{-\hat{p}/\hat{\mu} } 
	\|\nabla f(\theta_{n} ) \|
	\phi_{\varepsilon }(w) 
	\geq 
	\gamma_{n}^{-\hat{p}/\hat{\mu} } 
	\|\nabla f(\theta_{n} ) \|
	\phi_{\varepsilon }(w), 
	\\
	&
	\hat{t} \|\nabla f(\theta_{n} ) \|^{2}/8
	\geq 
	(\tilde{C}^{2} \hat{t} /8 ) \gamma_{n}^{-2\hat{p}/\hat{\mu} }
	(\phi_{\varepsilon }(w) )^{2}
	\geq 
	\hat{C}_{1} \gamma_{n}^{-2\hat{p}/\hat{\mu} } (\phi_{\varepsilon }(w) )^{2}
\end{align*}
(notice that 
$\tilde{C} \hat{t} /8 = \hat{C}_{1}^{1/2} \geq 1$, 
$\tilde{C}^{2} \hat{t} /8 = 8\hat{C}_{1}/\hat{t} > \hat{C}_{1}$). 
Combining this with (\ref{l1.4.3'}), 
we get 
\begin{align}\label{l1.4.7}
	u(\theta_{a(n,\hat{t} ) } ) - u(\theta_{n} )
	\leq 
	-\hat{t} \|\nabla f(\theta_{n} ) \|^{2}/4,  
\end{align}
which directly contradicts (\ref{l1.4.1}). 
Hence, (\ref{l1.4.1*}) is true for $n > \tau_{\varepsilon }$. 
Then, as a result of (\ref{l1.4.5'}) and 
the fact that $B_{n,\varepsilon } \subseteq A_{n,\varepsilon }$
for $n\geq 0$, 
we get
\begin{align*}
	&
	\left(
	u(\theta_{a(n,\hat{t} ) } ) 
	- 
	u(\theta_{n} ) 
	+
	(\hat{t}/\hat{C}_{3} ) \: u(\theta_{n} ) 
	\right)
	I_{B_{n,\varepsilon } }
	\\
	&
	\leq
	\left(
	u(\theta_{a(n,\hat{t} ) } ) 
	- 
	u(\theta_{n} ) 
	+
	(\hat{M} \hat{t} /\hat{C}_{3} ) \: \|\nabla f(\theta_{n} ) \|^{2}
	\right)
	I_{B_{n,\varepsilon } }
	\\
	&
	\leq
	\left(
	u(\theta_{a(n,\hat{t} ) } ) 
	- 
	u(\theta_{n} ) 
	+
	\hat{t} \|\nabla f(\theta_{n} ) \|^{2}/4
	\right)
	I_{B_{n,\varepsilon } }
	\leq 
	0
\end{align*}
for $n> \tau_{\varepsilon}$
(notice that $u(\theta_{n} ) > 0$ on $B_{n,\varepsilon}$ for 
each $n\geq 0$; also notice that $\hat{C}_{3} \geq 4 \hat{M}$). 
Thus, (\ref{l1.4.3*}) is true for $n> \tau_{\varepsilon}$. 

Now, let us prove (\ref{l1.4.5*}). 
To do so, we again use contradiction: 
Suppose that (\ref{l1.4.3*}) does not hold for some $n>\tau_{\varepsilon }$. 
Consequently, we have 
$\hat{\mu} < 2$, $u(\theta_{a(n,\hat{t} ) } ) > 0$ and 
\begin{align}
	&\label{l1.4.9}
	\gamma_{n}^{\hat{p} } \: u(\theta_{n} ) 
	\geq 
	\hat{C}_{2} 
	(\phi_{\varepsilon }(w) )^{\hat{\mu} }
	>
	0, 
	\\
	&\label{l1.4.11}
	v(\theta_{a(n,\hat{t} ) } ) 
	- 
	v(\theta_{n} ) 
	<
	(\hat{t} / \hat{C}_{3} ) 
	(\phi_{\varepsilon }(w) )^{-\hat{\mu}/\hat{p} }.  
\end{align}
Combining (\ref{l1.4.9}) with
(already proved) (\ref{l1.4.1*}), 
we get (\ref{l1.4.7}), 
while $\hat{\mu} < 2$ implies 
 \begin{align} \label{l1.4.11'}
	2/\hat{\mu} 
	= 
	1 + 1/(\hat{\mu} \hat{r} ) 
	\leq 
	1 + 1/\hat{p}
\end{align}
(notice that 
$\hat{r} = 1/(2 - \hat{\mu} )$
owing to $\hat{\mu} < 2$; 
also notice that 
$\hat{p} = \hat{\mu} \min\{r,\hat{r} \} \leq \hat{\mu} \hat{r}$). 
As $0 < u(\theta_{n} ) \leq \hat{\delta} \leq 1$
(due to (\ref{l1.4.9}) and the definition of $\tau_{\varepsilon }$), 
inequalities (\ref{l1.4.5'}), (\ref{l1.4.11'})  
yield
\begin{align}\label{l1.4.15}
	\|\nabla f(\theta_{n} ) \|^{2}
	\geq
	\left(
	u(\theta_{n} )/\hat{M} 
	\right)^{2/\hat{\mu} }
	\geq 
	\left(
	u(\theta_{n} )
	\right)^{1+1/\hat{p} }
	/\hat{M}^{2}  
\end{align}
(notice that $\hat{M}^{2/\hat{\mu} } \leq \hat{M}^{2}$ due to 
$\hat{\mu} < 2$, $\hat{M}\geq 1$). 
Since 
$\|\nabla f(\theta_{n} ) \| > 0$
and 
$0 < u(\theta_{a(n,\hat{t} ) } ) < u(\theta_{n} )$
(due to (\ref{l1.4.5'}), (\ref{l1.4.7}), (\ref{l1.4.9})), 
inequalities (\ref{l1.4.7}), (\ref{l1.4.15}) give
\begin{align*}
	\frac{\hat{t} }{4}
	\leq 
	\frac{u(\theta_{n} ) - u(\theta_{a(n,\hat{t} ) } ) }
	{\|\nabla f(\theta_{n} ) \|^{2} }
	\leq &
	\hat{M}^{2}
	\frac{u(\theta_{n} ) - u(\theta_{a(n,\hat{t} ) } ) }
	{\left(u(\theta_{n} ) \right)^{1+\hat{p} } }
	\\
	= &
	\hat{M}^{2}
	\int^{u(\theta_{n} ) }_{u(\theta_{a(n,\hat{t} ) } ) }
	\frac{du}{\left(u(\theta_{n} ) \right)^{1+\hat{p} } }
	\\
	\leq &
	\hat{M}^{2}
	\int^{u(\theta_{n} ) }_{u(\theta_{a(n,\hat{t} ) } ) }
	\frac{du}{u^{1+\hat{p} } }
	\\
	= &
	\hat{p}
	\hat{M}^{2}
	\left(
	v(\theta_{a(n,\hat{t} ) } )
	-
	v(\theta_{n} )
	\right).
\end{align*}
Therefore, 
\begin{align*}
	v(\theta_{a(n,\hat{t} ) } )
	-
	v(\theta_{n} )
	\geq 
	\hat{t}/(4\hat{p} \hat{M}^{2} )
	\geq
	(\hat{t}/\hat{C}_{3} ) 
\end{align*}
(notice that 
$\hat{p} \leq r$, $\hat{C}_{3} \geq 4 r \hat{M}^{2}$), 
which directly contradicts (\ref{l1.4.11}). 
Thus, (\ref{l1.4.5*}) is satisfied for $n>\tau_{\varepsilon}$. 
\end{proof}

\begin{lemma} \label{lemma1.5}
Suppose that Assumptions \ref{a1.1} -- \ref{a1.3} hold. 
Moreover, let $\varepsilon \in (0,\infty )$
be an arbitrary positive real number. 
Then, 
\begin{align}	\label{l1.5.1*} 
	u(\theta_{n} ) 
	\geq 
	-
	\hat{C}_{2} \gamma_{n}^{-\hat{p} } 
	(\phi_{\varepsilon }(w) )^{\hat{\mu} }
\end{align}
on $\Lambda\setminus N_{0}$
for $n > \tau_{\varepsilon }$. 
Furthermore, 
there exists a random quantity  
$\hat{C}_{4} \in [1,\infty )$ 
(which is a deterministic function of $r$, $\hat{C}$, $\hat{M}$)
such that 
$1\leq \hat{C}_{4} < \infty$ everywhere 
and such that 
\begin{align} \label{l1.5.3*}
	\|\nabla f(\theta_{n} ) \|^{2}
	\leq 
	\hat{C}_{4} 
	\left(
	\varphi(u(\theta_{n} ) )
	+ 
	\gamma_{n}^{-\hat{p} } (\phi_{\varepsilon }(w) )^{\hat{\mu} } 
	\right)
\end{align}
on $\Lambda\setminus N_{0}$ 
for $n> \tau_{\varepsilon }$, 
where
function $\varphi(\cdot )$
is defined by 
$\varphi(x) = x \:{\rm I}_{(0,\infty )}(x)$, $x\in \mathbb{R}$. 
\end{lemma}

\begin{proof}
Let 
$\hat{C}_{4} = 4 \hat{C}_{2} / \hat{t}$, 
while 
$\omega$ is an arbitrary sample from $\Lambda\setminus N_{0}$
(notice that all formulas which follow in the proof correspond to 
this $\omega$). 

First, 
we prove (\ref{l1.5.1*}). 
To do so, we use contradiction:  
Assume that (\ref{l1.5.1*}) is not satisfied for 
some $n > \tau_{\varepsilon }$. 
Define $\{n_{k} \}_{k\geq 0}$ recursively by
$n_{0}=n$
and  
$n_{k} = a(n_{k-1},\hat{t} )$ 
for $k\geq 1$.  
Let us show by induction that 
$\{u(\theta_{n_{k} } ) \}_{k\geq 0}$ is non-increasing: 
Suppose that 
$u(\theta_{n_{l} } ) \leq u(\theta_{n_{l-1} } )$ for 
$0\leq l \leq k$. 
Consequently, 
\begin{align*}
	u(\theta_{n_{k} } ) 
	\leq 
	u(\theta_{n_{0} } ) 
	\leq 
	-
	\hat{C}_{2} \gamma_{n_{0} }^{-\hat{p} }
	(\phi_{\varepsilon }(w) )^{\hat{\mu} }
	\leq 
	-
	\hat{C}_{2} \gamma_{n_{k} }^{-\hat{p} }
	(\phi_{\varepsilon }(w) )^{\hat{\mu} } 
\end{align*}
(notice that $\{\gamma_{n} \}_{n\geq 0}$ is increasing). 
Then, Lemma \ref{lemma1.4} (relations (\ref{l1.4.1*}), (\ref{l1.4.21*})) yields 
\begin{align*}
	u(\theta_{n_{k+1} } ) - u(\theta_{n_{k} } ) 
	\leq 
	-
	\hat{t} \|\nabla f(\theta_{n_{k} } ) \|^{2}/4 
	\leq 
	0,  
\end{align*}
i.e., $u(\theta_{n_{k+1} } ) \leq u(\theta_{n_{k} } )$. 
Thus, $\{u(\theta_{n_{k} } ) \}_{k\geq 0}$ is non-increasing. 
Therefore,
\begin{align*}
	\limsup_{n\rightarrow \infty } 
	u(\theta_{n_{k} } ) 
	\leq 
	u(\theta_{n_{0} } ) 
	< 
	0. 
\end{align*}
However, this is not possible, as 
$\lim_{n\rightarrow \infty } u(\theta_{n} ) = 0$
(due to Lemma \ref{lemma1.3'}). 
Hence, (\ref{l1.5.1*}) indeed holds 
for $n> \tau_{\varepsilon }$. 

Now, (\ref{l1.5.3*}) is demonstrated. 
Again, we proceed by contradiction: 
Suppose that (\ref{l1.5.3*}) is violated for some 
$n> \tau_{\varepsilon }$. 
Consequently, 
\begin{align*}
	\|\nabla f(\theta_{n} ) \|^{2}
	\geq 
	\hat{C}_{4} \gamma_{n}^{-\hat{p} } (\phi_{\varepsilon }(w) )^{\hat{\mu} }
	\geq 
	\hat{C}_{2} \gamma_{n}^{-\hat{p} } (\phi_{\varepsilon }(w) )^{\hat{\mu} } 
\end{align*}
(notice that $\hat{C}_{4} \geq \hat{C}_{2}$),
which, together with Lemma \ref{lemma1.4}
(relations (\ref{l1.4.1*}), (\ref{l1.4.21*})), yields 
\begin{align*}
	u(\theta_{a(n,\hat{t} ) } ) - u(\theta_{n} )
	\leq 
	-\hat{t} \|\nabla f(\theta_{n} ) \|^{2}/4. 
\end{align*}
Then, (\ref{l1.5.1*}) implies 
\begin{align*}
	\|\nabla f(\theta_{n} ) \|^{2}
	\leq &
	(4/\hat{t} )
	\left(
	u(\theta_{n} ) 
	-
	u(\theta_{a(n,\hat{t} )} ) 
	\right)
	\\
	\leq &
	(4/\hat{t} )
	\left(
	\varphi(u(\theta_{n} ) )
	+
	\hat{C}_{2} \gamma_{a(n,\hat{t} ) }^{-\hat{p} } 
	(\phi_{\varepsilon }(w) )^{\hat{\mu} }
	\right)
	\\
	\leq &
	\hat{C}_{4} 
	\left(
	\varphi(u(\theta_{n} ) )
	+ 
	\gamma_{n}^{-\hat{p} } 
	(\phi_{\varepsilon }(w) )^{\hat{\mu} }
	\right). 
\end{align*}
However, this directly contradicts our assumption 
that $n$ violates (\ref{l1.5.3*}). 
Thus, (\ref{l1.5.3*}) is indeed satisfied for 
$n>\tau_{\varepsilon }$. 
\end{proof}

\begin{lemma} \label{lemma1.6}
Suppose that Assumptions \ref{a1.1} -- \ref{a1.3} hold. 
Then, 
there exists a random quantity 
$\hat{C}_{5}$ 
(which is a deterministic function of $r$, $\hat{C}$, $\hat{M}$)
such that 
$1\leq \hat{C}_{5} < \infty$ everywhere
and 
such that 
\begin{align} \label{l1.6.1*}
	\liminf_{n\rightarrow \infty } 
	\gamma_{n}^{\hat{p} } \:
	u(\theta_{n} ) 
	\leq 
	\hat{C}_{5}
	(\phi(w) )^{\hat{\mu} }
\end{align}
on $\Lambda\setminus N_{0}$. 
\end{lemma}

\begin{proof}
Let 
$\hat{C}_{5} = \hat{C}_{2}+\hat{C}_{3}^{2r}$. 
We prove (\ref{l1.6.1*}) by contradiction: 
Assume that (\ref{l1.6.1*}) is violated for some 
sample $\omega$ from $\Lambda\setminus N_{0}$
(notice that the formulas which follow in the proof correspond to 
this $\omega$). 
Consequently, there exist $\varepsilon \in (0,\infty )$ 
and 
$n_{0} > \tau_{\varepsilon }$ 
such that 
\begin{align} \label{l1.6.1}
	u(\theta_{n} ) 
	\geq 
	\hat{C}_{5} \gamma_{n}^{-\hat{p} } 
	(\phi_{\varepsilon }(w) )^{\hat{\mu} }
\end{align}
for $n\geq n_{0}$. 
Let $\{n_{k} \}_{k\geq 0}$ be defined recursively by 
$n_{k} = a(n_{k-1},\hat{t} )$ for $k\geq 1$. 	
In what follows in the proof, we consider separately 
the cases $\hat{\mu} < 2$
and $\hat{\mu} = 2$. 

{\em Case $\hat{\mu} < 2$:}
Due to (\ref{l1.6.1}), we have 
\begin{align*}
	&
	\begin{aligned}[b]
		v(\theta_{n_{k} } ) 
		\leq &
		\hat{C}_{5}^{-1/\hat{p} } 
		\gamma_{n_{k} }
		(\phi_{\varepsilon }(w) )^{-\hat{\mu} /\hat{p} }
		\leq &
		\hat{C}_{5}^{-1/(2r) } \gamma_{n_{k} }
		(\phi_{\varepsilon}(w) )^{-\hat{\mu}/\hat{p} }  
	\end{aligned}
\end{align*}
(notice that $\hat{p} \leq 2r$). 
On the other side, Lemma \ref{lemma1.4} 
(relations (\ref{l1.4.5*}), (\ref{l1.4.25*})) 
and (\ref{l1.6.1}) yield
\begin{align*}
	&
	\begin{aligned}[b]
		v(\theta_{n_{k+1} } ) - v(\theta_{n_{k} } ) 
		\geq 
		(\hat{t}/\hat{C}_{3} )
		(\phi_{\varepsilon }(w) )^{-\hat{\mu}/\hat{p} }
		\geq 
		(1/\hat{C}_{3} ) (\gamma_{n_{k+1} } - \gamma_{n_{k} } )
		(\phi_{\varepsilon }(w) )^{-\hat{\mu}/\hat{p} }
	\end{aligned}
\end{align*}
for $k\geq 0$
(notice that $\hat{C}_{5} \geq \hat{C}_{2}$;
also notice that $\hat{t} \geq \gamma_{n_{k+1} } - \gamma_{n_{k} }$).  
Therefore, 
\begin{align*}
	(1/\hat{C}_{3} ) (\gamma_{n_{k} }  - \gamma_{n_{0} } )
	(\phi_{\varepsilon }(w) )^{-\hat{\mu}/\hat{p} }
	\leq &
	\sum_{i=0}^{k-1} 
	(v(\theta_{n_{i+1} } ) - v(\theta_{n_{i} } ) )
	\\
	= &
	v(\theta_{n_{k} } ) - v(\theta_{n_{0} } ) 
	\\
	\leq &
	\hat{C}_{5}^{-1/(2r) } \gamma_{n_{k} }
	(\phi_{\varepsilon }(w) )^{-\hat{\mu}/\hat{p} }
\end{align*}
for $k\geq 1$. 
Thus, 
\begin{align*}
	(1 - \gamma_{n_{0} }/\gamma_{n_{k} } )
	\leq 
	\hat{C}_{3} \hat{C}_{5}^{-1/(2r) }
\end{align*}
for $k\geq 1$. 
However, this is impossible, since the limit process 
$k\rightarrow \infty$ (applied to the previous relation)  
yields 
$\hat{C}_{3} \geq \hat{C}_{5}^{1/(2r) }$
(notice that $\hat{C}_{5} > \hat{C}_{3}^{2r}$). 
Hence, (\ref{l1.6.1*}) holds on 
$\Lambda\setminus N_{0}$ when $\hat{\mu} < 2$. 

{\em Case $\hat{\mu} = 2$:} 
As a result of 
Lemma \ref{lemma1.4} (relations (\ref{l1.4.3*}), (\ref{l1.4.23*})) and (\ref{l1.6.1}), we get 
\begin{align*}
	u(\theta_{n_{k+1} } )
	\leq
	(1 - \hat{t}/\hat{C}_{3} ) u(\theta_{n_{k} } )
	\leq
	\left(
	1
	-
	(\gamma_{n_{k+1} } - \gamma_{n_{k} } )/ \hat{C}_{3}
	\right)
	u(\theta_{n_{k} } )
\end{align*}
for $k\geq 0$. 
Consequently, 
\begin{align*}
	u(\theta_{n_{k} } )
	\leq &
	u(\theta_{n_{0} } )
	\prod_{i=1}^{k} 
	\left(
	1 
	-  
	(\gamma_{n_{i} } - \gamma_{n_{i-1} } )/\hat{C}_{3} 	
	\right)
	\\
	\leq &
	u(\theta_{n_{0} } )	
	\exp\left(
	-
	(1/\hat{C}_{3} )
	\sum_{i=1}^{k} (\gamma_{n_{i} } - \gamma_{n_{i-1} } )	
	\right)
	\\
	= & 
	u(\theta_{n_{0} } )	
	\exp\left(
	- 
	(\gamma_{n_{k} } - \gamma_{n_{0} } )/\hat{C}_{3}	
	\right)
\end{align*}
for $k\geq 0$. 
Then, (\ref{l1.6.1}) yields 
\begin{align*}
	\hat{C}_{5} 
	(\phi_{\varepsilon}(w) )^{\hat{\mu} }
	\leq 
	u(\theta_{n_{0} } )
	\gamma_{n_{k} }^{\hat{p} } 
	\exp\left(
	-(\gamma_{n_{k} } - \gamma_{n_{0} } )/\hat{C}_{3} 
	\right)
\end{align*}
for $k\geq 0$. 
However, this is not possible, as the limit 
process $k\rightarrow \infty$
(applied to the previous relation) 
implies 
$\hat{C}_{5} (\phi_{\varepsilon}(w) )^{\hat{\mu} }
\leq 0$. 
Thus, (\ref{l1.6.1*}) holds on 
$\Lambda\setminus N_{0}$ also when $\hat{\mu} = 2$. 
\end{proof}

\begin{lemma} \label{lemma1.7}
Suppose that Assumptions \ref{a1.1} -- \ref{a1.3} hold. 
Then, 
there exists a random quantity  
$\hat{C}_{6}$ (which is a deterministic function of $r$, $\hat{C}$, $\hat{M}$)
such that $1\leq \hat{C}_{6} < \infty$ everywhere 
and such that 
\begin{align} \label{l1.7.1*}
	\limsup_{n\rightarrow \infty } 
	\gamma_{n}^{\hat{p} } \:
	u(\theta_{n} ) 
	\leq 
	\hat{C}_{6} 
	(\phi(w) )^{\hat{\mu} }	
\end{align}
on $\Lambda\setminus N_{0}$. 
\end{lemma}

\begin{proof}
Let 
$\tilde{C}_{1} = \hat{C}_{1} + \hat{C}_{4} + \hat{C}_{5}$, 
$\tilde{C}_{2} = 6 \tilde{C}_{1} \hat{C}_{2} + \hat{C}_{3}^{2r}$
and 
$\hat{C}_{6} = 2 (\tilde{C}_{1} + \tilde{C}_{2} )^{2}$. 
We use contradiction to show (\ref{l1.7.1*}): 
Suppose that (\ref{l1.7.1*}) is violated for some sample 
$\omega$ from $\Lambda\setminus N_{0}$
(notice that the formulas which appear in the proof correspond to 
this $\omega$).
Then, it can be deduced from Lemma \ref{lemma1.6} that 
there exist
$\varepsilon \in (0,\infty )$
and 
$n_{0} > m_{0} > \tau_{\varepsilon }$ such that 
\begin{align}
	& \label{l1.7.1}
	\gamma_{m_{0} }^{\hat{p} } u(\theta_{m_{0} } ) 
	\leq 
	\tilde{C}_{2} 
	(\phi_{\varepsilon }(w) )^{\hat{\mu} }, 
	\\
	& \label{l1.7.3}
	\gamma_{n_{0} }^{\hat{p} } u(\theta_{n_{0} } ) 
	\geq 
	\hat{C}_{6} 
	(\phi_{\varepsilon }(w) )^{\hat{\mu} }, 
	\\
	& \label{l1.7.5}
	\min_{m_{0}<n\leq n_{0} }
	\gamma_{n}^{\hat{p} } \: u(\theta_{n} ) 
	>
	\tilde{C}_{2}  
	(\phi_{\varepsilon }(w) )^{\hat{\mu} },   
	\\
	& \label{l1.7.5'}
	\max_{m_{0}\leq n < n_{0} }
	\gamma_{n}^{\hat{p} } \: u(\theta_{n} ) 
	<
	\hat{C}_{6} 
	(\phi_{\varepsilon }(w) )^{\hat{\mu} }  
\end{align}
(notice that $\tilde{C}_{2} > \tilde{C}_{1} > \hat{C}_{5}$) 
and such that 
\begin{align} 
	& \label{l1.7.7'}
	(\gamma_{a(m_{0},\hat{t} ) }/\gamma_{m_{0} } )^{\hat{p} }
	\leq
	\min\{
	2, (1 - \hat{t}/\hat{C}_{3} )^{-1} 
	\}, 
	\\
	&	\label{l1.7.7}
	\gamma_{m_{0} }^{-2\hat{p}/\hat{\mu} } 
	(\phi_{\varepsilon }(w) )^{2}
	\leq 
	\gamma_{m_{0} }^{-\hat{p} } 
	(\phi_{\varepsilon }(w) )^{\hat{\mu} }
\end{align}
(to see that (\ref{l1.7.7'}) holds for all, but finitely many $m_{0}$, 
notice that 
$\lim_{n\rightarrow \infty} \gamma_{a(n,\hat{t} ) }/\gamma_{n} = 1$; 
to conclude that (\ref{l1.7.7}) is true for all, but finitely many $m_{0}$, 
notice that 
$2\hat{p}/\hat{\mu} >  
\hat{p}$ if $\hat{\mu} < 2$
and
that the left and right-hand sides of (\ref{l1.7.7}) are 
equal when $\hat{\mu}=2$). 

Let $l_{0} = a(m_{0},\hat{t} )$. 
As a direct consequence of Lemmas \ref{lemma1.2}, \ref{lemma1.5} 
(relations (\ref{l1.2.1*}),(\ref{l1.5.3*})) and 
(\ref{l1.7.7}), we get
\begin{align} \label{l1.7.9}
	u(\theta_{n} )
	-
	u(\theta_{m_{0} } )
	\leq &
	\gamma_{m_{0} }^{-\hat{p}/\hat{\mu} } 
	\|\nabla f(\theta_{m_{0} } ) \| 
	\phi_{\varepsilon}(w) 
	+
	\hat{C}_{1} 
	\gamma_{m_{0} }^{-2\hat{p}/\hat{\mu} } 
	(\phi_{\varepsilon}(w) )^{2} 
	\nonumber\\
	\leq &
	\|\nabla f(\theta_{m_{0} } ) \|^{2}/2 
	+
	(\hat{C}_{1} + 1/2 )
	\gamma_{m_{0} }^{-2\hat{p}/\hat{\mu} }
	(\phi_{\varepsilon }(w) )^{2}
	\nonumber \\
	\leq &
	\hat{C}_{4} \: \varphi(u(\theta_{m_{0} } ) )
	+
	(\hat{C}_{1} + \hat{C}_{4} + 1 )
	\gamma_{m_{0} }^{-\hat{p} }
	(\phi_{\varepsilon }(w) )^{\hat{\mu} }
	\nonumber \\
	\leq &
	\tilde{C}_{1} 
	\left(
	\varphi(u(\theta_{m_{0} } ) )
	+
	\gamma_{m_{0} }^{-\hat{p} }
	(\phi_{\varepsilon }(w) )^{\hat{\mu} }
	\right) 
\end{align}
for $m_{0} \leq n \leq l_{0}$
(notice that 
$\hat{C}_{1} + \hat{C}_{4} + 1 < \tilde{C}_{1}$). 
Then, 
(\ref{l1.7.5}), (\ref{l1.7.7'}), (\ref{l1.7.9}) yield 
\begin{align} \label{l1.7.21'}
	u(\theta_{m_{0} } ) 
	+
	\tilde{C}_{1} 
	\varphi(u(\theta_{m_{0} } ) )
	\geq &
	u(\theta_{m_{0} + 1 } ) 
	-
	\tilde{C}_{1} 
	\gamma_{m_{0} }^{-\hat{p} }
	(\phi_{\varepsilon }(w) )^{\hat{\mu} }
	\nonumber \\
	\geq &
	(\tilde{C}_{2} 
	\gamma_{m_{0} + 1 }^{-\hat{p} }
	-
	\tilde{C}_{1} 
	\gamma_{m_{0} }^{-\hat{p} })
	(\phi_{\varepsilon }(w) )^{\hat{\mu} }
	\nonumber \\
	\geq &
	\left(
	\tilde{C}_{2} 
	(\gamma_{m_{0} + 1 }/\gamma_{m_{0} } )^{-\hat{p} } 
	-
	\tilde{C}_{1} 
	\right)
	\gamma_{m_{0} }^{-\hat{p} }
	(\phi_{\varepsilon }(w) )^{\hat{\mu} }
	\nonumber \\
	\geq &
	(\tilde{C}_{2}/2 - \tilde{C}_{1} )  
	\gamma_{m_{0} }^{-\hat{p} }
	(\phi_{\varepsilon }(w) )^{\hat{\mu} } 
	> 0 
\end{align}
(notice that 
$(\gamma_{m_{0}+1 } /\gamma_{m_{0} } )^{\hat{p} } \leq 
(\gamma_{l_{0} } /\gamma_{m_{0} } )^{\hat{p} } \leq 2$; 
also notice that $\tilde{C}_{2}/2 \geq 3 \tilde{C}_{1}$), 
while (\ref{l1.7.1}), (\ref{l1.7.7'}), (\ref{l1.7.9}) imply 
\begin{align} \label{l1.7.21}
	u(\theta_{n} )  
	\leq &
	(1 + \tilde{C}_{1} )
	u(\theta_{m_{0} } ) 
	+
	\tilde{C}_{1} 
	\gamma_{m_{0} }^{-\hat{p} }
	(\phi_{\varepsilon }(w) )^{\hat{\mu} }
	\nonumber \\
	\leq &
	(\tilde{C}_{1} + \tilde{C}_{2} + \tilde{C}_{1} \tilde{C}_{2} )
	\gamma_{m_{0} }^{-\hat{p} } (\phi_{\varepsilon }(w) )^{\hat{\mu} } 
	\nonumber \\
	< &
	(\hat{C}_{6}/2) 
	(\gamma_{n}/\gamma_{m_{0} } )^{\hat{p} }
	\gamma_{n}^{-\hat{p} } (\phi_{\varepsilon }(w) )^{\hat{\mu} } 
	\nonumber \\
	\leq  &
	\hat{C}_{6} 
	\gamma_{n}^{-\hat{p} } (\phi_{\varepsilon }(w) )^{\hat{\mu} } 
\end{align}
for $m_{0} \leq n \leq l_{0}$
(notice that 
$(\gamma_{n}/\gamma_{m_{0} } )^{\hat{p} } \leq 
(\gamma_{l_{0} }/\gamma_{m_{0} } )^{\hat{p} } \leq 2$
for $m_{0} \leq n \leq l_{0}$;
also notice that $\hat{C}_{6}/2 = (\tilde{C}_{1} + \tilde{C}_{2} )^{2} >
\tilde{C}_{1} + \tilde{C}_{2} + \tilde{C}_{1} \tilde{C}_{2}$).  
Due to  
(\ref{l1.7.3}), (\ref{l1.7.5'}), (\ref{l1.7.21}), we have 
$l_{0} < n_{0}$.
On the other side, 
as $x + \tilde{C}_{1} \varphi(x) \geq 0$ only if $x\geq 0$
and $x + \tilde{C}_{1} \varphi(x) = (1 + \tilde{C}_{1} ) x$ for $x\geq 0$, 
inequality (\ref{l1.7.21'}) implies   
\begin{align} \label{l1.7.23} 
	u(\theta_{m_{0} } )
	\geq &
	(1 + \tilde{C}_{1} )^{-1} (\tilde{C}_{2}/2 - \tilde{C}_{1} )
	\gamma_{m_{0} }^{-\hat{p} }
	(\phi_{\varepsilon }(w) )^{\hat{\mu} }
	\geq 
	\hat{C}_{2} 
	\gamma_{m_{0} }^{-\hat{p} }
	(\phi_{\varepsilon }(w) )^{\hat{\mu} } 
\end{align}
(notice that 
$\tilde{C}_{2}/2 - \tilde{C}_{1}
\geq \tilde{C}_{1} (3 \hat{C}_{2} - 1 )
\geq 2 \tilde{C}_{1} \hat{C}_{2} 
\geq (1+\tilde{C}_{1} ) \hat{C}_{2}$).

In what follows in the proof, we consider separately the cases 
$\hat{\mu} < 2$ and $\hat{\mu} = 2$. 

{\em Case $\hat{\mu} < 2$:}
Owing to Lemma \ref{lemma1.4} (relations (\ref{l1.4.5*}), (\ref{l1.4.25*})) and 
(\ref{l1.7.1}), 
(\ref{l1.7.23}), we have 
\begin{align*}
	v(\theta_{l_{0} } ) 
	\geq &
	v(\theta_{m_{0} } ) 
	+
	(\hat{t}/\hat{C}_{3} ) 
	(\phi_{\varepsilon }(w) )^{-\hat{\mu}/\hat{p} }
	\\
	\geq &
	\left(
	\tilde{C}_{2}^{-1/\hat{p} } 
	\gamma_{m_{0} } 
	+
	\hat{C}_{3}^{-1} 
	(\gamma_{l_{0} } - \gamma_{m_{0} } )
	\right)
	(\phi_{\varepsilon }(w) )^{-\hat{\mu}/\hat{p} }
	\\
	> &
	\min\{\hat{C}_{2}^{-1/\hat{p} }, \hat{C}_{3}^{-1} \} 
	\gamma_{l_{0} } (\phi_{\varepsilon}(w) )^{-\hat{\mu}/\hat{p} }
	\\
	= &
	\tilde{C}_{2}^{-1/\hat{p} } 
	\gamma_{l_{0} } 	
	(\phi_{\varepsilon }(w) )^{-\hat{\mu}/\hat{p} } 
\end{align*}
(notice that 
$\hat{t} \geq \gamma_{l_{0} } - \gamma_{m_{0} }$; 
also notice 
$\tilde{C}_{2}^{-1/\hat{p} } \leq 
\tilde{C}_{2}^{-1/(2r) } < 
\hat{C}_{3}^{-1}$).
Consequently, 
\begin{align*}
	u(\theta_{l_{0} } ) 
	=
	\left(
	v(\theta_{l_{0} } ) 
	\right)^{-\hat{p} }
	< 
	\tilde{C}_{2} 
	\gamma_{l_{0} }^{-\hat{p} }
	(\phi_{\varepsilon }(w) )^{\hat{\mu} }. 
\end{align*}
However, this directly contradicts 
(\ref{l1.7.5}) and the fact that 
$l_{0} < n_{0}$. 
Thus, (\ref{l1.7.1*}) holds 
when $\hat{\mu} < 2$. 

{\em Case $\hat{\mu} = 2$:}
Using Lemma \ref{lemma1.4} (relations (\ref{l1.4.3*}), (\ref{l1.4.23*})) and (\ref{l1.7.23}), 
we get
\begin{align*}
	u(\theta_{l_{0} } )
	\leq 
	\left(
	1
	- 
	\hat{t}/\hat{C}_{3} 
	\right)
	u(\theta_{{m}_{0} } ). 
\end{align*}
Then, (\ref{l1.7.1}), (\ref{l1.7.7'}) yield 
\begin{align*}
	u(\theta_{l_{0} } )
	\leq &
	\tilde{C}_{2} 
	(1 - \hat{t}/\hat{C}_{3} )
	(\gamma_{l_{0} }/\gamma_{m_{0} } )^{\hat{p} }
	\gamma_{l_{0} }^{-\hat{p} }
	(\phi_{\varepsilon}(w) )^{\hat{\mu} }
	\leq 
	\tilde{C}_{2} 
	\gamma_{l_{0} }^{-\hat{p} }
	(\phi_{\varepsilon}(w) )^{\hat{\mu} }. 
\end{align*}
However, this is impossible due to (\ref{l1.7.5})
and the fact that $l_{0} < n_{0}$. 
Hence, 
(\ref{l1.7.1*}) 
also in the case $\hat{\mu} = 2$. 
\end{proof}

\begin{vproof}{Theorems \ref{theorem1.1} and \ref{theorem1.2}}
{Theorem \ref{theorem1.1} is an immediate consequence of 
Lemmas \ref{lemma1.2}, \ref{lemma1.3}. 
To show Theorem \ref{theorem1.2}, we use the following notations: 
$\hat{K} = (\hat{C}_{2} + \hat{C}_{4} + \hat{C}_{6} )^{2}$, 
$\hat{L} = \hat{K} \hat{N}$. 
Then, Lemmas \ref{lemma1.4} and \ref{lemma1.6} imply 
\begin{align} \label{t1.1.1} 
	\limsup_{n\rightarrow \infty } 
	\gamma_{n}^{\hat{p} } |u(\theta_{n} ) |
	\leq 
	(\hat{C}_{2} + \hat{C}_{6} ) 
	(\phi(w) )^{\hat{\mu} }
\end{align}
on $\Lambda\setminus N_{0}$. 
On the other side, 
Lemma \ref{lemma1.4} and (\ref{t1.1.1}) yield
\begin{align} \label{t1.1.3}
	\limsup_{n\rightarrow \infty } 
	\gamma_{n}^{\hat{p} } \|\nabla f(\theta_{n} ) \|^{2}
	\leq &
	\hat{C}_{4} 
	(\phi(w) )^{\hat{\mu} }
	+
	\hat{C}_{4}
	\limsup_{n\rightarrow \infty } 
	\gamma_{n}^{\hat{p} } 
	\varphi(u(\theta_{n} ) )
	\nonumber \\
	\leq &
	(\hat{C}_{2} + \hat{C}_{4} + \hat{C}_{6} )^{2} 
	(\phi(w) )^{\hat{\mu} }
\end{align}
on $\Lambda\setminus N_{0}$. 
Combining (\ref{t1.1.1}), (\ref{t1.1.3}) with Assumption \ref{a1.4}, we get 
\begin{align} \label{t1.1.5}
	\limsup_{n\rightarrow \infty } 
	\gamma_{n}^{\hat{q} } d(\theta_{n}, S ) 
	\leq &
	\hat{N} 
	\limsup_{n\rightarrow \infty } 
	\left(
	\gamma_{n}^{\hat{p} } \|\nabla f(\theta_{n} ) \|^{2} 
	\right)^{\nu_{\hat{Q} }/2 }
	\nonumber \\
	\leq &
	\hat{N} (\hat{C}_{2} + \hat{C}_{4} + \hat{C}_{6} )^{2} 
	(\phi(w) )^{\hat{\nu} }
\end{align}
on $\Lambda\setminus N_{0}$. 
As a direct consequence of (\ref{t1.1.1}) -- (\ref{t1.1.5}), 
we have that (\ref{t1.1.1*}) -- (\ref{t1.1.5*}) are satisfied  
on $\Lambda\setminus N_{0}$. 
Hence, Theorem \ref{theorem1.2} holds, too.
}
\end{vproof}

\section{Proof of Theorem \ref{theorem2.1}} \label{section2*} 

The following notation is used in this section. 
For $\theta \in \mathbb{R}^{d_{\theta } }$, $\xi \in \mathbb{R}^{d_{\xi} }$, 
$E_{\theta, \xi }(\cdot )$ denotes 
$E(\cdot|\theta_{0}=\theta, \xi_{0}=\xi)$. 
Moreover, let 
\begin{align*} 
	& 
	w_{n} 
	=
	F(\theta_{n}, \xi_{n+1} ) - \nabla f(\theta_{n} ), 
	\\
	&
	w_{1,n} 
	=
	\tilde{F}(\theta_{n}, \xi_{n+1} ) - (\Pi\tilde{F} )(\theta_{n}, \xi_{n} ), 
	\nonumber \\
	&
	w_{2,n}
	=
	(\Pi\tilde{F} )(\theta_{n}, \xi_{n} ) - (\Pi\tilde{F} )(\theta_{n-1}, \xi_{n} ), 
	\nonumber \\
	&
	w_{3,n}
	=
	-(\Pi\tilde{F} )(\theta_{n}, \xi_{n+1} ) 
	\nonumber 
\end{align*}
for $n\geq 1$. 
Then, it is obvious that 
algorithm (\ref{2.1}) admits the form (\ref{1.1}), 
while Assumption \ref{a2.2} yields  
\begin{align} \label{2.3*}
	\sum_{i=n}^{k} 
	\alpha_{i} \gamma_{i}^{r} w_{i} 
	= &
	\sum_{i=n}^{k} 
	\alpha_{i} \gamma_{i}^{r} w_{1,i} 
	+
	\sum_{i=n}^{k} 
	\alpha_{i} \gamma_{i}^{r} w_{2,i} 
	-
	\sum_{i=n}^{k} 
	(\alpha_{i} \gamma_{i}^{r} - \alpha_{i+1} \gamma_{i+1}^{r} ) w_{3,i} 
	\nonumber \\
	&
	-
	\alpha_{k+1} \gamma_{k+1}^{r} w_{3,k} 
	+
	\alpha_{n} \gamma_{n}^{r} w_{3,n-1} 
\end{align}
for $1\leq n \leq k$. 

\begin{lemma} \label{lemma11.1}
Let Assumption \ref{a2.1} hold. 
Then, there exists a real number $s \in (0,1)$ such that 
$\sum_{n=0}^{\infty } \alpha_{n}^{1+s} \gamma_{n}^{r} < \infty$. 
\end{lemma} 

\begin{proof}
Let $p=(2+2r)/(2+r)$, 
$q=(2+2r)/r$, 
$s=(2+r)/(2+2r)$. 
Then, using the H\"{o}lder inequality, we get
\begin{align*}
	\sum_{n=0}^{\infty } 
	\alpha_{n}^{1+s} \gamma_{n}^{r} 
	=
	\sum_{n=1}^{\infty } 
	(\alpha_{n}^{2} \gamma_{n}^{2r} )^{1/p}
	\left(
	\frac{\alpha_{n} }{\gamma_{n}^{2} } 
	\right)^{1/q} 
	\leq
	\left(
	\sum_{n=1}^{\infty }
	\alpha_{n}^{2} \gamma_{n}^{2r} 
	\right)^{1/p}
	\left(
	\sum_{n=1}^{\infty } 
	\frac{\alpha_{n} }{\gamma_{n}^{2} } 
	\right)^{1/q}. 
\end{align*}
Since 
$\gamma_{n+1}/\gamma_{n} = 1 + \alpha_{n}/\gamma_{n} = O(1)$ 
for $n\rightarrow \infty$ and 
\begin{align*}
	\sum_{n=1}^{\infty } 
	\frac{\alpha_{n} }{\gamma_{n}^{2} }
	=
	\sum_{n=1}^{\infty } 
	\frac{\gamma_{n+1} - \gamma_{n} }{\gamma_{n}^{2} }
	\leq 
	\sum_{n=1}^{\infty } 
	\left(\frac{\gamma_{n+1} }{\gamma_{n} } \right)^{2}
	\int_{\gamma_{n} }^{\gamma_{n+1} } 
	\frac{dt}{t^{2} }
	\leq 
	\frac{1}{\gamma_{1} }
	\max_{n\geq 0} \left(\frac{\gamma_{n+1} }{\gamma_{n} } \right)^{2}, 
\end{align*}
it is obvious that 
$\sum_{n=0}^{\infty } \alpha_{n}^{1+s} \gamma_{n}^{r}$ converges. 
\end{proof} 

\begin{vproof}{Theorem \ref{theorem2.1}}
{Let $Q \subset \mathbb{R}^{d_{\theta } }$ be an arbitrary compact set, 
while $s\in (0,1)$ is a real number such that 
$\sum_{n=0}^{\infty} \alpha_{n}^{1+s} \gamma_{n}^{r} < \infty$.  
Obviously, it is sufficient to show that  
$\sum_{n=0}^{\infty } \alpha_{n} \gamma_{n}^{r} w_{n}$ 
converges w.p.1 on $\bigcap_{n=0}^{\infty} \{\theta_{n} \in Q \}$. 

Due to
Assumption \ref{a2.1}, we have 
\begin{align*}
	&
	\alpha_{n-1}^{s} \alpha_{n} \gamma_{n}^{r} 
	=
	\left(
	1 + \alpha_{n-1} (\alpha_{n}^{-1} - \alpha_{n-1}^{-1} )
	\right)^{s}
	\alpha_{n}^{1+s} \gamma_{n}^{r}
	=
	O(\alpha_{n}^{1+s} \gamma_{n}^{r} ), 
	\\
	&
	(\alpha_{n-1} - \alpha_{n} ) \gamma_{n}^{r}
	=
	(\alpha_{n}^{-1} - \alpha_{n-1}^{-1} ) 
	\left(
	1 + \alpha_{n-1} (\alpha_{n}^{-1} - \alpha_{n-1}^{-1} )
	\right)
	\alpha_{n}^{2} \gamma_{n}^{r} 
	=
	O(\alpha_{n}^{2} \gamma_{n}^{r} ), 	
	\\
	&
	\alpha_{n} (\gamma_{n+1}^{r} - \gamma_{n}^{r} )
	=
	\alpha_{n} \gamma_{n}^{r} 
	\left(
	(1 + \alpha_{n}/\gamma_{n} )^{r} - 1 
	\right)
	=
	\alpha_{n} \gamma_{n}^{r} 
	\left(
	r\alpha_{n}/\gamma_{n} + o(\alpha_{n}/\gamma_{n} ) 
	\right)
	=
	o(\alpha_{n}^{2} \gamma_{n}^{r} )
\end{align*}
as $n\rightarrow \infty$. 
Consequently, 
\begin{align}
	& \label{t2.1.1}
	\sum_{n=0}^{\infty } 
	\alpha_{n}^{s} \alpha_{n+1} \gamma_{n+1}^{r} 	
	< 
	\infty, 
	\\
	& \label{t2.1.3} 
	\sum_{n=0}^{\infty } 
	|\alpha_{n} \gamma_{n}^{r} - \alpha_{n+1} \gamma_{n+1}^{r} |
	\leq 
	\sum_{n=0}^{\infty } 
	\alpha_{n} |\gamma_{n}^{r} - \gamma_{n+1}^{r} |
	+
	\sum_{n=0}^{\infty } 
	|\alpha_{n} - \alpha_{n+1} | \gamma_{n+1}^{r} 
	< 
	\infty. 
\end{align}
On the other side, 
as a result of Assumption \ref{a2.3}, we get 
\begin{align*}
	E_{\theta,\xi }
	\left(
	\|w_{1,n} \|^{2}
	I_{ \{\tau_{Q} > n \} }
	\right)
	\leq &
	2 
	E_{\theta,\xi }
	\left(
	\varphi_{Q,s }^{2}(\xi_{n+1} ) 
	I_{ \{\tau_{Q} > n \} }
	\right)
	+
	2 
	E_{\theta,\xi }
	\left(
	\varphi_{Q,s }^{2}(\xi_{n} ) 
	I_{ \{\tau_{Q} > n-1 \} }
	\right), 
	\\
	E_{\theta,\xi }
	\left(
	\|w_{2,n} \|^{2}
	I_{ \{\tau_{Q} > n \} }
	\right)
	\leq &
	E_{\theta,\xi }
	\left(
	\varphi_{Q,s }(\xi_{n} )
	\|\theta_{n} - \theta_{n-1} \|^{s} 
	I_{ \{\tau_{Q} > n-1 \} }
	\right)
	\\
	\leq & 
	\alpha_{n-1}^{s} 
	E_{\theta,\xi }
	\left(
	\varphi_{Q,s }^{2}(\xi_{n} ) 
	I_{ \{\tau_{Q} > n-1 \} }
	\right), 
	\\
	E_{\theta,\xi }
	\left(
	\|w_{3,n} \|^{2}
	I_{ \{\tau_{Q} > n \} }
	\right)
	\leq &
	E_{\theta,\xi }
	\left(
	\varphi_{Q,s }^{2}(\xi_{n+1} )
	I_{ \{\tau_{Q} > n \} }
	\right)
\end{align*}
for all $\theta \in \mathbb{R}^{d_{\theta } }$, 
$\xi \in \mathbb{R}^{d_{\xi } }$, $n\geq 1$. 
Then, Assumption \ref{a2.1} and (\ref{t2.1.1}) yield 
\begin{align*}
	&
	E_{\theta,\xi }
	\left(
	\sum_{n=1}^{\infty } 
	\alpha_{n}^{2} \gamma_{n}^{2r} 
	\|w_{1,n} \|^{2}
	I_{ \{\tau_{Q} > n \} }
	\right)
	\leq 
	4
	\left(
	\sum_{n=1}^{\infty } 
	\alpha_{n}^{2} \gamma_{n}^{2r} 
	\right)
	\sup_{n\geq 0}
	E_{\theta,\xi }
	\left(
	\varphi_{Q,s }^{2}(\xi_{n} ) 
	I_{ \{\tau_{Q} \geq n \} }
	\right)
	< 
	\infty, 
	\\
	&
	E_{\theta,\xi }
	\left(
	\sum_{n=1}^{\infty } 
	\alpha_{n} \gamma_{n}^{r} 
	\|w_{2,n} \| 
	I_{ \{\tau_{Q} > n \} }
	\right)
	\leq 
	\left(
	\sum_{n=1}^{\infty } 
	\alpha_{n-1}^{s} \alpha_{n} \gamma_{n}^{r} 
	\right)
	\sup_{n\geq 0}
	E_{\theta,\xi }
	\left(
	\varphi_{Q,s }^{2}(\xi_{n} ) 
	I_{ \{\tau_{Q} \geq n \} }
	\right)
	<
	\infty 
\end{align*}
for any $\theta \in \mathbb{R}^{d_{\theta } }$, 
$\xi \in \mathbb{R}^{d_{\xi } }$, 
while (\ref{t2.1.3}) implies 
\begin{align*}
	&
	E_{\theta,\xi }
	\left(
	\sum_{n=1}^{\infty } 
	|\alpha_{n} \gamma_{n}^{r} - \alpha_{n+1} \gamma_{n+1}^{r} |
	\|w_{3,n} \| 
	I_{ \{\tau_{Q} > n \} }
	\right)
	\\
	& \;\;\; 
	\leq 
	\left(
	\sum_{n=1}^{\infty } 
	|\alpha_{n} \gamma_{n}^{r} - \alpha_{n+1} \gamma_{n+1}^{r} |
	\right)
	\sup_{n\geq 0}
	\left(
	E_{\theta,\xi }
	\left(
	\varphi_{Q,s }^{2}(\xi_{n} ) 
	I_{ \{\tau_{Q} \geq n \} }
	\right)
	\right)^{1/2} 
	<
	\infty, 
	\\
	&
	E_{\theta,\xi }
	\left(
	\sum_{n=1}^{\infty } 
	\alpha_{n+1}^{2} \gamma_{n+1}^{2r} 
	\|w_{3,n} \|^{2} 
	I_{ \{\tau_{Q} > n \} }
	\right)
	\\
	& \;\;\; 
	\leq 
	\left(
	\sum_{n=1}^{\infty } 
	\alpha_{n+1}^{2} \gamma_{n+1}^{2r} 
	\right)
	\sup_{n\geq 0}
	E_{\theta,\xi }
	\left(
	\varphi_{Q,s }^{2}(\xi_{n} ) 
	I_{ \{\tau_{Q} \geq n \} }
	\right)
	<
	\infty  
\end{align*}
for each $\theta \in \mathbb{R}^{d_{\theta } }$, 
$\xi \in \mathbb{R}^{d_{\xi } }$.  
Since 
\begin{align*}
	E_{\theta,\xi }
	\left(
	w_{1,n} 
	I_{ \{\tau_{Q} > n \} }
	|
	{\cal F}_{n} 
	\right)
	=
	\left(
	E_{\theta,\xi }
	\left(
	\tilde{F}(\theta_{n}, \xi_{n+1} )
	|
	{\cal F}_{n} 
	\right)
	-
	(\Pi\tilde{F} )(\theta_{n},\xi_{n} )
	\right) 
	I_{ \{\tau_{Q} > n \} }
	=
	0
\end{align*}
w.p.1 for every $\theta\in \mathbb{R}^{d_{\theta } }$, 
$\xi \in \mathbb{R}^{d_{\xi } }$, $n\geq 1$,  
it can be deduced easily that 
series 
\begin{align*}
	\sum_{n=1}^{\infty } \alpha_{n} \gamma_{n}^{r} w_{1,n}, 
	\;\;\; 
	\sum_{n=1}^{\infty } \alpha_{n} \gamma_{n}^{r} w_{2,n}, 
	\;\;\; 
	\sum_{n=1}^{\infty } (\alpha_{n} \gamma_{n}^{r} - \alpha_{n+1} \gamma_{n+1}^{r} ) w_{3,n}
\end{align*}
converge w.p.1 
on $\bigcap_{n=0}^{\infty } \{\theta_{n} \in Q \}$, 
as well as that 
$
	\lim_{n\rightarrow \infty } \alpha_{n} \gamma_{n}^{r} w_{3,n-1} = 0
$	
w.p.1 on the same event. 
Owing to this and (\ref{2.3*}), 
we have that 
$\sum_{n=0}^{\infty } \alpha_{n} \gamma_{n}^{r} w_{n}$
converges w.p.1 on $\bigcap_{n=0}^{\infty } \{\theta_{n} \in Q \}$. 
}
\end{vproof}

\section{Proof of Theorems \ref{theorem4.1} and \ref{theorem4.2}} \label{section4*} 

In this section, we use the following notation. 
For 
$\theta \in \mathbb{R}^{d_{\theta } }$, $x \in \mathbb{R}^{d_{x} }$, $y \in \mathbb{R}$, 
$\xi = (x,y)$, 
let 
\begin{align*}
	F(\theta, \xi ) = (y - G_{\theta }(x) ) H_{\theta }(x), 
\end{align*}
while 
$\xi_{n} = (x_{n}, y_{n} )$ for $n\geq 0$. 
With this notation, it is obvious that algorithm (\ref{4.1})
admits the form (\ref{2.1}). 

\begin{vproof}{Theorem \ref{theorem4.1}}
{Let $\theta = [a'_{1} \cdots a'_{N_{1} } \; a''_{1,1} \cdots a''_{N_{1},N_{2} } ]^{T}
\in \mathbb{R}^{d_{\theta } }$, while 
\begin{align*}
	\delta_{\theta } 
	=
	\frac{\varepsilon}{2KL N_{1} N_{2} (1 + \|\theta \| ) }
\end{align*}
and
$\hat{U}_{\theta } = 
\{\eta \in \mathbb{C}^{d_{\theta } }: \|\eta - \theta \| < \delta_{\theta } \}$
($\varepsilon$ is specified in Assumption \ref{a4.1}).
Moreover, for 
$\eta = [b'_{1} \cdots b'_{N_{1} } \; b''_{1,1} \cdots b''_{N_{1},N_{2} } ]^{T} 
\in \mathbb{C}^{d_{\theta } }$, $x\in \mathbb{R}^{d_{x} }$, let 
\begin{align*}
	&
	\hat{G}_{\eta}(x)
	=
	\hat{\phi}_{1}\left(
	\sum_{i_{1}=1}^{N_{1} }
	b'_{i_{1} }
	\hat{\phi}_{2}\left(
	\sum_{i_{2}=1}^{N_{2} } 
	b''_{i_{1},i_{2} } \phi_{i_{2} }(x)
	\right)
	\right),
	\\
	&
	\hat{f}(\eta )
	=
	\frac{1}{2}
	\int (y - \hat{G}_{\eta}(x) )^{2} \pi(dx,dy).  
\end{align*}
Then, we have 
\begin{align*}
	\left|
	\sum_{i_{2}=1}^{N_{2} } b''_{i_{1},i_{2} } \psi_{i_{2} }(x) 
	- 
	\sum_{i_{2}=1}^{N_{2} } a''_{i_{1},i_{2} } \psi_{i_{2} }(x) 
	\right|
	\leq 
	\sum_{i_{2}=1}^{N_{2} } 
	|b''_{i_{1},i_{2} } - a''_{i_{1},i_{2} } |
	\: |\psi_{i_{2} }(x) |
	\leq 
	\delta_{\theta } L N_{2} 
	< 
	\varepsilon
\end{align*}
for all 
$\eta = [b'_{1} \cdots b'_{N_{1} } \; b''_{1,1} \cdots b''_{N_{1},N_{2} } ]^{T}
\in \hat{U}_{\theta }$, 
$1\leq i_{1} \leq N_{1}$
and each 
$x\in \mathbb{R}^{d_{x} }$
satisfying 
$\max_{1\leq k \leq N_{2} } |\psi_{k}(x) | \leq L$. 
Consequently, Assumption \ref{a4.1} implies
\begin{align*}
	&
	\left|
	\sum_{i_{1}=1}^{N_{1} }
	b'_{i_{1} } 
	\hat{\phi}_{2}\left(
	\sum_{i_{2}=1}^{N_{2} } b''_{i_{1},i_{2} } \psi_{i_{2} }(x) 
	\right)
	-
	\sum_{i_{1}=1}^{N_{1} }
	a'_{i_{1} } 
	\phi_{2}\left(
	\sum_{i_{2}=1}^{N_{2} } a''_{i_{1},i_{2} } \psi_{i_{2} }(x) 
	\right)
	\right|
	\\
	&
	\begin{aligned}[b]
	\leq &
	\sum_{i_{1}=1}^{N_{1} }
	|b'_{i_{1} } - a'_{i_{1} } |
	\left|
	\hat{\phi}_{2}\left(
	\sum_{i_{2}=1}^{N_{2} } b''_{i_{1},i_{2} } \psi_{i_{2} }(x) 
	\right)
	\right| 
	\\
	&
	+
	\sum_{i_{1}=1}^{N_{1} }
	|a'_{i_{1} } |
	\left|
	\hat{\phi}_{2}\left(
	\sum_{i_{2}=1}^{N_{2} } b''_{i_{1},i_{2} } \psi_{i_{2} }(x) 
	\right)
	-
	\hat{\phi}_{2}\left(
	\sum_{i_{2}=1}^{N_{2} } a''_{i_{1},i_{2} } \psi_{i_{2} }(x) 
	\right)
	\right|
	\end{aligned}
	\\
	&
	\leq 
	\delta_{\theta } K N_{1} 
	+
	K 
	\sum_{i_{1}=1}^{N_{1} }
	|a'_{i_{1} } |
	\left|
	\sum_{i_{2}=1}^{N_{2} } b''_{i_{1},i_{2} } \psi_{i_{2} }(x) 
	-
	\sum_{i_{2}=1}^{N_{2} } a''_{i_{1},i_{2} } \psi_{i_{2} }(x) 
	\right|
	\\
	&
	\leq 
	\delta_{\theta } K N_{1} 
	+
	\delta_{\theta } K L N_{1} N_{2} \|\theta \|
	< 
	\varepsilon
\end{align*}
for any 
$\eta = [b'_{1} \cdots b'_{N_{1} } \; b''_{1,1} \cdots b''_{N_{1},N_{2} } ]^{T}
\in \hat{U}_{\theta }$ 
and each
$x\in \mathbb{R}^{d_{x} }$
satisfying 
$\max_{1\leq k \leq N_{2} } |\psi_{k}(x) | \leq L$. 
Then, it can be deduced that 
for all $x\in \mathbb{R}^{d_{x} }$ satisfying 
$\max_{1\leq k \leq N_{2} } |\psi_{k}(x) | \leq L$, 
$\hat{G}_{\eta}(x)$ is analytical in $\eta$ on $\hat{U}_{\theta }$. 
Moreover, Assumption \ref{a4.1} yields
\begin{align*}
	&
	|\hat{G}_{\eta}(x) |
	\leq 
	K
	\left(
	1 
	+
	\sum_{i_{1}=1}^{N_{1} } 
	|b'_{i_{1} } |
	\left|
	\hat{\phi}_{2}\left(
	\sum_{i_{2}=1}^{N_{2} } 
	b''_{i_{1},i_{2}} \psi_{i_{2} }(x) 
	\right)
	\right|
	\right)
	\leq 
	K^{2} (1 + \|\eta \| ), 
	\\
	&
	\left|
	\frac{\partial }{\partial b'_{k_{1} } } 
	\hat{G}_{\eta}(x)
	\right|
	=
	\left|
	\hat{\phi}'_{1}\left(
	\sum_{i_{1}=1}^{N_{1} } 
	b'_{i_{1} } 
	\hat{\phi}_{2}\left(
	\sum_{i_{2}=1}^{N_{2} } b''_{i_{1},i_{2} } 
	\psi_{i_{2} }(x) 
	\right)
	\right)
	\hat{\phi}_{2}\left(
	\sum_{i_{2}=1}^{N_{2} } b''_{k_{1},i_{2} } 
	\psi_{i_{2} }(x) 
	\right)
	\right|
	\leq 
	K^{2}, 
	\\
	&
	\begin{aligned}[b]
	\left|
	\frac{\partial }{\partial b''_{k_{1},k_{2} } } 
	\hat{G}_{\eta}(x)
	\right|
	= 
	&
	\left|
	\hat{\phi}'_{1}\left(
	\sum_{i_{1}=1}^{N_{1} } 
	b'_{i_{1} } 
	\hat{\phi}_{2}\left(
	\sum_{i_{2}=1}^{N_{2} } b''_{i_{1},i_{2} } 
	\psi_{i_{2} }(x) 
	\right)
	\right)
	\right.
	\\
	&
	\left.
	\cdot
	\hat{\phi}'_{2}\left(
	\sum_{i_{2}=1}^{N_{2} } b''_{k_{1},i_{2} } 
	\psi_{i_{2} }(x) 
	\right)
	b'_{k_{1} } b''_{k_{1}, k_{2} } \psi_{k_{2} }(x)
	\right|
	\\
	\leq &
	K^{2} L \|\eta \|^{2}
	\end{aligned}
\end{align*} 
for all 
$\eta = [b'_{1} \cdots b'_{N_{1} } \; b''_{1,1} \cdots b''_{N_{1},N_{2} } ]^{T}
\in \hat{U}_{\theta }$, 
$1\leq k_{1} \leq N_{1}$, $1 \leq k_{2} \leq N_{2}$ 
and each
$x\in \mathbb{R}^{d_{x} }$
satisfying 
$\max_{1\leq k \leq N_{2} } |\psi_{k}(x) | \leq L$. 
Therefore, 
\begin{align*}
	\|\nabla_{\eta} \hat{G}_{\eta}(x) \|
	\leq 
	K^{2} L N_{1} N_{2} (1 + \|\eta \| )^{2}
\end{align*}
for any $\eta \in \hat{U}_{\theta }$
and each $x\in \mathbb{R}^{d_{x} }$
satisfying $\max_{1\leq k \leq N_{2} } \|\psi_{k}(x) | \leq L$. 
Thus, 
\begin{align*}
	\|\nabla_{\eta } (y - \hat{G}_{\eta}(x) )^{2} \|
	=
	2 
	|y - \hat{G}_{\eta}(x) | 
	\|\nabla_{\eta} \hat{G}_{\eta}(x) \|
	\leq 
	4 K^{4} L^{2} N_{1} N_{2} 
	(1 + \|\eta \| )^{3}
\end{align*}
for all $\eta \in \hat{U}_{\theta }$
and each $x\in \mathbb{R}^{d_{x} }$, $y\in \mathbb{R}$
satisfying $\max_{1\leq k \leq N_{2} } \|\psi_{k}(x) | \leq L$, 
$|y| \leq L$. 
Then, the dominated convergence theorem and Assumption \ref{a4.2} 
imply that 
$\hat{f}(\cdot )$ is differentiable on $\hat{U}_{\theta }$.
Consequently, $\hat{f}(\cdot )$ is analytical on $\hat{U}_{\theta }$. 
Since $f(\theta ) = \hat{f}(\theta )$ for all $\theta \in \mathbb{R}^{d_{\theta } }$, 
we conclude that $f(\cdot )$ is real-analytic on entire $\mathbb{R}^{d_{\theta } }$. 
}
\end{vproof}

\begin{vproof}{Theorem \ref{theorem4.2}}
{As $\{\xi_{n} \}_{n\geq 0}$
can be interpreted as a Markov chain whose 
transition kernel does not depend 
on $\{\theta_{n} \}_{n\geq 0}$, 
it is straightforward to show that Assumptions \ref{a2.2} 
and \ref{a2.3} hold. 
The theorem's assertion then follows directly from 
Theorem \ref{theorem2.1}. 
}
\end{vproof}

\section{Proof of Theorems \ref{theorem5.1} and \ref{theorem5.2}} \label{section5*} 

In this section, we rely on the following notation. 
For $n\geq 0$, let $\xi_{n+1} = (x_{n}, x_{n+1}, y_{n} )$, 
while 
\begin{align*}
	F(\theta, \xi ) 
	= 
	-(c(i) + \beta G_{\theta }(j) - G_{\theta }(i) ) y
\end{align*}
for $\theta, y \in \mathbb{R}^{d_{\theta } }$, $i,j \in {\cal X}$ 
and $\xi = (i,j,y)$. 
Moreover, let 
\begin{align*}
	\Pi_{\theta }((i,j,y), (i',j')\times B )
	= &
	P(\xi_{1} \in (i',j')\times B 
	| \xi_{0} = (i,j,y) )
	\\
	= &
	I_{B}(\beta y + H_{\theta }(j) ) P(x_{1}=j'|x_{0}=j) I_{j}(i')
\end{align*}
for $\theta, y \in \mathbb{R}^{d_{\theta } }$, 
$B \in {\cal B}^{d_{\theta } }$, 
$i,i',j,j' \in {\cal X}$. 
Then, it is straightforward to verify that recursion  
(\ref{5.1}), (\ref{5.3}) admits the form of the algorithm 
studied in Section \ref{section2}. 

The following notation is also used in this section. 
$e$ is an $N$-dimensional column vector whose all components are one. 
For $1\leq i \leq N$,  
$e_{i} = [e_{i,1} \cdots e_{i,N} ]^{T}$ 
is an $N$-dimensional column vector such that 
$e_{i,i}=1$ and $e_{i,k}=0$ for $k\neq i$. 
$P$ and $\pi$
denote (respectively) the transition probability matrix 
and the invariant column probability vector of 
$\{x_{n} \}_{n\geq 0}$ 
(notice that 
$j,i$ entry of $P$ is 
$P(x_{1}=j|x_{0}=i)$).  
Furthermore
$c = [c(1) \cdots c(N) ]$
and 
$g = c \sum_{n=0}^{\infty } \beta^{n} P^{n}$, 
while 
$G_{\theta } = [G_{\theta }(1) \cdots G_{\theta }(N) ]$, 
$\tilde{G}_{\theta } = c + \beta G_{\theta } P - G_{\theta }$
and 
$H_{\theta } = [H_{\theta }(1) \cdots H_{\theta }(N) ]$
for $\theta \in \mathbb{R}^{d_{\theta } }$
(notice that $c$, $g$, $G_{\theta }$, $\tilde{G}_{\theta }$
are row vectors). 

\begin{lemma} \label{lemma5.1}
Let Assumption \ref{a5.1} and \ref{a5.2} hold. 
Then, there exists a real number $\varepsilon \in (0,1)$
and for any compact set $Q \subset \mathbb{R}^{d_{\theta } }$, 
there exists another real number $C_{Q} \in [1, \infty )$
such that 
\begin{align}
	&
	\|
	(\Pi^{n} F)(\theta, \xi )
	-
	\nabla f(\theta ) 
	\|
	\leq 
	C_{Q} n \varepsilon^{n} 
	(1 + \|y \| ), 
	\nonumber \\
	&
	\|
	\left(
	(\Pi^{n} F)(\theta', \xi )
	-
	\nabla f(\theta' ) 
	\right)
	-
	\left(
	(\Pi^{n} F)(\theta'', \xi )
	-
	\nabla f(\theta'' ) 
	\right)
	\|
	\leq 
	C_{Q} n \varepsilon^{n} \|\theta' - \theta'' \|
	(1 + \|y \| ), 
	\nonumber \\
	& \label{l5.1.1*} 
	E\left(
	\|y_{n} \|^{2} I_{ \{\tau_{\rho } \geq n \} }
	|\theta_{0}=\theta, \xi_{0}=\xi 
	\right)
	\leq 
	C_{Q} (1 + \|y \| )^{2} 
\end{align}
for all $\theta, \theta', \theta'' \in Q$, 
$y \in \mathbb{R}^{d_{\theta } }$, $i,j \in {\cal X}$ and 
$\xi = (i,j,y)$. 
\end{lemma}

\begin{proof}
Let $Q\subset \mathbb{R}^{d_{\theta } }$ be an arbitrary compact set, 
while $\varepsilon \in (0, 1 )$, $\tilde{C} \in [1,\infty )$
are real numbers such that 
$\varepsilon \geq \max\{1/2,\beta \}$, 
$\|P^{n} \| \leq \tilde{C}$ and 
\begin{align*}
	\|P^{n} - \pi e^{T} \|
	\leq 
	\tilde{C} \varepsilon^{n}
\end{align*}
for $n\geq 0$
(the existence of $\varepsilon, \tilde{C}$ is ensured by Assumption \ref{a5.1}). 
Moreover, 
$\tilde{C}_{1,Q} \in [1,\infty )$ denotes 
an upper bound of 
$\|G_{\theta } \|, \|\tilde{G}_{\theta } \|, \|H_{\theta } \|$ 
on 
$Q$, 
while $\tilde{C}_{2,Q} \in [1,\infty )$ is 
a Lipschitz constant of 
$G_{\theta }, \tilde{G}_{\theta }, H_{\theta }$ 
on the same set. 
Furthermore, 
$C_{Q} = 
6 \tilde{C}^{2} (\tilde{C}_{1,Q} + \tilde{C}_{2,Q} )^{2}
/(1- \varepsilon )^{2}$. 

It is straightforward to show 
$\nabla f(\theta ) = 
H_{\theta } \diag(G_{\theta } - g ) \pi$
and 
\begin{align*}
	(\Pi^{n} F)(\theta, \xi )
	= &	
	\begin{aligned}[t]
	-
	E\Bigg(
	&
	(c(x_{n} ) + \beta G_{\theta }(x_{n+1} ) - G_{\theta }(x_{n} ) )
	\\
	&
	\cdot
	\left. 
	\left(
	\beta^{n} y 
	+
	\sum_{k=0}^{n-1} \beta^{k} H_{\theta }(x_{n-k} ) 
	\right)
	\right| x_{1} = j 
	\Bigg)
	\end{aligned}
	\\
	= &
	-\beta^{n} y \tilde{G}_{\theta } P^{n-1} e_{j} 
	-
	\sum_{k=0}^{n-1} 
	\beta^{k} H_{\theta } 
	\diag(\tilde{G}_{\theta } P^{k} ) P^{n-k-1} e_{j} 
	\\
	= &
	\nabla f(\theta ) 
	-\beta^{n} y \tilde{G}_{\theta } P^{n-1} e_{j} 
	+
	H_{\theta } 
	\diag\left(
	\tilde{G}_{\theta } \sum_{k=n}^{\infty } \beta^{k} P^{k} 
	\right)
	\pi
	\\
	&
	-
	\sum_{k=0}^{n-1} 
	\beta^{k} H_{\theta } 
	\diag(\tilde{G}_{\theta } P^{k} )
	(P^{n-k-1} - \pi e^{T} ) e_{j} 
\end{align*}
for $\theta, y \in \mathbb{R}^{d_{\theta } }$, $i,j \in {\cal X}$
and 
$\xi = (i,j,y)$. 
Therefore, 
\begin{align*}
	&
	\|
	(\Pi^{n} F)(\theta, \xi )
	-
	\nabla f(\theta ) 
	\|
	\\
	&
	\;\;\; 
	\leq 
	\tilde{C} \tilde{C}_{1,Q} 
	\beta^{n} \|y\|
	+
	\tilde{C} \tilde{C}_{1,Q}^{2} 
	\sum_{k=n}^{\infty } \beta^{k} 
	+ 
	\tilde{C}^{2} \tilde{C}_{1,Q}^{2} 
	\sum_{k=0}^{n-1} \beta^{k} \varepsilon^{n-k-1} 
	\\
	&
	\;\;\; 
	\leq 
	C_{Q} n \varepsilon^{n} (1 + \|y \| )
\end{align*}
for all 
$\theta \in Q$, $y\in \mathbb{R}^{d_{\theta } }$, 
$i,j \in {\cal X}$, $n\geq 0$ and $\xi = (i,j,y)$. 
Moreover, 
\begin{align*}
	&
	\|
	\left(
	(\Pi^{n} F)(\theta', \xi )
	-
	\nabla f(\theta' ) 
	\right)
	-
	\left(
	(\Pi^{n} F)(\theta'', \xi )
	-
	\nabla f(\theta'' ) 
	\right)
	\|
	\\
	&
	\;\;\; 
	\begin{aligned}[b]
		\leq &
		\tilde{C} 
		\beta^{n} \|y\|
		\|\tilde{G}_{\theta' } - \tilde{G}_{\theta'' } \|
		+
		\tilde{C} \tilde{C}_{1,Q} 
		(\|\tilde{G}_{\theta' } - \tilde{G}_{\theta'' } \|
		\\
		&
		+
		\|H_{\theta' } - H_{\theta'' } \| )
		\sum_{k=n}^{\infty } \beta^{k} 
		+ 
		\tilde{C}^{2} \tilde{C}_{1,Q}
		(\|\tilde{G}_{\theta' } - \tilde{G}_{\theta'' } \|
		+
		\|H_{\theta' } - H_{\theta'' } \| )
		\sum_{k=0}^{n-1} \beta^{k} \varepsilon^{n-k-1} 
		\\
		\leq &
		C_{Q} n \varepsilon^{n} \|\theta' - \theta'' 
		\| (1 + \|y \| )
	\end{aligned}
\end{align*}
for any 
$\theta', \theta'' \in Q$, $y\in \mathbb{R}^{d_{\theta } }$, 
$i,j \in {\cal X}$, $n\geq 0$ and $\xi = (i,j,y)$. 
On the other side, we have
\begin{align*}
	\|y_{n+1} \| I_{\{\tau_{Q} \geq n+1 \} }
	\leq 
	\beta \|y_{n} \| I_{\{\tau_{Q} \geq n \} }
	+
	\tilde{C}_{1,Q}
\end{align*}
for $n\geq 0$. 
Consequently, 
\begin{align*}
	\beta \|y_{n} \| I_{\{\tau_{Q} \geq n \} }
	\leq 
	\|y_{0} \| 
	+
	\tilde{C}_{1,Q} \sum_{k=0}^{n-1} \beta^{k}
	\leq 
	C_{Q}^{1/2} (1 + \|y_{0} \| )
\end{align*}
for $n\geq 0$, 
wherefrom (\ref{l5.1.1*}) immediately follows. 
\end{proof}

\begin{vproof}{Theorem \ref{theorem5.1}}
{Since 
\begin{align*}
	f(\theta )
	=
	\frac{1}{2}
	\sum_{i=1}^{N_{x} } 
	(g(i) - G_{\theta }(i) )^{2} 
	\pi(i)
\end{align*}
for each $\theta \in \mathbb{R}^{d_{\theta } }$
($\pi(i)$ is the $i$-th component of $\pi$), 
Assumption \ref{a5.2} implies that $f(\cdot )$ is analytic 
on entire $\mathbb{R}^{d_{\theta } }$. 
}
\end{vproof}

\begin{vproof}{Theorem \ref{theorem5.2}}
{Using Lemma \ref{lemma5.1}, it can be concluded easily that 
Assumption \ref{a2.2} and \ref{a2.3} hold. 
Then, the theorem's assertion directly follows from Theorem 
\ref{theorem2.1}. 
}
\end{vproof}

\section{Proof of Theorems \ref{theorem6.1} and \ref{theorem6.2}} \label{section6*}

In this section, we use the following notation. 
For $n\geq 0$, let 
\begin{align*}
	z_{n} = [x_{n}^{T} \; y_{n} \cdots y_{n-M + 1} ]^{T}, 
	\;\;\;\;\;
	\xi_{n} = [z_{n}^{T} \; \varepsilon_{n} \; \psi_{n}^{T} \cdots \varepsilon_{n-N +1} \; \psi_{n-N +1}^{T} ]^{T}, 
\end{align*}
while $d_{\xi} = L + M + N (d_{\theta } + 1 )$. 
For $\theta \in \Theta$, 
let $\varepsilon_{0}^{\theta } = \cdots = \varepsilon_{-N+1} = 0$, 
$\psi_{0}^{\theta } = \cdots = \psi_{-N+1} = 0$, 
while 
$\{\varepsilon_{n}^{\theta } \}_{n\geq 0}$, 
$\{\psi_{n}^{\theta } \}_{n\geq 0}$ are defined by the following recursion: 
\begin{align*}
	&
	\phi_{n-1}^{\theta } 
	= 
	[y_{n-1} \cdots y_{n-M} \; \varepsilon_{n-1}^{\theta} 
	\cdots \varepsilon_{n-N}^{\theta } ]^{T}, 
	\\
	&
	\varepsilon_{n}^{\theta}
	=
	y_{n} - (\phi_{n-1}^{\theta } )^{T} \theta, 
	\\
	&
	\psi_{n}^{\theta }
	=
	\phi_{n-1}^{\theta }
	-
	[\psi_{n-1}^{\theta } \cdots \psi_{n-N}^{\theta} ] A_{0} \theta, 
	\\
	&
	\xi_{n}^{\theta } 
	=
	[z_{n}^{T} \; \varepsilon_{n}^{\theta } \; ( \psi_{n}^{\theta } )^{T} \cdots 
	\varepsilon_{n-N+1}^{\theta } \; ( \psi_{n-N+1}^{\theta } )^{T} ]^{T},
	\;\;\; n\geq 1. 
\end{align*}
Then, it is straightforward to verify that 
$\{\varepsilon_{n}^{\theta } \}_{n\geq 0}$ satisfies the recursion 
(\ref{6.1''}), as well as that 
$\psi_{n}^{\theta } = \nabla_{\theta } \varepsilon_{n}^{\theta }$
for $n\geq 0$. 
Moreover, it can be deduced easily 
that there exist
a matrix valued function $G_{\theta } : \Theta \rightarrow \mathbb{R}^{d_{\xi } \times d_{\xi } }$
and a matrix $H \in \mathbb{R}^{d_{\xi} \times L}$ with the following properties: 
\begin{romannum}
\item
$G_{\theta }$ is linear in $\theta$ and
its eigenvalues lie outside 
$\{z \in \mathbb{C}: |z| \leq 1 \}$ for each $\theta \in \Theta$. 
\item
Equations 
\begin{align*}
	&
	\xi_{n+1}^{\theta } = G_{\theta } \xi_{n}^{\theta } + H w_{n}, 
	\;\;\;\;\;
	\xi_{n+1} = G_{\theta_{n} } \xi_{n} + H w_{n}
\end{align*}
hold for all $\theta \in \Theta$, $n\geq 0$. 
\end{romannum}

The following notation is also used in this section. 
For $\theta \in \Theta$, 
$z \in \mathbb{R}^{L+M}$, 
$u_{1},\dots,u_{N} \in \mathbb{R}$, 
$v_{1},\dots,v_{N} \in \mathbb{R}^{d_{\theta } }$
and 
$\xi = [z^{T} \; u_{1} \; v_{1}^{T} \cdots u_{N} \; v_{N}^{T} ]^{T}$, 
let 
\begin{align*}
	F(\theta, \xi ) = v_{1} u_{1}, 
	\;\;\;\;\;
	\phi(\xi ) = u_{1}^{2}, 
\end{align*}	
while 
\begin{align*}
	\Pi_{\theta }(\xi,B)
	=
	E(I_{B}(G_{\theta } \xi + H w_{0} ))
\end{align*}
for a Borel-measurable set 
$B$ from $\mathbb{R}^{d_{\xi}}$. 
Then, it can be deduced easily that 
recursion (\ref{6.1}) -- (\ref{6.7}) admits the form of 
the algorithm considered in Section \ref{section2}. 
Furthermore, it can be shown that 
\begin{align}
	& \label{6.1*}
	(\Pi^{n} \phi)(\theta, 0 ) 
	= 
	E\big((\varepsilon_{n}^{\theta } )^{2} \big), 
	\\
	& \label{6.3*}
	(\Pi^{n} F)(\theta, 0 ) 
	= 
	E\big(\psi_{n}^{\theta } \varepsilon_{n}^{\theta } \big)
	=
	\nabla_{\theta }(\Pi^{n} \phi)(\theta, 0 ) 
\end{align}
for each $\theta \in \Theta$, $n\geq 0$. 

\begin{vproof}{Theorem \ref{theorem6.1}}
{Let 
$m=E(y_{0} )$ and 
$r_{k} = r_{-k} = {\rm Cov}(y_{0},y_{k} )$ for $k\geq 0$, 
while 
\begin{align*}
	\varphi(\omega )
	=
	\sum_{k=-\infty }^{\infty } r_{k} e^{-i\omega k }
\end{align*}
for $\omega \in [-\pi,\pi]$.
Moreover, 
for $\theta \in \Theta$, $z \in \mathbb{C}$, 
let $C_{\theta }(z) = A_{\theta }(z)/B_{\theta }(z)$, 
while 
\begin{align*}
	\alpha_{\theta}
	=
	1
	+
	\max_{\omega \in [-\pi,\pi] } |A_{\theta }(e^{i\omega } ) |, 
	\;\;\;\;\; 
	\beta_{\theta}
	=
	\min_{\omega \in [-\pi,\pi] } |B_{\theta }(e^{i\omega } ) |, 
	\;\;\;\;\; 
	\delta_{\theta }
	=
	\frac{\beta_{\theta } }{4d_{\theta } \alpha_{\theta } }. 
\end{align*}
Obviously, $1\leq \alpha_{\theta } < \infty$, 
$0<\beta_{\theta }, \delta_{\theta } < \infty$
(notice that the zeros of $B_{\theta }(\cdot )$ are outside 
$\{z\in \mathbb{C}: |z|\leq 1 \}$). 

As $\sum_{k=0}^{\infty } r_{k} < \infty$, 
$|\varphi(\cdot )|$ is uniformly bounded. 
Consequently, 
the spectral theory for stationary processes
(see e.g. \cite[Chapter 2]{caines})
yields 
\begin{align*}
	&
	\lim_{n\rightarrow \infty } E(\varepsilon_{n}^{\theta } ) 
	=
	C_{\theta }(1) m, 
	\\
	&
	\lim_{n\rightarrow \infty } 
	{\rm Cov}(\varepsilon_{n}^{\theta }, \varepsilon_{n+k}^{\theta } )
	=
	\frac{1}{2\pi} 
	\int_{-\pi}^{\pi} 
	|C_{\theta }(e^{i\omega } ) |^{2} \varphi(\omega ) e^{i\omega k } d\omega
\end{align*}
for all $\theta \in \Theta$, $k\geq 0$
(notice that 
$\varepsilon_{n}^{\theta } = C_{\theta }(q) y_{n}$
and the poles of $C_{\theta }(\cdot )$ are in 
$\{z\in \mathbb{C}: |z| > 1 \}$).  
Therefore, 
\begin{align} \label{t6.1.1}
	f(\theta ) 
	=
	\frac{1}{4\pi} 
	\int_{-\pi}^{\pi} 
	|C_{\theta }(e^{i\omega } ) |^{2} \varphi(\omega ) d\omega
	+
	|C_{\theta }(1) |^{2} \frac{m^{2} }{2} 
\end{align}
for any $\theta \in \Theta$. 
On the other side, it is straightforward to verify 
\begin{align*}
	&
	\frac{\partial}{\partial a_{k} } A_{\theta }(e^{i\omega } )
	=
	- e^{-i\omega k }, 
	\\
	&
	\frac{\partial^{2} }{\partial a_{k_{1} } \partial a_{k_{2} } } A_{\theta }(e^{i\omega } )
	=
	0, 
	\\
	&
	\begin{aligned}[b]
	\frac{\partial^{l_{1}+\cdots+l_{N} } }
	{\partial b_{1}^{l_{1} } \cdots \partial b_{N}^{l_{N} } }
	\left(
	\frac{1}{B_{\theta }(e^{i\omega } ) }
	\right)
	= &
	-
	(l_{1} + l_{2} + \cdots + l_{N} )! \:
	e^{-i\omega (l_{1} + 2l_{2} + \cdots + Nl_{N} ) }
	\\
	&
	\cdot
	\left(
	-\frac{1}{B_{\theta }(e^{i\omega } ) }
	\right)^{l_{1} + l_{2} + \cdots + l_{N} + 1}
	\end{aligned}
\end{align*}
for every $\theta =[a_{1} \cdots a_{M} \; b_{1} \cdots b_{N} ]^{T} \in \Theta$, $\omega \in [-\pi,\pi]$, 
$1\leq k,k_{1},k_{2} \leq M$,
$l_{1},\dots,l_{N} \geq 0$. 
Thus, 
\begin{align*}
	&
	\left|
	\frac{\partial^{k_{1} + \cdots + k_{M} + l_{1} + \cdots l_{N} } }
	{\partial a_{1}^{k_{1} } \cdots \partial a_{M}^{k_{M} } 
	\partial b_{1}^{l_{1} } \cdots \partial b_{N}^{l_{N} } }
	C_{\theta }(e^{i\omega } )
	\right|
	\\
	&
	= 
	\left|
	\frac{\partial^{k_{1} + \cdots + k_{M} } }
	{\partial a_{1}^{k_{1} } \cdots \partial a_{M}^{k_{M} } }
	A_{\theta }(e^{i\omega } )
	\right|
	\;
	\left|
	\frac{\partial^{l_{1} + \cdots l_{N} } }
	{\partial b_{1}^{l_{1} } \cdots \partial b_{N}^{l_{N} } }
	\left(\frac{1}{B_{\theta }(e^{i\omega } ) } \right)
	\right|
	\\
	&
	\leq 
	(l_{1} + \cdots l_{N} )! \:
	\alpha_{\theta } 
	(1/\beta_{\theta } )^{l_{1} + \cdots l_{N} + 1}
\end{align*}
for all $\theta =[a_{1} \cdots a_{M} \; b_{1} \cdots b_{N} ]^{T} \in \Theta$, $\omega \in [-\pi,\pi]$, 
$k_{1},\dots,k_{M} \geq 0$, $l_{1},\dots,l_{N} \geq 0$. 
Then, it can be deduced easily 
\begin{align*}
	|D_{\theta }^{k_{1},\dots,k_{d_{\theta } } } C_{\theta }(e^{i\omega } ) |
	\leq 
	(k_{1}+\cdots+k_{d_{\theta} } )!
	(\alpha_{\theta }/ \beta_{\theta } )^{k_{1}+\cdots+k_{d_{\theta} } + 1 }
\end{align*}
for all 
$\theta \in \Theta$, $\omega \in [-\pi,\pi]$, $k_{1},\dots,k_{d_{\theta} } \geq 0$
($D_{\theta }^{k_{1},\dots,k_{d_{\theta } } }$ denotes 
$\partial^{k_{1}+\cdots+k_{d_{\theta} } }/
\partial \vartheta_{1}^{k_{1} } \cdots \partial \vartheta_{d_{\theta } }^{k_{\theta} }$, 
where $\vartheta_{i}$ is the $i$-th component of $\theta$). 
Since 
\begin{align*}
	D_{\theta}^{k_{1},\dots,k_{d_{\theta } } } |C_{\theta }(e^{i\omega } ) |^{2}
	=	
	\sum_{j_{1}=0}^{k_{1} } \cdots \sum_{j_{d_{\theta } }=0}^{k_{d_{\theta } } } 
	&
	\binom{k_{1} }{j_{1} } \cdots \binom{k_{d_{\theta } } }{j_{d_{\theta } } } 
	D_{\theta }^{j_{1},\dots,j_{d_{\theta } } } C_{\theta }(e^{i\omega } ) 
	\\
	&
	\cdot
	D_{\theta}^{k_{1}-j_{1},\dots,k_{d_{\theta } } - j_{d_{\theta } } } C_{\theta }(e^{-i\omega } )
\end{align*}
for each $\theta \in \Theta$, $\omega \in [-\pi,\pi]$, 
$k_{1},\dots,k_{d_{\theta}} \geq 0$, 
we have 
\begin{align*}
	&
	\big|
	D_{\theta}^{k_{1},\dots,k_{d_{\theta } } } |C_{\theta }(e^{i\omega } ) |^{2}
	\big|
	\\
	&
	\begin{aligned}[b]
	\leq 
	(k_{1} + \cdots + k_{d_{\theta } } )!
	\left(
	\frac{\alpha_{\theta } }{\beta_{\theta } } 
	\right)^{k_{1} + \cdots + k_{d_{\theta } } + 2 } 
	\sum_{j_{1}=0}^{k_{1} } \cdots \sum_{j_{d_{\theta } }=0}^{k_{d_{\theta } } } 
	\frac{\binom{k_{1} }{j_{1} } \cdots \binom{k_{d_{\theta } } }{j_{d_{\theta } } } }
	{\binom{k_{1} + \cdots k_{d_{\theta } } }{j_{1} + \cdots j_{d_{\theta } } } }
	\end{aligned}
	\\
	&
	\leq 
	(k_{1} + \cdots + k_{d_{\theta } } )!
	\left(
	\frac{\alpha_{\theta } }{\beta_{\theta } } 
	\right)^{k_{1} + \cdots + k_{d_{\theta } } + 2 } 
	\sum_{j_{1}=0}^{k_{1} } \cdots \sum_{j_{d_{\theta } }=0}^{k_{d_{\theta } } } 
	\binom{k_{1} }{j_{1} } \cdots \binom{k_{d_{\theta } } }{j_{d_{\theta } } } 
	\\
	&
	\leq 
	(k_{1} + \cdots + k_{d_{\theta } } )!
	\left(
	\frac{2\alpha_{\theta } }{\beta_{\theta } } 
	\right)^{k_{1} + \cdots + k_{d_{\theta } } + 2 } 
\end{align*}
for any $\theta \in \Theta$, $\omega \in [-\pi,\pi]$, 
$k_{1},\dots,k_{d_{\theta}} \geq 0$.   
Consequently, the multinomial formula
(see \cite[Theorem 1.3.1]{krantz&parks}) implies 
\begin{align*}
	&
	\sum_{k_{1}=0}^{\infty } \cdots \sum_{k_{d_{\theta } }=0}^{\infty } 
	\frac{\big|
	D_{\theta}^{k_{1},\dots,k_{d_{\theta } } } |C_{\theta }(e^{i\omega } ) |^{2}
	\big|}
	{k_{1}! \cdots k_{d_{\theta} }! }
	\delta_{\theta }^{k_{1} + \cdots + k_{d_{\theta } } }
	\\
	&
	\leq 
	\left(
	\frac{2\alpha_{\theta } }{\beta_{\theta } } 
	\right)^{2} 
	\sum_{k_{1}=0}^{\infty } \cdots \sum_{k_{d_{\theta } }=0}^{\infty } 
	\frac{(k_{1} + \cdots + k_{d_{\theta } } )! }
	{k_{1}! \cdots k_{d_{\theta} }! }
	\left(
	\frac{2\alpha_{\theta } \delta_{\theta } }{\beta_{\theta } } 
	\right)^{k_{1} + \cdots + k_{d_{\theta } } } 
	\\
	&
	= 
	\left(
	\frac{2\alpha_{\theta } }{\beta_{\theta } } 
	\right)^{2} 
	\sum_{n=0}^{\infty }
	\;\: 
	\sum_{\stackrel{\scriptstyle 0\leq k_{1},\dots,k_{d_{\theta } } \leq n}
	{k_{1} + \cdots k_{d_{\theta } } = n } } 
	\frac{(k_{1} + \cdots + k_{d_{\theta } } )! }
	{k_{1}! \cdots k_{d_{\theta} }! }
	\left(
	\frac{2\alpha_{\theta } \delta_{\theta } }{\beta_{\theta } } 
	\right)^{k_{1} + \cdots + k_{d_{\theta } } } 
	\\
	&
	= 
	\left(
	\frac{2\alpha_{\theta } }{\beta_{\theta } } 
	\right)^{2} 
	\sum_{n=0}^{\infty }
	\left(
	\frac{2 d_{\theta } \alpha_{\theta } \delta_{\theta } }{\beta_{\theta } } 
	\right)^{n} 
	\\
	&
	= 
	\left(
	\frac{2\alpha_{\theta } }{\beta_{\theta } } 
	\right)^{2} 
	\sum_{n=0}^{\infty }
	\left(
	\frac{1}{2} 
	\right)^{n} 
	< \infty
\end{align*}
for every $\theta \in \Theta$, $\omega \in [-\pi,\pi]$. 
Then, the analyticity of $f(\cdot )$
directly follows from (\ref{t6.1.1})
and the fact that $|\varphi(\cdot )|$ is uniformly bounded
(also notice that $C_{\theta}(1)$ is analytic in $\theta$).
}
\end{vproof}

\begin{vproof}{Theorem \ref{theorem6.2}}
{It is straightforward to show 
\begin{align*}
	&
	\max\{\|F(\theta, \xi ) \|, \phi(\xi ) \} 
	\leq 
	\|\xi \|, 
	\\
	&
	\max\{\|F(\theta, \xi' ) - F(\theta, \xi'' ) \|, |\phi(\xi' ) - \phi(\xi'' ) | \}
	\leq 
	2\|\xi' - \xi'' \|(\|\xi' \| + \|\xi'' \| )
\end{align*}
for all $\theta \in \Theta$, $\xi,\xi',\xi'' \in \mathbb{R}^{d_{\xi} }$. 
Moreover, it can be deduced easily 
that for 
any compact set $Q \subset \mathbb{R}^{d_{\theta } }$, 
there exist real numbers 
$\delta_{1,Q} \in (0,1)$, 
$C_{1,Q} \in [1,\infty )$ such that 
$\|G_{\theta }^{n} \| \leq C_{1,Q} \delta_{1,Q}^{n}$
and 
\begin{align*}
	\|G_{\theta'} - G_{\theta''} \|
	\leq 
	C_{1,Q} \|\theta' - \theta'' \|
\end{align*}
for each $\theta, \theta', \theta'' \in Q$, $n\geq 0$. 
Then, the results of \cite[Section II.2.3]{benveniste}  
imply that 
there exist a locally Lipschitz continuous function 
$g:\Theta \rightarrow \mathbb{R}^{d_{\theta } }$ and 
a Borel-measurable function 
$\tilde{F}: \Theta \times \mathbb{R}^{d_{\xi} } \rightarrow \mathbb{R}^{d_{\theta } }$
such that 
\begin{align*}
	F(\theta, \xi )
	-
	g(\theta )
	=
	\tilde{F}(\theta, \xi ) - (\Pi \tilde{F} )(\theta, \xi )
\end{align*}
for every $\theta \in \Theta$, $\xi \in \mathbb{R}^{d_{\xi } }$. 
Due to the same results, 
there exists a locally Lipschitz continuous function 
$h: \Theta \rightarrow \mathbb{R}$
and for any compact set $Q \subset \mathbb{R}^{d_{\theta } }$, there exist 
real numbers $\delta_{2,Q} \in (0,1)$, $C_{2,Q} \in [1,\infty )$
such that 
\begin{align}
	& \label{t6.2.1}
	\max\{
	\|(\Pi^{n} F)(\theta, \xi ) - g(\theta ) \|, 
	|(\Pi^{n} \phi)(\theta, \xi ) - h(\theta ) |
	\}
	\leq 
	C_{2,Q} \delta_{2,Q}^{n} (1 + \|\xi\| )^{2}, 
	\\ \nonumber 
	&
	\max\{
	\|\tilde{F}(\theta, \xi ) \|, \|(\Pi \tilde{F} )(\theta, \xi ) \| 
	\}
	\leq 
	C_{2,Q} (1 + \|\xi \|)^{2}, 
	\\ \nonumber 
	&
	\|\tilde{F}(\theta', \xi ) - \tilde{F}(\theta'', \xi ) \|
	\leq 
	C_{2,Q} \|\theta' - \theta'' \|(1 + \|\xi \| )^{2}
\end{align}
for each 
$\theta, \theta', \theta'' \in Q$, $\xi, \xi', \xi'' \in \mathbb{R}^{d_{\xi } }$. 
Combining (\ref{6.1*}), (\ref{6.3*}), (\ref{t6.2.1})
with the dominated convergence theorem, 
we get 
$h(\cdot ) = f(\cdot )$, 
$g(\cdot ) = \nabla f(\cdot )$. 
On the other side, owing to the fact that 
$\{x_{n} \}_{n\geq 0}$ is a geometrically ergodic Markov chain,
we have that  
$\{y_{n} \}_{n\geq 0}$ admits a stationary regime for $n\rightarrow \infty$. 
Consequently, Theorem \ref{theorem6.1} implies that 
$f(\cdot )$ is analytic on $\Theta$. 
Then, the theorem's assertion directly follows from Theorem \ref{theorem2.1}. 
}
\end{vproof}

\Appendix\section*{}
In this section, we study certain aspects of Assumption \ref{a1.3}. 
More specifically, we show that Assumption \ref{a1.3} is true 
if its `local version', Assumption \ref{a1.3}$^{\prime}$ (below) holds. 
We also demonstrate that (Lojasiewicz coefficients) 
$\delta_{Q,a}$, $\mu_{Q,a}$ and $M_{Q,a}$
have `measurable versions'
for which 
$\hat{\delta}$, $\hat{\mu}$ and $\hat{M}$ 
(defined in Section \ref{section1}) are random variables 
in probability space $(\Omega, {\cal F}, P )$
(i.e., measurable with respect to ${\cal F}$). 
We study these aspects of 
Assumption \ref{a1.3} under the following condition: 

\begin{bassumption}{\ref{a1.3}$^{\prime}$} 
{There exists an open vicinity $U$ of $S$
with the following property: 
For any compact set $Q \subset U$ and any real number 
$a\in f(Q)$, 
there exist real numbers 
$\delta'_{Q,a} \in (0,1)$, $\mu'_{Q,a} \in (1,2]$, 
$M'_{Q,a} \in [1,\infty )$ such that 
\begin{align*}
	|f(\theta ) - a |
	\leq 
	M'_{Q,a} \|\nabla f(\theta ) \|^{\mu'_{Q,a} }
\end{align*}
for all $\theta \in Q$
satisfying $|f(\theta ) - a | \leq \delta'_{Q,a}$.
}
\end{bassumption}

Throughout this section, we rely on the following notation. 
$\varepsilon \in (0,1)$ is a fixed constant. 
For a compact set $Q \subset \mathbb{R}^{d_{\theta } }$, $a\in f(Q)$ 
and $\delta \in (0,1)$, let 
\begin{align*}
	\phi_{Q,a}(\delta )
	=
	\sup
	\left\{
	\frac{1}{2}, 
	\frac{\log\|\nabla f(\theta ) \| }{\log|f(\theta ) - a | }
	:
	\theta \in Q\setminus S, 
	0 < |f(\theta ) - a | \leq \delta 
	\right\}, 
\end{align*}
while 
\begin{align*}
	\delta_{Q,a}
	=
	\sup
	\left\{
	\varepsilon \, \delta
	:
	\delta \in (0,1), 
	\phi_{Q,a}(\delta ) < 1
	\right\}
\end{align*}
and 
$\mu_{Q,a} = 1/\phi_{Q,a}(\delta_{Q,a} )$, 
$M_{Q,a} = 1$. 

\begin{lemma} \label{lemma1a}
Let Assumption \ref{a1.3}$^{\prime}$ hold. 
Moreover, let $Q \subset \mathbb{R}^{d_{\theta } }$ be an arbitrary compact set, 
while $a\in f(Q)$ is an arbitrary real number. 
Then, 
$\delta_{Q,a}$, $\mu_{Q,a}$, $M_{Q,a}$ specified in this section 
satisfy all requirements of Assumption \ref{a1.3}. 
\end{lemma}

\begin{proof}
First, we show $\delta_{Q,a} > 0$. 
To do so, 
we consider separately the following cases: 

{\em Case $Q\cap S = \emptyset$:}
Let 
\begin{align*}
	\tilde{\delta}_{Q,a}
	=
	\inf\left\{
	\exp(-2|\log\|\nabla f(\theta ) \| \, | \, )
	: 
	\theta \in Q
	\right\}. 
\end{align*}
Obviously, 
$0 < \tilde{\delta}_{Q,a} < 1$
(notice that $\inf_{\theta \in Q} \|\nabla f(\theta ) \| > 0$). 
We also have 
\begin{align}\label{l1a.1}
	2|\log\|\nabla f(\theta ) \| \, | 
	\leq 
	\log(1/\tilde{\delta}_{Q,a} )
\end{align}
for all $\theta \in Q$. 
Consequently, 
\begin{align}\label{l1a.3}
	\left|
	\frac{\log\|\nabla f(\theta ) \| }{\log|f(\theta ) - a | }
	\right|
	\leq 
	\frac{|\log\|\nabla f(\theta ) \| \, | }{\log(1/\tilde{\delta}_{Q,a} ) }
	\leq 
	1/2
\end{align}
for any $\theta \in Q$ satisfying $0 < |f(\theta ) - a | \leq \tilde{\delta}_{Q,a}$. 
Thus, 
$\phi_{Q,a}(\delta ) \leq 1/2$ for each $\delta \in (0,\tilde{\delta }_{Q,a} ]$, 
and hence, 
$\delta_{Q,a} \geq \varepsilon \tilde{\delta }_{Q,a} > 0$. 

{\em Case $Q\cap S \neq \emptyset$, $a\not\in f(Q\cap S )$:}
Let 
\begin{align*}
	&
	\tilde{\delta}'_{Q,a}
	=
	\frac{1}{2}
	\inf\left\{
	1, |f(\theta ) - a |
	:
	\theta \in Q\cap S
	\right\}, 
	\\
	&
	\tilde{\delta}''_{Q,a} 
	=
	\inf\left\{
	\exp(-2|\log\|\nabla f(\theta ) \| \, | )
	: 
	\theta \in Q, 
	|f(\theta ) - a | \leq \tilde{\delta}'_{Q,a}
	\right\}, 
\end{align*}
while 
$\tilde{\delta}_{Q,a} = \min\{\tilde{\delta}'_{Q,a}, \tilde{\delta}''_{Q,a} \}$. 
Obviously, 
$0<\tilde{\delta}_{Q,a} \leq 1/2$ 
(notice that $0 < \tilde{\delta}'_{Q,a} \leq 1/2$
and that 
$\theta \not\in Q\cap S$ if $|f(\theta ) - a | \leq \tilde{\delta}'_{Q,a}$; 
also notice that 
$0 < \inf\{\|\nabla f(\theta ) \| : \theta \in Q, |f(\theta ) - a | \leq \tilde{\delta}'_{Q,a} \}$). 
Moreover, 
(\ref{l1a.1}) holds for all $\theta \in Q$ 
satisfying $0 < |f(\theta ) - a | \leq \tilde{\delta}_{Q,a}$. 
Then, (\ref{l1a.3}) is true for any $\theta \in Q$ 
fulfilling $0 < |f(\theta ) - a | \leq \tilde{\delta}_{Q,a}$. 
Hence, 
$\phi_{Q,a}(\delta ) \leq 1/2$
for all $\delta \in (0,\tilde{\delta}_{Q,a} ]$, 
and 
consequently, 
$\delta_{Q,a} \geq \varepsilon \tilde{\delta}_{Q,a} > 0$. 

{\em Case $Q\cap S \neq \emptyset$, $a\in f(Q\cap S )$:}
Let 
$\rho_{Q} = d(Q\cap S, U^{c} ) /2$
and 
$\tilde{Q} = \{\theta \in \mathbb{R}^{d_{\theta } }: d(\theta, Q\cap S ) \leq \rho_{Q} \}$, 
while 
$\tilde{\delta}'_{Q,a} = \delta'_{\tilde{Q},a }$, 
$\tilde{\mu}_{Q,a} = \mu'_{\tilde{Q},a }$, 
$\tilde{M}_{Q,a} = M'_{\tilde{Q},a }$
($\delta'_{\tilde{Q},a }$, 
$\mu'_{\tilde{Q},a }$, 
$M'_{\tilde{Q},a }$
are introduced in Assumption \ref{a1.3}$^{\prime}$). 
Moreover, let 
\begin{align*}
	\tilde{\delta}''_{Q,a}
	=
	\inf\left\{
	\frac{1}{2},
	\exp(-2 |\log\|\nabla f(\theta ) \| \, | )
	: 
	\theta \in Q\setminus \tilde{Q}
	\right\}
\end{align*}
and 
$\tilde{\delta}_{Q,a} = 
\min\left\{\tilde{\delta}'_{Q,a}, \tilde{\delta}''_{Q,a}, \tilde{M}_{Q,a}^{-2/(\tilde{\mu}_{Q,a} - 1 ) }
\right\}$. 
Obviously, 
$\tilde{Q} \subset U$
and 
$0 < \tilde{\delta}_{Q,a} \leq 1/2$. 
Moreover, 
(\ref{l1a.1}) is true for all $\theta \in Q\setminus \tilde{Q}$. 
Therefore, 
(\ref{l1a.3}) holds for 
all $\theta \in Q\setminus \tilde{Q}$
satisfying $0 < |f(\theta ) - a | \leq \tilde{\delta}_{Q,a}$. 
On the other side, Assumption \ref{a1.3}$^{\prime}$ implies 
\begin{align*}
	\log|f(\theta ) - a |
	\leq
	\log \tilde{M}_{Q,a} 
	+
	\tilde{\mu}_{Q,a} 
	\log\|\nabla f(\theta ) \|
\end{align*}
for all $\theta \in \tilde{Q}\setminus S$
satisfying 
$0 < |f(\theta ) - a | \leq \tilde{\delta}_{Q,a}$
(notice that $\tilde{\delta }_{Q,a} \leq \delta'_{\tilde{Q},a}$). 
Consequently, 
\begin{align}\label{l1a.5}
	\frac{\log\|\nabla f(\theta ) \| }{\log|f(\theta ) - a | }
	\leq &
	\frac{1}{\tilde{\mu}_{Q,a} }
	\left(
	1
	-
	\frac{\log\tilde{M}_{Q,a} }{\log|f(\theta ) - a | }
	\right)
	\nonumber\\
	\leq &
	\frac{1}{\tilde{\mu}_{Q,a} }
	\left(
	1
	+
	\frac{\log\tilde{M}_{Q,a} }{\log(1/\tilde{\delta}_{Q,a} ) }
	\right)
	\nonumber\\
	\leq & 
	\frac{\tilde{\mu}_{Q,a} + 1}{2 \tilde{\mu}_{Q,a} }
	< 
	1
\end{align}
for all $\theta \in \tilde{Q}\setminus S$ 
satisfying $0 < |f(\theta ) - a | \leq \tilde{\delta}_{Q,a}$
(notice that 
$\log(1/\tilde{\delta}_{Q,a} ) \geq 2 \log \tilde{M}_{Q,a} / (\tilde{\mu}_{Q,a} - 1 )$). 
Thus, as a result of (\ref{l1a.3}), (\ref{l1a.5}), 
we have $\phi_{Q,a}(\delta ) < 1$
for all $\delta \in (0, \tilde{\delta}_{Q,a} ]$, 
and consequently, 
$\delta_{Q,a} \geq \varepsilon \tilde{\delta}_{Q,a} > 0$. 

Now, we prove that 
$\delta_{Q,a}$, $\mu_{Q,a}$, $M_{Q,a}$
fulfill all other requirements of Assumption \ref{a1.3}. 
By the definition of $\phi_{Q,a}(\cdot )$ and $\delta_{Q,a}$, 
we have 
$0 < \delta_{Q,a} < 1$, 
$1/2 \leq \phi_{Q,a}(\delta_{Q,a} ) < 1$ and 
\begin{align*}
	\frac{\log\|\nabla f(\theta ) \| }{\log|f(\theta ) - a | }
	\leq 
	\phi_{Q,a}(\delta_{Q,a} ) 
\end{align*}
for all $\theta \in Q\setminus S$ satisfying 
$0 < |f(\theta ) - a | \leq \delta_{Q,a}$. 
Therefore, 
$1 < \mu_{Q,a} = 1/\phi_{Q,a}(\delta_{Q,a} ) \leq 2$
and 
\begin{align*}
	\mu_{Q,a} \log\|\nabla f(\theta ) \|
	=
	\frac{\log\|\nabla f(\theta ) \| }{\phi_{Q,a}(\delta_{Q,a} ) }
	\geq 
	\log|f(\theta ) - a |
\end{align*}
for each $\theta \in Q\setminus S$ 
fulfilling $0 < |f(\theta ) - a | \leq \delta_{Q,a}$. 
Hence, (\ref{a1.3.1}) holds for all $\theta \in Q$ 
satisfying $0 < |f(\theta ) - a | \leq \delta_{Q,a}$. 
\end{proof}

\begin{lemma} \label{lemma3a}
Let $\hat{\delta}$, $\hat{\mu}$, $\hat{M}$
be defined using (\ref{1.21}), (\ref{1.23}) 
and $\delta_{Q,a}$, $\mu_{Q,a}$, $M_{Q,a}$ specified in this section. 
Then, $\hat{\delta}$, $\hat{\mu}$, $\hat{M}$ are random variables in 
probability space $(\Omega, {\cal F}, P )$. 
\end{lemma}

\begin{proof}
For $\theta \in \mathbb{R}^{d_{\theta } }$, $\delta \in (0,1)$, let 
\begin{align*}
	\hat{\Phi}(\theta, \delta )
	=
	\frac{\log\|\nabla f(\theta ) \| }{\log|f(\theta ) - \hat{f} | }
	\; I_{S^{c} }(\theta )
	\; I_{(0,\delta ] }\left( |f(\theta ) - \hat{f} | \right)
	\; I_{[0,\rho ] }\left(\liminf_{n\rightarrow \infty } \|\theta - \theta_{n} \| \right)
\end{align*}
($\rho$ is specified in the definition of $\hat{Q}$, Section \ref{section1}), 
while 
\begin{align*}
	\hat{\phi}(\delta)
	=
	\sup\left\{
	1/2, \hat{\Phi}(\theta, \delta ) 
	: 
	\theta \in \mathbb{R}^{d_{\theta } }
	\right\}
	\; I_{\Lambda }
\end{align*}
($\Lambda$ is defined in Section \ref{section1*}). 
Obviously, 
$\hat{\Phi}(\theta, \delta )$ and 
$\hat{\phi}(\delta )$
are measurable random functions of 
$(\theta, \delta )$ and $\delta$
(i.e., $\hat{\Phi}(\theta, \delta )$ and 
$\hat{\phi}(\delta )$ are measurable with respect 
to $\sigma$-algebras
${\cal B}(\mathbb{R}^{d_{\theta } } ) \times {\cal B}( (0,1) ) \times {\cal F}$
and 
${\cal B}( (0,1) ) \times {\cal F}$). 
On the other side, it is straightforward to verify that 
\begin{align*}
	\hat{\delta}
	=
	\sup\{\varepsilon\, \delta : \delta \in (0,1), \hat{\phi}(\delta ) < 1 \}
\end{align*}
and $\hat{\mu} = 1/\hat{\phi}(\hat{\delta} )$
on $\Lambda$. 
Then, it is clear that $\hat{\delta}$, $\hat{\mu}$, $\hat{M}$
are random variables in probability space $(\Omega, {\cal F}, P )$.
\end{proof}

\end{document}